\documentclass[oneside, a4paper]{amsart}

	\usepackage[a4paper, margin=1.4in]{geometry}
	\usepackage{amsmath, amsthm, amssymb, latexsym, mathtools, mathrsfs, eucal}
	\usepackage{dsfont, soul, stackrel, tikz-cd, graphicx, etoolbox, multirow}
	\DeclareGraphicsExtensions{.pdf,.png,.jpg}
	\usepackage{fancyhdr}
	\usepackage[shortlabels]{enumitem}
	\usepackage[pdfmenubar=true, pdfborder={0 0 0 [3 3]}, colorlinks, citecolor=magenta]{hyperref}
	\numberwithin{equation}{section}
	\setenumerate{listparindent=\parindent}

	\newtheoremstyle{Mytheorem}%
	{1em}{1em}%
	{\slshape}{}%
	{\bfseries}{.}%
	{ }{}
 
	\newtheoremstyle{Mydefinition}%
	{1em}{1em}%
	{}{}%
	{\bfseries}{.}%
	{ }{}

	\theoremstyle{Mydefinition}
	\newtheorem{statement}{Statement}[section]
	\newtheorem{definition}[statement]{Definition}
	
	\newtheorem{remark}[statement]{Remark}
	
	\newtheorem{example}[statement]{Example}

	\newtheorem*{comment*}{Comment}

	\theoremstyle{Mytheorem}
	\newtheorem{theorem}[statement]{Theorem}
	\newtheorem{corollary}[statement]{Corollary}
	
	\newtheorem{proposition}[statement]{Proposition}
	\newtheorem{lemma}[statement]{Lemma}

	\newcommand{\f}[2]{\frac{#1}{#2}}

	\newcommand{\G}{GL_2^{+}(\mathbb{Q})}

	\newcommand{\pa}[1]{\frac{\partial}{\partial{#1}}}
	\newcommand{\prt}[2]{\frac{\partial{#1}}{\partial{#2}}}

	\newcommand{\pr}{\operatorname{prim}}

	\newcommand{\nc}{\newcommand}
	\newcommand{\be}{\begin{eqnarray*}}
	\newcommand{\ee}{\end{eqnarray*}}
	\newcommand{\bea}{\begin{eqnarray}}
	\newcommand{\eea}{\end{eqnarray}}
	\newcommand{\bs}{\begin{split}}
	\newcommand{\es}{\end{split}}
	\newcommand{\bal}{\begin{align}}
	\newcommand{\eal}{\end{align}}

	\nc{\bei}{\begin{itemize}}
	\nc{\eei}{\end{itemize}}
	\nc{\bee}{\begin{enumerate}}
	\nc{\eee}{\end{enumerate}}
	\nc{\bet}{\begin{thm}}
	\nc{\eet}{\end{thm}}
	\nc{\bed}{\begin{defn}}
	\nc{\eed}{\end{defn}}
	\nc{\bel}{\begin{lem}}
	\nc{\eel}{\end{lem}}
	\nc{\bep}{\begin{prop}}
	\nc{\eep}{\end{prop}}
	\nc{\bec}{\begin{corollary}}
	\nc{\eec}{\end{corollary}}
	\nc{\ber}{\begin{rem}}
	\nc{\eer}{\end{rem}}
	\nc{\beex}{\begin{example}}
	\nc{\eeex}{\end{example}}
	\nc{\bpm}{\begin{pmatrix}}
	\nc{\epm}{\end{pmatrix}}
	\nc{\bspm}{\left(\begin{smallmatrix}}
	\nc{\espm}{\end{smallmatrix}\right)}

	\newcommand{\cA}{\mathcal{A}}
	
	\newcommand{\cC}{\mathcal{C}}

	\newcommand{\cK}{\mathcal{K}}

	\newcommand{\cO}{\mathcal{O}}
	\newcommand{\cP}{\mathcal{P}}

	\newcommand{\bA}{\mathbb{A}}
	
	\newcommand{\bC}{\mathbb{C}}

	\newcommand{\bH}{\mathbb{H}}

	\newcommand{\bZ}{\mathbb{Z}}

	\newcommand{\BP}{\mathbf{P}}

	\nc{\frf}{\mathfrak{f}}

	\nc{\frs}{\mathfrak{s}}  
	\nc{\frt}{\mathfrak{t}} 
	\nc{\fru}{\mathfrak{u}}
	\nc{\lsl}{\mathfrak{sl}}
	\nc{\lgl}{\mathfrak{gl}}
	\nc{\upsi}{\underline{\psi}}
	\nc{\uchi}{\underline{\chi}}
	 
	\newcommand{\lra}{\longrightarrow}    
	
	\nc{\surjto}{\twoheadrightarrow}
	\nc{\ts}{\times}
	\nc{\ds}{\displaystyle}
	\nc{\nd}{\noindent}  
	\nc{\ud}{\underline}
	\nc{\ov}{\overline}
	\nc{\maplra}[1]{\buildrel #1 \over \lra}
	\nc{\mapto}[1]{\buildrel #1 \over \to}
	\nc{\setb}[1]{\{  #1\}}
	\nc{\cHom}{\mathcal{H}om}

	\def\a{\alpha}
	\def\b{\beta}
	
	\def\d{\delta} 
	 
	\def\f{\phi} 
	\def\g{\gamma} \def\G{\Gamma}
	
	\def\l{\lambda} 
	\def\m{\mu}
	\def\n{\nu}
	\def\r{\rho}

	\def\C{\mathbb{C}}
	
	\def\Z{\mathbb{Z}}

	\def\rd{\partial}
	
	\def\Ker{\hbox{Ker}\;}
	\def\Im{\hbox{Im}\;}

	\def\wt{\hbox{\it wt}}
	\def\ch{\hbox{\it ch}}

	\newcommand{\downtriangleBase}{%
		\begin{tikzpicture}%
			\draw[line cap=round,-] (0,1.6ex) -- (1.7ex,1.6ex);%
			\draw[line cap=round,-] (0,1.6ex) -- (0.85ex,0);%
			\draw[line cap=round,-] (0.85ex,0) -- (1.7ex,1.6ex);%
		\end{tikzpicture}%
	}
	\DeclareMathOperator{\downtriangle}{\downtriangleBase}

\begin{document}

\title[Algorithms for Frobenius manifolds and higher residue pairings]{Smooth projective Calabi-Yau complete intersections and algorithms for their Frobenius manifolds and higher residue pairings}
{}

\author{Younggi Lee}
\email{yglee.math@gmail.com}
\address{Department of Mathematics, POSTECH (Pohang University of Science and Technology), San 31, Hyoja-Dong, Nam-Gu, Pohang, Gyeongbuk, 790-784, South Korea. }

\author{Jeehoon Park}
\email{jeehoonpark@postech.ac.kr}
\address{Department of Mathematics, POSTECH (Pohang University of Science and Technology), San 31, Hyoja-Dong, Nam-Gu, Pohang, Gyeongbuk, 790-784, South Korea. }

\author{Jaehyun Yim}
\email{yimjaehyun@postech.ac.kr}
\address{Department of Mathematics, POSTECH (Pohang University of Science and Technology), San 31, Hyoja-Dong, Nam-Gu, Pohang, Gyeongbuk, 790-784, South Korea. }


\begin{abstract}
	The goal of this article is to provide an explicit algorithmic construction of formal $F$-manifold structures, formal Frobenius manifold structures, and higher residue pairings on the primitive middle-dimensional cohomology $\bH$ of a smooth projective Calabi-Yau complete intersection variety $X$ defined by homogeneous polynomials $G_1(\ud x), \dots, G_k(\ud x)$. Our main method is to analyze a certain dGBV (differential Gerstenhaber-Batalin-Vilkovisky) algebra $\cA$ obtained from the twisted de Rham complex which computes $\bH$. More explicitly, we introduce a notion of \textit{a weak primitive form} associated to a solution of the Maurer-Cartan equation of $\cA$ and the Gauss-Manin connection, which is a weakened version of Saito's primitive form (\cite{Saito}). In addition, we provide an explicit algorithm for a weak primitive form based on the Gr\"obner basis in order to achieve our goal. Our approach through the weak primitive form can be viewed as a unifying link (based on Witten's gauged linear sigma model, \cite{W93}) between the Barannikov-Kontsevich's approach to Frobenius manifolds via dGBV algebras (non-linear topological sigma model, \cite{BK}) and Saito's approach to Frobenius manifolds via primitive forms and higher residue pairings (Landau-Ginzburg model, \cite{ST}).
\end{abstract}

\maketitle
\tableofcontents


\section{Introduction}

Let $M$ be a complex manifold. A Frobenius manifold structure on $M$ is a commutative algebra structure on the holomorphic tangent bundle $\mathcal{T}_M$ with a metric satisfying certain compatibility conditions: see Definition \ref{Def_Frob}. Such a structure was first axiomatized by Dubrovin in \cite{Dubrovin} and its first example was given by K. Saito (the universal unfolding of an isolated hypersurface singularity) in \cite{Saito} and \cite{MSaito}. It also plays an important role in formulating the mirror symmetry conjecture: for example, see \cite{BK} and \cite{CS99}. We refer to \cite{Ma03}, \cite{Maninbook}, and \cite{Her02} for various constructions of Frobenius manifolds.
There is a slight weakening: an $F$-manifold, which appears in \cite{HM} and \cite{Ma05}, is essentially a Frobenius manifold structure without a metric. We refer to Definition \ref{Def_F} and Proposition \ref{Coor_F} for its precise definition. \par

One can also think of formal Frobenius manifold structures and formal $F$-manifold structures. In this paper, we will provide a new construction of formal Frobenius manifolds and formal $F$-manifolds on primitive middle-dimensional cohomology space $\bH$ of a smooth projective Calabi-Yau complete intersection. Our construction is explicit and algorithmic so that it can be implemented in a computer program.\par

On the $B$-side of the mirror symmetry conjecture, there are two well-known constructions of Frobenius manifolds: the work \cite{BK} of Barannikov-Kontsevich and the work \cite{Saito}, \cite{MSaito} of K. Saito and M. Saito. Barannikov-Kontsevich used the algebra $\cP$ of holomorphic polyvector fields on a compact Calabi-Yau manifold $X$ with valued in anti-holomorphic forms in order to construct Frobenius manifold structures on the cohomology group $H^\bullet(X, \bC)$. 
They equipped $\cP$ with structure of the dGBV (differential Gerstenhaber-Batalin-Vilkovisky) algebra, whose Maurer-Cartan functor governs generalized deformations of complex structures on $X$.
Then they proved the formality of the Maurer-Cartan functor attached to $\cP$, which enables them to construct Frobenius manifold structures on $H^\bullet(X, \bC)$ with the help of $\partial\bar{\partial}$-lemma, and certain period integrals on $X$.\par

Saito's theory deals with the universal unfolding of an isolated singularity of a holomorphic function; see \cite{ST} for a nice survey. When a holomorphic function $f$, defined on a neighborhood of the origin $0$ in $\bC^{n+1}$ with $f(0)=0$, has an isolated critical point at $0$, then, under certain assumptions, Saito and his collaborators constructed Frobenius manifold structures on a small open subset of $\cO_{\bC^{n+1}, 0}/(\partial{f}/\partial{x_0},\dots, \partial{f}/\partial{x_n}) \simeq \bC^{\mu}$ where $\mu$ is an integer and $\cO_{\bC^{n+1}, 0}$ is the germ of holomorphic functions on $\bC^{n+1}$ around $0$. Their main tools are higher residue pairings $\operatorname{\cK}$ (Theorem 5.1, \cite{ST}) and primitive forms associated to $\operatorname{\cK}$ (Definition 6.1, \cite{ST}), which are realized on a certain quantization twisted de Rham complex (section 4.1, \cite{ST}) involving $d + df$.

Our key idea to construct Frobenius manifolds and $F$-manifolds lies in the unification based on Witten's gauged linear sigma model (see section 5.4, \cite{W93} for detailed physical explanations) of the above two different constructions.
If the projective smooth Calabi-Yau complete intersection variety $X_{\ud G} \subseteq \BP^n$ is defined by homogeneous polynomial equations $G_1(\ud x), \dots, G_k(\ud x)$, where $\ud x=(x_0 ,\dots, x_n)$ is the homogeneous coordinate of the projective $n$-space $\BP^n$, we will use a function $S(\underline{y},\underline{x}):=\sum_{i=1}^k y_i G_i(\ud x)$ with a new set of variables $y_1, \dots, y_k$. This function $S(\underline{y},\underline{x})$ will play a similar role as the defining equation of a hypersurface of Saito's theory. But $S(\underline{y},\underline{x})$ has a non-isolated singularity locus differently from Saito. In fact, the weight zero part of the singularity locus of $S(\underline{y},\underline{x})$ will be exactly the smooth projective complete intersection $X_{\underline{G}}$.  
As far as the authors know, the superpotential $S(\ud y, \ud x)$ was first introduced by B. Dwork in his proof of rationality of zeta functions of algebraic varieties over a finite field; the auxiliary variables $y_1, \dots, y_k$ were introduced to count the number of rational points of algebraic varieties over finite fields by using the exponential sum (see (16), \cite{Dw60}).\footnote{If one is familiar with the Griffiths-Dwork method (for calculating the Picard-Fuchs ideal), section 5.3, \cite{CS99}, our use of the function $S$ can also be viewed as a generalization of the G-D method to the algorithmic construction of Frobenius manifolds.}  But later E. Witten explained its physical meaning by looking at phase transitions which occur as parameters are varied in supersymmetric gauge theories, section 5.4, \cite{W93}.\footnote{The introduction of $S(\ud y, \ud x)$ is also known as ``the Cayley trick" to mathematicians; see subsection \ref{sec2.4}.} \par

The two facts [(1) the primitive middle-dimensional cohomology $\bH=H^{n-k}_\mathrm{pirm}(X,\mathbb{C})$ is the cohomology of the twisted de Rham complex involving $d + dS$  (see Theorem \ref{Phi_isom} and Remark \ref{cqi}) 
and (2) we can modify the multiplication structure of the twisted de Rham complex to obtain the dGBV structure (see Proposition \ref{dgbvp})]
link our study of formal Frobenius manifolds and higher residue pairings on $\bH$ to those in \cite{BK}, \cite{ST}. In order to deal with non-isolated singularities of $S(\ud y, \ud x)$, we introduce the charge grading $\textit{ch}$ (see subsection \ref{grading_subsec}) on the dGBV algebra; for more detailed comparisons with \cite{BK}, \cite{ST}, we refer to subsection \ref{subcomparison}.
 Our key tools are dGBV algebras similar to \cite{BK} and \textit{weak primitive forms} (Definition \ref{Def_WP})\footnote{In the table \ref{Comparison_Table}, we summarize how this simple notion of weak primitive forms unifies the approaches of Saito-Takahashi \cite{ST}, Barannikov-Kontsevich \cite{BK}, and ours. There is also a categorical unification of both approaches, as explained in the appendix of \cite{ST} but our unification is more explicit and down-to-earth. } which are weakened version of primitive forms in \cite{ST}.
More explicitly, our procedure can be summarized as follows:
\begin{enumerate}[(1)]
	\item Construct a dGBV algebra which computes the primitive middle-dimensional cohomology $\mathbb{H}$ of $X_{\underline{G}}$. (This construction is motivated by the ($0+0$)-dimensional QFT (Quantum Field Theory) viewpoint for the cohomology of $X_{\ud G}$, which was explained in the introduction of \cite{PP16} and Witten's work \cite{W93}.)
	\item Consider the Maurer-Cartan deformation problem associated to the shifted DGLA (differential graded Lie algebra) coming from the dGBV algebra. In particular, find a solution $\Gamma$ of the corresponding Maurer-Cartan equation.
	\item Consider the Gauss-Manin connection associated to the solution $\Gamma$ and provide an algorithm to compute the connection matrix with respect to some particular basis (the concept of \textit{weak primitive forms}).
	\item If we choose $\Gamma$ and a particular basis properly, then the connection matrix will give rise to an associative and commutative structure on the space $T_\mathbb{H}$ of the formal tangent vector fields (a formal $F$-manifold structure on $\mathbb{H}$) and a metric (a formal Frobenius manifold structure on $\mathbb{H}$).
\end{enumerate}

The theory of primitive forms and higher residue pairings, which was developed by K. Saito, provides a Frobenius manifold structure (an F-manifold structure and the metric), by taking flat coordinates with respect to a metric and takes $\hbar\to0$ (so called, a classical limit). One may think of a concept which provides an $F$-manifold in a similar manner; we introduce a notion of ``weak primitive forms'' (see Definition \ref{Def_WP}), which provides an $F$-manifold structure when one takes flat coordinates, and $\hbar\to0$.  Note that the higher residue pairing with $\hbar\to 0$ essentially corresponds to the metric. We summarize this in Table \ref{Table1}.

\begin{table}[ht]
	\caption{Weak primitive forms}
	\begin{tabular}{|c|c|c|}
	\hline
	&&\\[-1.5ex]
	Classical Complex $(\hbar\to0)$ & & Quantization complex ($\hbar$-perturbation)\\[1ex]
	\hline
	\hline
	&&\\[-1.5ex]
	Frobenius manifolds & \multirow{3}{*}{$\xleftarrow[\hbar\to0]{\begin{array}{c}\textrm{\scriptsize Take}\\\textrm{\scriptsize flat coordinates}\end{array}}$} &$\begin{array}{c}\textrm{Primitive forms and}\\\textrm{higher residue pairings}\end{array}$\\[2.5ex]
	\cline{1-1}\cline{3-3}
	&&\\[-1ex]
	$F$-manifolds & & Weak primitive forms\\[2ex]
	\hline
	\end{tabular}
	\label{Table1}
\end{table}

For any dGBV algebra and its Maurer-Cartan solution, we now define the notion of \textit{weak primitive forms} associated to them. The adjective ``weak'' means that we drop some conditions from the definition of primitive form; we take (S2) of Proposition 7.3, \cite{ST}, which is implied by (P1), (P2), and (P4) of Definition 6.1 (the original definition of primitive form), \cite{ST}, as the main ingredient of our definition.
Let $(\mathcal{A},\cdot,\Delta,\ell_2^\Delta(\cdot,\cdot),Q)$ be a dGBV algebra (Definition \ref{bvd}), and $\{\ud t= t^\alpha\}$ be a coordinate of the tangent space for the Maurer-Cartan deformation functor associated to the shifted DGLA $(\mathcal{A},\Delta+Q,\ell_2^{\Delta+Q}(\cdot,\cdot))$. Let $\Gamma\in\mathcal{A}[\![\underline{t}]\!]$ be a Maurer-Cartan solution of $(\mathcal{A},\Delta+Q,\ell_2^{\Delta+Q}(\cdot,\cdot))$:
$$
(\Delta+Q) (\G)+\frac{1}{2}\ell_2^{\Delta+Q}(\G,\G)=0,
$$
where $\Delta+Q$ and $\ell_2^{\Delta+Q}(\cdot, \cdot)$ are assumed to be $\ud t$-linear.
Then quintuple $(\mathcal{A}[\![\underline{t}]\!],\cdot,\Delta+Q+\ell_2^\Delta(\Gamma,\cdot),\ell_2^\Delta(\cdot,\cdot),Q+\ell_2^\Delta(\Gamma,\cdot))$ is again a dGBV algebra.\par
 
One can consider the following quantization dGBV algebra:
\[
	\big(\mathcal{A}[\![\underline{t}]\!](\!(\hbar)\!),\cdot,\hbar\Delta+Q+\ell_2^\Delta(\Gamma,\cdot),\tfrac{1}{\hbar}\ell_2^{\hbar\Delta}(\cdot,\cdot),Q+\ell_2^\Delta(\Gamma,\cdot)\big)
\]
with a formal parameter $\hbar$. Note that $\tfrac{1}{\hbar}\ell_2^{\hbar\Delta}(\cdot,\cdot)=\ell_2^\Delta(\cdot,\cdot)=\tfrac{1}{\hbar}\ell_2^{\hbar\Delta+Q+\ell_2^\Delta(\Gamma,\cdot)}(\cdot,\cdot)$. We define the Gauss-Manin connection as follows (see Definition \ref{GMC} for more details):
\[
	\hbar\downtriangle_\alpha^{\frac{\Gamma}{\hbar}}:=\hbar\cdot e^{-\tfrac{\Gamma}{\hbar}}\frac{\partial}{\partial t^\alpha}e^{\tfrac{\Gamma}{\hbar}}:H^r\big(\mathcal{A}[\![\underline{t}]\!](\!(\hbar)\!),\hbar\Delta+Q+\ell_2^\Delta(\Gamma,\cdot)\big)\to H^r\big(\mathcal{A}[\![\underline{t}]\!](\!(\hbar)\!),\hbar\Delta+Q+\ell_2^\Delta(\Gamma,\cdot)\big).
\]

\begin{definition}[weak primitive forms]\label{Def_WP}
	An element $\boldsymbol\zeta\in\mathcal{A}^0[\![\underline{t}]\!][\![\hbar]\!]$ is called a weak primitive form if
	\begin{enumerate}[(1)]
		\item the set $\{\hbar\downtriangle_\alpha^{\frac{\Gamma}{\hbar}}\boldsymbol\zeta\}$ is a $\mathbb{C}[\![\underline{t}]\!][\![\hbar]\!]$-basis of $H^0\big(\mathcal{A}[\![\underline{t}]\!][\![\hbar]\!],\hbar\Delta+Q+\ell_2^\Delta(\Gamma,\cdot)\big)$. In particular, there are unique $\mathbf{A}_{\alpha\beta}{}^\rho\in\mathbb{C}[\![\underline{t}]\!][\![\hbar]\!]$ such that
		\begin{equation}\label{WPF}
			\hbar\downtriangle_\beta^{\frac{\Gamma}{\hbar}}\hbar\downtriangle_\alpha^{\frac{\Gamma}{\hbar}}\boldsymbol\zeta=\sum_{\rho}\mathbf{A}_{\alpha\beta}{}^\rho\cdot\hbar\downtriangle_\rho^{\frac{\Gamma}{\hbar}}\boldsymbol\zeta+\big(\hbar\Delta+Q+\ell_2^\Delta(\Gamma,\cdot)\big)(\mathbf{\Lambda}_{\alpha\beta}),
		\end{equation}
		\item\label{cond2} the coefficient $\mathbf{A}_{\alpha\beta}{}^\rho$ is of the form $\mathbf{A}_{\alpha\beta}{}^\rho={}^0A_{\alpha\beta}{}^\rho+{}^1A_{\alpha\beta}{}^\rho\hbar$.
	\end{enumerate}
\end{definition}
From the definition, we have
                \[
				\textrm{(Commutativity)}\quad\mathbf{A}_{\alpha\beta}{}^\rho=\mathbf{A}_{\beta\alpha}{}^\rho.
	        \]
Moreover, if the coefficient ${{}^1A}_{\alpha\beta}{}^\rho=0$, then the following conditions hold:
		\[
			\begin{aligned}
				&\textrm{(Associativity)}&\sum_\rho {{}^0A}_{\alpha\beta}{}^\rho{{}^0A}_{\rho\gamma}{}^\delta&=\sum_\rho{{}^0A}_{\beta\gamma}{}^\rho{{}^0A}_{\rho\alpha}{}^\delta,\\
				&\textrm{(Potential)}\quad&\frac{\partial}{\partial t^\alpha}{{}^0A}_{\gamma\beta}{}^\rho&=\frac{\partial}{\partial t^\gamma}{{}^0A}_{\alpha\beta}{}^\rho.
			\end{aligned}
		\]
The reason we put the restrictive condition \ref{cond2} $\mathbf{A}_{\alpha\beta}{}^\rho={}^0A_{\alpha\beta}{}^\rho+{}^1A_{\alpha\beta}{}^\rho\hbar$ in the above definition is its relevance to formal $F$-manifolds; see subsection \ref{sec4.1} how weak primitive forms provides formal $F$-manifolds. Moreover, if one compares \ref{cond2} with equality (S2) in Proposition 7.3, \cite{ST}, then one can check that a weak primitive form equipped with a compatible higher residue pairing becomes a primitive form in the sense of Saito-Takahashi \cite{ST}. 
\par

%


There are two main results and two auxiliary results in the current article:
\begin{enumerate}[\bfseries (1)]
	\item An explicit algorithmic construction of formal $F$-manifolds on $\mathbb{H}$; see section \ref{Sec_F}.
	\item An explicit algorithmic construction of weak primitive forms related to $\mathbb{H}$; see section \ref{Sec_WP}.
	\item\label{sub_result} An explicit algorithmic construction of Frobenius manifolds on $\mathbb{H}$; see subsection \ref{sec5.2}.
	\item\label{sub_result2} An explicit algorithmic construction of the \textit{modified higher residue pairings}\footnote{The \textit{modified higher residue pairing} is essentially the same as the higher residue pairing but there is a slight technical difference; see Definition \ref{mhrp} and the last paragraph of subsection \ref{subcomparison} for details.} related to $\mathbb{H}$; see subsection \ref{sec5.3}.
\end{enumerate}
The reason why we call \ref{sub_result} and \ref{sub_result2} auxiliary results is that their constructions are more or less technically trivial. 
All these algorithms will be deduced by solving the equation \eqref{WPF} for weak primitive forms. For different choices of $\Gamma,\boldsymbol\zeta,\mathbf{A}_{\alpha\beta}{}^\rho$, we will solve the differential equation \eqref{WPF} algorithmically and connect them to the concepts of F-manifolds, Frobenius manifolds, and modified higher residue pairings. We record such different choices in the table \ref{Comparison_Table}. The equation \eqref{WPF} is subtle to solve; see Remarks \ref{nontrivialE} and Appendix \ref{sec6.2}. In fact, the key technical obstacle for the current project was to solve algorithmically its special case \eqref{F-QM2}, which is also a main differential equation of \cite{BK} appeared in the proof of the lemma 7.1, \cite{BK}. \par


We need to mention the work \cite{LLS} of Li-Li-Saito, the book of Hertling \cite{Her02}, and Sabbah's work \cite{Sab}. There is a similarity between ours and \cite{LLS} in that both use dGBV algebras (based on the polyvector fields) to deal with higher residue pairings and to unify the methods of \cite{BK} and \cite{ST}. But there is a notable difference: the main set-up of \cite{LLS} consists of a Stein domain $M \subset \bC^m$ and a holomorphic function $f: M \to \bC$ with finite critical points (section 1.2, \cite{LLS}), which corresponds to the affine space $\bC^{n+k+1}$ itself and the polynomial function $S:\bC^{n+k+1} \to \bC$, which does not have finite critical points, in our set-up.\footnote{The physical background of our theory is \cite{W93}, which is different from the physical background of \cite{LLS}, the BCOV (Bershadsky, Cecotti, Ooguri, and Vafa) theory.} Thus we can not apply \cite{LLS} to study higher residue pairings and primitive forms relevant to $\bH=H_\mathrm{prim}^{n-k}(X,\mathbb{C})$.
This phenomenon regarding non-isolated singularities limits us to obtain results only on ``weak" primitive forms in our set-up. Our method, a systematic algorithmic approach to the differential equation \eqref{WPF}, which naturally appears in the study of primitive forms, is not refined enough to obtain results on primitive forms.
In \cite{LLS}, \textit{a good opposite filtration} plays a crucial role to produce analytic primitive forms; the Hodge filtration on $\bH$ seems to be related to good opposite filtrations (see subsection \ref{sec6.1} for some evidence) but we are not sure how those are related precisely. The book \cite{Her02} on Frobenius manifolds concentrates on the case of isolated singularities and \cite{Sab} deals with tame functions with isolated singularities, while we work with the case of non-isolated singularities. \par
\vspace{1em}


Now we explain the contents of the paper briefly. Section \ref{sec2} is devoted to explain how to construct the dGBV algebra which computes the primitive middle-dimensional cohomology $\bH$ of $X_{\ud G}$ and to study its properties.
The materials from subsection \ref{sec2.1} to subsection \ref{sec2.7} are not limited to the Calabi-Yau case. But from subsection \ref{sec2.8} throughout the article, we will assume that $X_{\ud G}$ is Calabi-Yau.
In subsection \ref{sec2.1}, the underlying cochain complex $(\cA, K_S=Q_S + \Delta)$ of the dGBV algebra is constructed (see \eqref{dgbv}). In subsection \ref{grading_subsec}, the grading structure of $(\cA, K_S)$ is explained. In subsection \ref{sec2.3}, the grading structure of the relevant twisted de Rham complex is explained.
Then, in subsection \ref{sec2.4}, we explain how $(\cA, K_S)$ compares with the twisted de Rham complex and consequently construct an explicit isomorphism between its 0-th cohomology and the primitive middle-dimensional cohomology (Theorem \ref{dgbvm}) via the Cayley trick. In subsection \ref{sec2.5}, we make the dGBV algebra $(\cA, \cdot, K_S, \ell^{K_S}, Q_S)$ by introducing the $\ell_2$-descendant $\ell_2^{K_S}$ (Definition \ref{bvd} and Proposition \ref{dgbvp}). In subsection \ref{sec2.6}, we compare the cohomologies of $(\cA, K_S)$ and $(\cA, Q_S)$ and derive the weak $Q_S\Delta$-lemma (Lemma \ref{Q-D_Lemma}) which plays a crucial role in our algorithm. In subsection \ref{sec2.7}, a formal parameter $\hbar$ is introduced to construct the quantization dGBV algebra. In subsection \ref{sec2.8}, we explain the deformation problem attached to $(\cA, \cdot, K_S, \ell^{K_S}, Q_S)$ and the corresponding Gauss-Manin connection on the deformation space.

Section \ref{Sec_F} is devoted to constructing the algorithm for formal $F$-manifolds on $\bH$. In subsection \ref{sec3.1}, we briefly recall the notion of formal $F$-manifolds. In subsection \ref{sec3.2}, we explain how the connection matrix for the Gauss-Manin connection leads to the formal $F$-manifolds. We provide an algorithm in subsection \ref{Subsec_Alg} and its proof in subsection \ref{sec3.4}.

Section \ref{Sec_WP} is devoted to constructing the algorithm for weak primitive forms when $\Gamma$ is of the form $L=\sum_{\a}t^\a u_\a$ (linear form). In subsection \ref{sec4.1}, we briefly explain how the weak primitive form is related to Saito's primitive form and gives rise to formal $F$-manifolds. We provide an algorithm in subsection \ref{Subsec_Alg_prim} and its proof in subsection \ref{sec4.3}.

Section \ref{Sec_Frob} is about the algorithms for the formal Frobenius manifold and the modified higher residue pairing on $\bH$. In subsection \ref{sec5.1}, we briefly recall the notion of formal Frobenius manifolds. We give an algorithm for Frobenius manifolds in subsection \ref{sec5.2}. In subsection \ref{sec5.3}, we explain how to produce an algorithm of the modified higher residue pairing which is essentially the same as the original higher residue pairing of Theorem 5.1, \cite{ST}; we point out a subtle difference between them and explain what is missing for constructing an algorithm of the original higher residue pairing.

Finally, in the appendix, we provide comparisons of our work with \cite{BK} and \cite{ST} (subsection \ref{subcomparison}). Then we give some remark on the Riemann-Hilbert-Birkhoff problem (subsection \ref{sec6.1}) and how to solve the equation \ref{WPF} algorithmically (subsection \ref{sec6.2}), regardless of its connection to $F$-manifolds and Frobenius manifolds.

\subsection{Acknowledgement}
The problem of constructing an algorithm for formal $F$-manifolds arose from the co-work (\cite{PP16}) and conversation of the second named author with Jae-Suk Park, to whom we would like to thank deeply for sharing his ideas and insights in \cite{Pa10}. Jeehoon Park was partially supported by Samsung Science \& Technology Foundation (SSTF-BA1502)
and BRL (Basic Research Lab) through the NRF (National Research Foundation) of South Korea (NRF-2018R1A4A1023590). We appreciate Cheol-hyun Cho for providing valuable comments on the first draft. Jeehoon Park would also like to thank KIAS (Korea Institute for Advance Study) for its support and hospitality during his stay on sabbatical leave, where the part of the paper was written.


\section{The dGBV Algebra and Quantization} \label{sec2}
\subsection{\texorpdfstring{Cochain complex $\mathcal{A}$}{Cochain complex AX}}\label{sec2.1}
Let $\mathbf{P}^n$ be a $n$-dimensional projective space over $\mathbb{C}$ for $n\geq 1$. Denote by $\mathbb{C}[\underline{x}]$ the usual homogeneous coordinate ring of $\mathbf{P}^n$ with $\underline{x} = (x_0, x_1, \dots, x_n)$. For $n\geq k \geq 1$, let $G_1(\underline{x}), \dots, G_k(\underline{x})$ be homogeneous polynomials of degree $d_1, \dots, d_k$ respectively. We consider a smooth projective variety $X_{\underline{G}}$ embedded in $\mathbf{P}^n$ defined by $G_1(\underline{x}), \dots, G_k(\underline{x})$, which satisfies the complete intersection property, i.e. $\dim X_{\underline{G}} = n-k$.\par
Let $X=X_{\underline{G}}(\mathbb{C})$ be the complex analytic manifold associated to $X_{\underline{G}}$ and consider the Dwork potential
\begin{equation}\label{DP}
	S(\underline{y},\underline{x}) := \sum_{\ell=1}^k y_\ell \cdot G_\ell(\underline{x}),
\end{equation}
where we introduce the formal variables $y_1,\dots,y_k$ corresponding to $G_1,\dots,G_k$. Let $N=n+k+1$, $q_1=y_1,\dots,q_k=y_k,q_{k+1}=x_0,\dots,q_N=x_n$ and define the $\mathbb{Z}$-graded super-commutative algebra $\mathcal{A}$ and differentials $\Delta$, $Q_S$ and $K_S$ as follows:
\begin{equation}\label{dgbv}
	\begin{aligned}
	\mathcal{A}&:= {\mathbb{C}}[\underline{q}][\underline{\eta}]={\mathbb{C}}[q_{1},q_2,\dots,q_N][\eta_{1},\eta_2,\dots,\eta_N],\\
	\Delta&:=\sum_{i=1}^N\pa{q_i}\pa{\eta_i}:\mathcal{A}\to \mathcal{A},\\
	Q_S&:=\sum_{i=1}^N \prt{S(\underline{q})}{q_i} \pa{\eta_i}: \mathcal{A} \to \mathcal{A},\\
	K_S&:=Q_S+\Delta.
	\end{aligned}
\end{equation}
where $\eta_i$'s are other variables (of cohomological degree $-1 $) corresponding to $q_i$'s.
The additive cohomological $\bZ$-grading of $\mathcal{A}$ is given 
by the rules ($|f|=m$ means $f \in \mathcal{A}^m$)
\[
	|q_i|=0, |\eta_i|=-1, \quad i=1, \dots, N.
\]
Then we have the following cochain complex\footnote{Note that the super-commutativity means that $a\cdot b=(-1)^{|a| |b|} b \cdot a$ for homogeneous elements $a, b$. Hence we see that $\eta_i \cdot \eta_j = - \eta_j \cdot \eta_i$, which implies that $\eta_i^2 =0$ and $\mathcal{A}^{-N-1}=\mathcal{A}^{-N-2}=\cdots=0$. }
\[
	0 \to \mathcal{A}^{-N} \maplra{K_S} \mathcal{A}^{-N+1} \maplra{K_S} \cdots \maplra{K_S} \mathcal{A}^0 \to 0.
\]
Let $A:=\cA^0=\bC[q_1, \dots, q_N]$.
We define the $\ell_2$-descendant of $K_S$ with respect to the product $\cdot$ as follows:
\[
	\ell_2^{K_S}(a,b):= K_S(ab)-K_S(a)b -(-1)^{|a|} a K_S(b), \quad a,b \in \mathcal{A}.
\]\par

\subsection{\texorpdfstring{Gradings on $\mathcal{A}$}{Grading on AX}}\label{grading_subsec}
We are going to introduce another two $\bZ$-gradings on $\mathcal{A}$ defined in \eqref{dgbv}: charge and weight.
For each  $q_\m$, assign a non-zero integer $\ch(q_\m)$ called
charge of $q_\m$ as follows:
\[
	\ch(q_i)=\ch(y_i)=-d_i, \quad \text{for } i=1, \dots, k, \quad \ch(q_i)=\ch(x_{i-k-1}) = 1, \quad \text{for }  i=k+1, \dots, N. 
\]
Also  assign $ \ch(\eta_\m):=-\ch(q_\m)$. Define the background charge $c_X$ by
\[
	c_X :=- \sum_{\m}\ch(q_\m) = \sum_{i=1}^k d_i - (n+1).
\]
Associated with $\ch(q_\m)_{\m=1,\dots, N}$, we define the Euler vector
\[
	E_{\ch} = \sum_{\m=1}^{N} \ch(q_\m)q_\m\frac{\rd}{\rd q_\m}+\sum_{\m=1}^{N} \ch(\eta_\m)\eta_\m\frac{\rd}{\rd \eta_\m}.
\]
For any monomial $M \in \mathcal{A}$, we have $E_{\ch}(M) =\ch(M)M$. Hence we have the charge decomposition of $\mathcal{A}$:
\[
	\mathcal{A} =\bigoplus_{\l \in \Z} \mathcal{A}_\l,
\]
where $\mathcal{A}_\l$ is the ${\C}$-subspace of $\mathcal{A}$ generated by the monomials of charge $\l$.
We also introduce the notion of weight such that $\wt(y_\ell)=1$ for all $ \ell =1,\dots, k$ and $\wt(x_i) = 0$ for $i=0,1,\dots, n$.
Define $\wt(\eta_\m) = 1 -\wt(q_\m)$ for all $\m=1,\dots, N$.
Associated with the weight, we define the Euler vector
\[
	E_{\wt} = \sum_{\m=1}^{N} \wt(q_\m)q_\m\frac{\rd}{\rd q_\m}+\sum_{\m=1}^{N} \wt(\eta_\m)\eta_\m\frac{\rd}{\rd \eta_\m} 
 =\sum_{\ell=1}^k y_\ell \frac{\rd}{\rd y_\ell}+ \sum_{\ell=k+1}^N \eta_\ell \frac{\rd}{\rd \eta_\ell}.
\]
For any monomial $M \in \mathcal{A}$, we have $E_{\wt}(M) =\wt(M)M$. Hence we have the weight decomposition of $\mathcal{A}$:
\[
	\mathcal{A} =\bigoplus_{w \geq 0} \mathcal{A}_{(w)},
\]
where $\mathcal{A}_{(w)}$ is the ${\C}$-subspace of $\mathcal{A}$ generated by the monomials of weight $w$.
Using the weight, we consider the increasing filtration $F^{\bullet}_{\wt}\mathcal{A}$ defined by
\begin{equation}\label{wf}
	F^{j}_{\wt}\mathcal{A} =\bigoplus_{0\leq w\leq j}\mathcal{A}_{(w)}.
\end{equation}

\begin{proposition}
	The weight filtration induces a filtered  cochain  complex $(F^\bullet_{\wt}\mathcal{A}, K_S)$.
\end{proposition}

\begin{proof}
	Note that $K_S=Q_S+\Delta$. Since $Q_S$ preserves $\wt$ and $\Delta$ decreases $\wt$ by 1, the result follows.%
\end{proof}

Note that $\mathcal{A} =A [\ud \eta]$, where $\eta_\m \eta_\n= -\eta_\n \eta_\m$, and
\[
	|q_\m|+|\eta_\m|=-1,\quad\ch(q_\m)+\ch(\eta_\m)=0,\quad\wt(q_\m)+\wt(\eta_\m)=1.
\]
We then have decomposition
\[
	\mathcal{A} = \bigoplus_{-N\leq j \leq 0}\bigoplus_{\l \in \Z}\bigoplus_{w \geq 0}\mathcal{A}^j_{\l, (w)},
\]
where $u \in \mathcal{A}^j_{\l, (w)}$ if and only if $ E_{\ch} (u) = \l \cdot u$, $ E_{\wt} (u) = w \cdot u$, and $|u|=j$.
Note that $S(\ud{q})\in \mathcal{A}^0_{0,(1)}$ \big(the Dwork potential $S(\ud{q})$ was defined in \eqref{DP}\big).

\subsection{The de Rham complex and gradings}\label{sec2.3}
Now we consider an algebraic twisted de Rham complex. For any commutative $\bC$-algebra C, the de Rham complex $(\Omega^\bullet_C, d)$ is defined and for any element $S\in C$, the twisted de Rham complex $(\Omega^\bullet_C, d+ dS \wedge)$ is defined. In particular, we will consider the following de Rham complexes:
\[
	(\Omega_A^\bullet,dS\wedge+d),~(\Omega_{A[S^{-1}]}^\bullet,d),~(\Omega_{A[S^{-1}]_{0,(0)}}^\bullet,d).
\]
We also define the charge and the weight on the de Rham complex by
\[
	\begin{cases}
		\ch(q_i)=\ch(dq_i), &i=1,2,\dots,N\\
		\wt(q_i)=\wt(dq_i). &i=1,2,\dots,N\\
	\end{cases}
\]
We have the weight and the charge decomposition of $\Omega$ such that
\[
	\Omega = \bigoplus_{-N\leq j \leq 0}\bigoplus_{\l \in \Z}\bigoplus_{w \geq 0}\Omega^j_{\l, (w)}.
\]

\subsection{Main theorem on cohomologies} \label{sec2.4}
We will find an explicit relationship between the cochain complex $(\mathcal{A},K_S)$ and the primitive middle dimensional cohomology $H_{\mathrm{prim}}^{n-k}(X_{\underline{G}},\mathbb{C})$.\par
If we are interested in the cohomology of the smooth projective complete intersection variety $X$ of dimension $n-k$, then the primitive middle dimensional cohomology $H_{\mathrm{prim}}^{n-k}(X,\mathbb{C})$ is the most interesting piece because the other degree cohomologies and non-primitive pieces can be easily described in terms of the cohomology of the projective space $\mathbf{P}^n$ due to the weak Lefschetz theorem and Poincare duality. For the computation of $H_{\mathrm{prim}}^{n-k}(X,\mathbb{C})$, the Gysin sequence and ``the Cayley trick'' play important roles. There is a long exact sequence, called the Gysin sequence:
\[
	\begin{tikzcd}
		\cdots\to H^{n+k-1}(\mathbf{P}^n,\mathbb{C})\to H^{n+k-1}(\mathbf{P}^n\setminus X,\mathbb{C})\arrow[r,"\mathrm{Res}_X"] & H^{n-k}(X,\mathbb{C})\to H^{n+k}(\mathbf{P}^n,\mathbb{C})\to\cdots,
	\end{tikzcd}
\]
where $\mathrm{Res}_X$ is the residue map (see p.96 of \cite{Dim95}). This sequence gives rise to an isomorphism
\[
	\mathrm{Res}_X:H^{n+k-1}(\mathbf{P}^n\setminus X,\mathbb{C})\stackrel{\sim}{\to}H_{\mathrm{prim}}^{n-k}(X,\mathbb{C}).
\]
The Cayley trick is about translating a computation of the cohomology of the complement of a complete intersection into a computation of the cohomology of the complement of a hypersurface in a bigger space. Let $\mathcal{E}:=\mathcal{O}_{\mathbf{P}^n}(d_1)\oplus\cdots\oplus\mathcal{O}_{\mathbf{P}^n}(d_k)$ be the locally free sheaf of $\mathcal{O}_{\mathbf{P}^n}$-modules with rank $k$. Let $\mathbf{P}$ be the projective bundle associated to $\mathcal{E}$ with fiber $\mathbf{P}^{k-1}$ over $\mathbf{P}^n$. Then $\mathbf{P}(\mathcal{E})$ is the smooth projective toric variety whose homogeneous coordinate ring is given by $A$. Then $S$ defines a hypersurface $X_S$ in $\mathbf{P}(\mathcal{E})$. The natural projection map $\mathbf{P}(\mathcal{E})\to\mathbf{P}^n$ induces a morphism $\mathrm{pr}_1:\mathbf{P}(\mathcal{E})\setminus X_S \to\mathbf{P}^n\setminus X$ which can be checked to be a homotopy equivalence. One can check that
\[
	s(\underline{x})=\left(\underline{x},G_1(\underline{x})^{e_1-1}\overline{G_1(\underline{x})}^{e_1},\dots,G_k(\underline{x})^{e_k-1}\overline{G_k(\underline{x})}^{e_k}\right),\quad e_i=\tfrac{\mathrm{lcm}(d_1,\dots,d_k)}{d_i}
\]
is a well-defined section of $\mathrm{pr}_1$. Hence there exists an isomorphism
\[
	s^{*}:H^{n+k-1}(\mathbf{P}(\mathcal{E})\setminus X_S,\mathbb{C})\stackrel{\sim}{\to} H^{n+k-1}(\mathbf{P}^n\setminus X,\mathbb{C}).
\]
		
\begin{definition} 
We define a sequence of maps as follows:
	\begin{enumerate}[(i)]
		\item Let $\mu:(\mathcal{A}_{c_X}^\bullet,Q_S+\Delta)\to\left(\left(\Omega_A^\bullet\right)_0[-N],dS+d\right)$ be the cochain map defined by
		\[
			\mu(\underline{q}^{\underline{u}}\eta_{i_1}\cdots\eta_{i_\ell})=(-1)^{i_1+\cdots+i_\ell+\ell+1}\underline{q}^{\underline{u}}dq_1\cdots\widehat{dq_{i_1}}\cdots\widehat{dq_{i_\ell}}\cdots dq_N
		\]
		where $\underline{q}^{\underline{u}}= q_1^{u_1} \cdots q_N^{u_N}$.
		\item Let $\rho:\left(\left(\Omega_A^\bullet\right)_0,dS+d\right)\to\left(\left(\Omega_{A[S^{-1}]}^\bullet\right)_{0,(0)},d\right)$ be the cochain map\footnote{We found this formula by trial and error. A simple computation confirms that it is indeed a cochain map.} defined by
		\[
			\rho(\underline{x}^{\underline{u}}\underline{y}^{\underline{v}}d \underline{x}_{\underline{\alpha}}\wedge d\underline{y}_{\underline{\beta}})=(-1)^{|\underline{v}|+|\underline{\beta}|-1}(|\underline{v}|+|\underline{\beta}|-1)!\frac{\underline{x}^{\underline{u}}\underline{y}^{\underline{v}}}{S^{|\underline{v}|}}d \underline{x}_{\underline{\alpha}}\wedge\frac{d \underline{y}_{\underline{\beta}}}{S^{|\underline{\beta}|}},
		\]
		where $\underline{x}^{\underline{u}} = x_0^{u_0} \cdots x_n^{u_n}$, $\underline{y}^{\underline{v}}=y_1^{v_1}\cdots y_k^{v_k}$, $d \underline{x}_{\underline{\alpha}}=dx_{\a_1}\wedge \cdots \wedge dx_{\a_i}$, and $d\underline{y}_{\underline{\beta}}=dy_{\b_1} \wedge \cdots \wedge dy_{\b_j}$ with  $|\ud v| =v_{1} + \cdots + v_{k}$ and $|\underline{\beta}|=j$.
		\item Let $\theta_{\ch}$ be the contraction operator with the vector field $\sum_{i=1}^N\ch(q_i)q_i\frac{\partial}{\partial q_i}$, and $\theta_{\wt}$ be the contraction operator with the vector field $\sum_{i=1}^N\wt(q_i)q_i\frac{\partial}{\partial q_i}$.
	\end{enumerate}
\end{definition}
Therefore, we have the following maps:
\[
	\begin{tikzcd}
		\mathcal{A}_{c_X}^0 \arrow[r,"\mu"] & \left(\Omega^N_{A}\right)_0 \arrow[r,"\rho"] & \left(\Omega^N_{A[S^{-1}]}\right)_{0,(0)} \arrow[r,"\theta_{wt}\circ\theta_{ch}"] & \Omega_{B}^{N-2} \arrow[r,"\textrm{quotient}"] & \Omega_B^{N-2}/d(\Omega_B^{N-3})
	\end{tikzcd}
\]
and
\[
	\begin{tikzcd}
		\Omega_B^{N-2}/d(\Omega_B^{N-3})=H_{\mathrm{dR}}^{N-2}(\mathbf{P}(\mathcal{E})\setminus X_S)\arrow[r,"\sim","s^*"'] & H_{\mathrm{dR}}^{N-2}(\mathbf{P}^n\setminus X_{\underline{G}})\arrow[r,"\sim","\mathrm{Res}_{\underline{G}}"'] & H_{\mathrm{prim}}^{n-k}(X_{\underline{G}},\mathbb{C})
	\end{tikzcd}
\]
where $B:=\mathbb{C}[\underline{q},S^{-1}]_{0,(0)}$ and $\mathbf{P}(\mathcal{E})\setminus X_S$ is a smooth affine variety whose coordinate ring is given by $B$. Let $\varphi:\mathcal{A}_{c_X}^0\to H_{\mathrm{prim}}^{n-k}(X_{\underline{G}},\mathbb{C})$ be the composition of the above $\mathbb{C}$-linear maps.

\begin{theorem}\label{dgbvm}
	The map $\varphi$ is surjective and $\mathrm{ker}\varphi\simeq(Q_S+\Delta)(\mathcal{A}^{-1}_{c_X})$. Thus we have an induced isomorphism $\varphi$ (using the same notation) of finite dimensional $\mathbb{C}$-vector spaces
	\[
		\begin{tikzcd}
			\mathcal{A}_{c_X}^0/(Q_S+\Delta)(\mathcal{A}^{-1}_{c_X})\arrow[r,"\sim","\varphi"'] & H_{\mathrm{prim}}^{n-k}(X_{\underline{G}},\mathbb{C}).
		\end{tikzcd}
	\]
\end{theorem}
Moreover, $\varphi$ sends the increasing filtration in \eqref{wf} to the increasing Hodge filtration of $H^{n-k}_{\pr}(X,\bC)$. The above explicit description of the map $\varphi$ in terms of series of cochain maps is our explicit reformulation\footnote{We could not find any reference which contains explicit descriptions of cochain maps realizing the isomorphism between $H_{\mathrm{prim}}^{n-k}(X_{\underline{G}},\mathbb{C})$ and $\mathcal{A}_{c_X}^0/(Q_S+\Delta)(\mathcal{A}^{-1}_{c_X})$. } of the results known to experts (see Theorem 1 in \cite{Dim95} and see \cite{Gr69}, \cite{Ko91}, and \cite{T}). This reformulation provides us an algebraic BV formalism for cohomologies and period integrals of smooth projective complete intersection varieties. This BV formalism will enable us to construct explicit algorithms for formal $F$-manifolds, weak primitive forms, formal Frobenius manifolds, and the modified higher residue pairings.

\subsection{\texorpdfstring{The dGBV algebra associated to $X$}{The dGBV algebra associated to X}}\label{sec2.5}
We explain that the quintuple $(\mathcal{A},\cdot,K_S,\ell_2^{K_S},Q_S)$ is a dGBV algebra. We briefly review the definitions of G-algebras, BV-algebras, GBV-algebras, and dGBV-algebras.

\begin{definition}\label{bvd}
	Let $k$ be a field. Let $(\cC,\cdot)$ be a unital $\Z$-graded super-commutative and associative $k$-algebra.
	Let $[\cdot, \cdot]: \cC \otimes \cC \to \cC$ be a bilinear map of degree 1.
	\begin{enumerate}[(a)]
		\item $(\cC, \cdot, [\cdot, \cdot])$ is called a G-algebra (Gerstenhaber algebra) over $k$ if
		\begin{align*}
			\quad [a, b] &= (-1)^{|a||b|}[b,a],\\
			[a, [b,c]]&= (-1)^{|a|+1}[[a,b],c]+(-1)^{(|a|+1)(|b|+1)} [b, [a,c]],\\
			\quad [a,b \cdot c]&= [a, b] \cdot c +(-1)^{(|a|+1)\cdot |b|} b \cdot [a,c],
		\end{align*}
		for any homogeneous elements $a, b, c \in \cC.$
		\item $(\cC,\cdot, K,\ell_2^{K})$ is called a GBV(Gerstenhaber-Batalin-Vilkovisky)-algebra\footnote{$(\cC,\cdot, K)$ is called a BV algebra, if $(\cC,\cdot, K,\ell_2^{K})$ is a GBV algebra} over $k$ where
		\begin{align}\label{elltwo}
			\ell_2^{K}(a,b):= K(a \cdot b)-K(a)\cdot b -(-1)^{|a|} a\cdot K(b), \quad a,b \in \cC,
		\end{align}
		if $(\cC, K, \ell_2^{K})$ is a shifted DGLA(differential graded Lie algebra) and $(\cC, \cdot, \ell_2^{K})$ is a G-algebra,
		\item $(\cC, \cdot, K, \ell_2^{K}, Q)$, where $Q:\cC \to \cC$ is a linear map of degree 1, is called a dGBV(differential Gerstenhaber-Batalin-Vilkovisky) algebra if $(\cC, \cdot, K,  \ell_2^{K}(\cdot, \cdot))$ is a GBV algebra
		and $(\cC,\cdot,Q)$ is a cdga(commutative differential graded algebra), i.e.
		\[
			Q^2=0, \quad Q(a \cdot b) = Q(a) \cdot b + (-1)^{|a|} a \cdot Q(b), \quad a,b \in \cC.
		\]
	 \end{enumerate}
\end{definition}

We also introduce the $\mathbb{C}$-linear map
\[
	\Delta:= K-Q=\sum_{i=1}^N  \pa{q_i} \pa{\eta_i}:\mathcal{A} \to \mathcal{A}.
\]
$\Delta$ is also a differential of degree 1 ($Q$ and $K$ also have degree 1), i.e. $\Delta^2=0$. Therefore we have
\[
	\Delta Q +Q \Delta=0.
\]
Then a straightforward (lengthy) computation shows the following proposition:
\begin{proposition}\label{dgbvp}
	The quintuple $(\mathcal{A}, \cdot, K_S, \ell_2^{K_S}, Q_S)$ defined in \eqref{dgbv} is a dGBV algebra over $\bC$ and so is $(\mathcal{A}, \cdot, \Delta, \ell_2^{\Delta}, Q_S)$.
\end{proposition}

\subsection{\texorpdfstring{The 0-th cohomology of ($\mathcal{A},K_S$) and the Jacobian ideal of $S$}{The 0-th cohomology of (AX,KX) and the Jacobian ideal of S}} \label{sec2.6}
Associated with $E_{\ch}$, we define
\[
	R:= \sum_{i=1}^N \ch(q_i) q_i \eta_i \in \mathcal{A}^{-1}_{0}.
\]
Note that, for any $f\in \mathcal{A}$,
\[
	\delta_Rf:=\ell^{K_S}_2(R,f)=\sum_{i=1}^N\textit{ch}(q_i)\left(q_i\frac{\partial}{\partial q_i}-\eta_i\frac{\partial}{\partial\eta_i}\right)f.
\]
Hence for any $f \in \mathcal{A}_\l$, we have $\d_R f = \l f$. Note also that $Q_S R =0$, while
\[
	K_S R =- c_X, \qquad c_X:=\sum_{i=1}^N-\ch(q_i)=\sum_{i=1}^k d_i -n-1=\hbox{the background charge}.
\]
Note that $c_X=0$ means that $X$ is a Calabi-Yau variety. Since $K_S$ preserves the charge grading, we have the following proposition.

\begin{proposition}
	For each charge $\l \in \Z$, we have a cochain subcomplex $(\mathcal{A}_\l, K_S)$.
\end{proposition}

It follows that the cohomology $H=H(\mathcal{A},K_S)$ of the cochain complex $(\mathcal{A}, K_S)$ has a charge decomposition
\[
	H =\bigoplus_{\l \in \Z} H_\l, \qquad H_\l := H(\mathcal{A}_\l,  K_S).
\]
\begin{proposition} \label{chc}
	The cohomology $H =H(\mathcal{A}, K_S)$ is concentrated in charge $c_X$, i.e., $H = H_{c_X}$. In particular, we have an isomorphism
	\[
		\mathcal{A}_{c_X}^0/(Q_S+\Delta)(\mathcal{A}_{c_X}^{-1})\simeq\mathcal{A}^0/(Q_S+\Delta)(\mathcal{A}^{-1}).
	\]
\end{proposition}

\begin{proof}
	From
	\begin{align*}
		K_S(f\cdot R) &=\ell_2^{K_S}(R, f) + K_SR\cdot f + K_Sf\cdot R\\
		&=\delta_R f - c_X f + K_Sf\cdot R\\
		&=(\lambda-c_X)f+K_Sf\cdot R~(\text{if}~f\in\mathcal{A}_\lambda),
	\end{align*}
	we see that any $f \in \mathcal{A}_\l\cap \Ker K_S$  belongs to $\Im K_S$ unless $\l=c_X$, since $ (\l -c_X)f= K_S(f\cdot R)$. 
\end{proof}

\begin{lemma}[Weak $Q_S\Delta$-lemma]\label{Q-D_Lemma}
	For $f\in\mathcal{A}_{c_X}^{-1}$, assume that $Q_S(f)=0$. Then $\Delta(f)=-Q_S(\Delta(v))$ for some $v\in\mathcal{A}_{c_X}^{-2}$.
\end{lemma}

\begin{proof}
	Let $\{u_\alpha:\alpha\in I\}$ be a $\mathbb{C}$-basis of $\mathcal{A}_{c_X}^0/Q_S(\mathcal{A}_{c_X}^{-1})$. Because there is an isomorphism $H^0(\mathcal{A}_{c_X},Q_S)\simeq H^{-1}(\mathcal{A}_{c_X},Q_S)$ given by $u\mapsto Ru$ where $R=\sum_i\ch(q_i)q_i\eta_i$ (see \cite{PP16}, Proposition 4.14), $\{Ru_\alpha:\alpha\in I\}$ is a $\mathbb{C}$-basis of $H^{-1}(\mathcal{A}_{c_X},Q_S)$. Since $Q_S(f)=0$, $f$ can be written as
	\[
		f=\sum_{\rho\in I}a^\rho\cdot Ru_\rho+Q_S(v)
	\]
	for some $v\in\mathcal{A}_{c_X}^{-2}$. A direct computation shows that
	\[
		\Delta(Ru_\rho)=0.
	\]
	Thus, we have $\Delta(f)=\Delta(Q_S(v))=-Q_S(\Delta(v))$.
\end{proof}

\begin{lemma}\label{bfQ_lemma}
	We have an isomorphism
	\[
		\begin{tikzcd}
			\mathcal{A}_{c_X}^0/(Q_S+\Delta)(\mathcal{A}_{c_X}^{-1}) & \mathcal{A}_{c_X}^0/Q_S(\mathcal{A}_{c_X}^{-1})\arrow[l,"\sim"']
		\end{tikzcd}
	\]
	as $\mathbb{C}$-vector spaces.
\end{lemma}

\begin{proof}
	Let $\dim \mathcal{A}^0_{c_X}/Q_S(\mathcal{A}_{c_X}^{-1})=\mu$. We can choose $e_1,\dots,e_\mu\in \mathcal{A}^0_{c_X}$ such that $\{e_\alpha+\mathrm{Im}(Q_S):\alpha\in I\}$ is $\mathbb{C}$-basis of $\mathcal{A}^0_{c_X}/Q_S(\mathcal{A}_{c_X}^{-1})$.\par
	Suppose that
	\[
		\sum_{\alpha=1}^\mu a_\alpha\cdot e_\alpha\equiv0\mod Q_S+\Delta.
	\]
	We want to show $a_\alpha=0$ for all $\alpha$, i.e., $\{e_\alpha+\mathrm{Im}(Q_S+\Delta):\alpha\in I\}$ is linearly independent in $\mathcal{A}^0_{c_X}/(Q_S+\Delta)(\mathcal{A}_{c_X}^{-1})$.\par
	Let $R_{(i)}=\sum_{\wt(e_\alpha)=i}a_\alpha\cdot e_\alpha$. Therefore
	\[
		\sum_{\alpha=1}^\mu a_\alpha\cdot e_\alpha=R_{(0)}+R_{(1)}+\cdots+R_{(n-k)}=(Q_S+\Delta)(\eta_{(0)}+\eta_{(1)}+\cdots+\eta_{(n-k+1)}),
	\]
	where $\wt(\eta_{(i)})=i$. Note that $\wt(\Delta)=-1$ and $Q_S(\eta_{n-k+1})=0$.\par
	Assuming that $R_{(i)}\neq0$ for some $i$, we can choose the largest $r$ such that $R_{(r)}\neq0$. Then
	\[
		R_{(r)}=Q_S(\eta_{(r)})+\Delta(\eta_{(r+1)}).
	\]
	But $Q_S(\eta_{r+1})=R_{(r+1)}=0$. The weak $Q_S\Delta$-lemma \ref{Q-D_Lemma} says that $\Delta(\eta_{(r+1)})\in\mathrm{Im}(Q_S)$. Therefore,
	\[
		R_{(r)}\in\mathrm{Im}(Q_S),~\textrm{i.e.,}~R_{(r)}\equiv0\mod Q_S.
	\]
	Since $\{e_\alpha+\mathrm{Im}(Q_S):\alpha\in I\}$ is $\mathbb{C}$-basis of $\mathcal{A}^0_{c_X}/Q_S(\mathcal{A}_{c_X}^{-1})$, $a_\alpha=0$ for all $\a$ with $\wt(e_\alpha)=r$. It means that $R_{(r)}=0$, which contradicts to the assumption that $R_{(i)}\neq0$ for some $i$. Thus $R_{(i)}=0$ for all $i$. Therefore $a_\alpha=0$ for all $\alpha\in I$.\par
	Now we verify that $\{e_\alpha+\mathrm{Im}(Q_S+\Delta)\}$ spans $\mathcal{A}^0_{c_X}/(Q_S+\Delta)(\mathcal{A}_{c_X}^{-1})$. For that we write $x$ as
	\begin{align*}
		x&=\sum_{\alpha\in I}n_{\alpha,0}\cdot e_\alpha+Q_S(u_0),\quad n_{\alpha,0}\in\mathbb{C},~u_0\in\mathcal{A}_{c_X}^{-1},\\
		&=\sum_{\alpha\in I}n_{\alpha,0}\cdot e_\alpha-\Delta(u_0)+K_S(u_0).
	\end{align*}
	Then we can write $-\Delta(u_0)$ as
	\begin{align*}
		-\Delta(u_0)&=\sum_{\alpha\in I}n_{\alpha,1}\cdot e_\alpha+Q_S(u_1),\quad n_{\alpha,1}\in\mathbb{C},~u_1\in\mathcal{A}_{c_X}^{-1},\\
		&=\sum_{\alpha\in I}n_{\alpha,1}\cdot e_\alpha-\Delta(u_1)+K_S(u_1).
	\end{align*}
	We continue this process:
	\[
		-\Delta(u_k)=\sum_{\alpha\in I}n_{\alpha,k+1}\cdot e_\alpha+Q_S(u_{k+1}).
	\]
	Since the differential $Q_S$ increases the polynomial degree at least by $\min\{d_i:i=1,\dots,k\}$, but $\Delta$ decreases the polynomial degree by 2, this process stops after finite steps, i.e., $\Delta(u_k)=0$ for some $k$. This implies that
	\[
		x=\sum_{\alpha\in I}n_\alpha\cdot e_\alpha+K_S(u)
	\]
	for some $n_\alpha\in\mathbb{C}$ and $u\in\mathcal{A}_{c_X}^{-1}$.
\end{proof}

Denote the isomorphism in Lemma \ref{bfQ_lemma} by $\mathbf{q}:\mathcal{A}_{c_X}^0/Q_S(\mathcal{A}_{c_X}^{-1})\stackrel{\sim}{\to}\mathcal{A}_{c_X}^0/(Q_S+\Delta)(\mathcal{A}_{c_X}^{-1})$.
\begin{theorem}\label{Phi_isom}
	There is an isomorphism
	\[
	 	\Phi:=\varphi\circ\mathbf{q}: H^0(\mathcal{A}_{c_X},Q_S) \mapto{\sim} H^{n-k}_{\pr}(X,\bC).
	\]
\end{theorem}
Note that $\varphi$ is defined in the level of cochains but $\mathbf{q}$ is defined in the level of cohomologies.
\begin{remark}\label{cqi}
It is well-known that there is an isomorphism between $H^0(\mathcal{A}_{c_X},Q_S)$ and $H^{n-k}_{\pr}(X,\bC)$ (for example, see Theorem 1 and relevant references in \cite{Dim95}). But we could not find a suitable literature which realizes that isomorphism explicitly; we provide them here, i.e., the maps $\varphi$ and $\mathbf{q}$. In fact, the middle bridge $\mathcal{A}_{c_X}^0/(Q_S+\Delta)(\mathcal{A}_{c_X}^{-1})$, the $0$-th cohomology of the dGBV algebra, will play a crucial role in our theory.
\end{remark}

\subsection{A quantization dGBV algebra} \label{sec2.7}
Proposition \ref{chc} and Lemma \ref{bfQ_lemma} suggest the following quantization complex\footnote{The quantization complex $\left(\mathcal{A}_{c_X}^\bullet(\!(\hbar)\!),Q_S+\hbar\Delta\right)$ can be viewed as a perturbative expansion of $(\mathcal{A}_{c_X},Q_S)$} $\left(\mathcal{A}_{c_X}^\bullet(\!(\hbar)\!),Q_S+\hbar\Delta\right)$ with the formal parameter $\hbar$ satisfying $\ch(\hbar)=1$, $\wt(\hbar)=1$, and $|\hbar|=0$: If $\hbar=0$, then we get the commutative differential graded algebra $(\mathcal{A}_{c_X}^\bullet,Q_S)$, and if $\hbar=1$, then we get the quantum complex $(\mathcal{A}_{c_X}^\bullet,Q_S+\Delta)$ whose 0-th cohomology is canonically isomorphic to $H_{dR}^{N-2}(\mathbf{P}^n\setminus X_{\underline{G}})\simeq H_{\mathrm{prim}}^{n-k}(X_{\underline{G}},\mathbb{C})$. Note that $\wt(Q_S)=1$ and $\wt(\Delta)=1$ imply that $\wt(Q_S+\hbar\Delta)=0$, and $\ch(Q_S)=\ch(\Delta)=0$ also implies that $\ch(Q_S+\hbar\Delta)=0$.\par
The complex $\left(\mathcal{A}_{c_X}(\!(\hbar)\!),Q_S+\hbar\Delta\right)$ is $\mathbb{Z}$-filtered:
\[
	\left(\mathcal{A}_{c_X}^\bullet (\!(\hbar)\!)^{(m)},Q_S+\hbar\Delta\right):=\left(\mathcal{A}_{c_X}^\bullet [\![\hbar]\!]\hbar^{-m},Q_S+\hbar\Delta\right)
\]
for $m\in\mathbb{Z}$. We define a cochain map $r^{(0)}$
\[
	\begin{tikzcd}
		\left(\mathcal{A}_{c_X}^\bullet[\![\hbar]\!],Q_S+\hbar\Delta\right)\arrow[r,"r^{(0)}"]& (\mathcal{A}_{c_X}^\bullet,Q_S)
	\end{tikzcd}
\]
by taking the quotient\footnote{Sending $\hbar=0$ is the classical limit of the quantization complex} by $\mathcal{A}_{c_X}^\bullet (\!(\hbar)\!)^{(-1)}$:
\[
	\begin{tikzcd}
		0 \arrow[r] & \mathcal{A}^\bullet_{c_X}(\!(\hbar)\!)^{(-1)}\arrow[r,hookrightarrow] & \mathcal{A}_{c_X}^\bullet (\!(\hbar)\!)^{(0)}\arrow[r,"r^{(0)}"] & \mathcal{A}_{c_X}^\bullet\arrow[r] &0.
	\end{tikzcd}
\]

\begin{definition}
	Let us define $\mathbb{H}_S$ and $\mathbb{H}_S^{(m)}$ as follows:
	\begin{align*}
		\mathbb{H}_S &:=\frac{\mathcal{A}_{c_X}^0(\!(\hbar)\!)}{(Q_S+\hbar\Delta)(\mathcal{A}_{c_X}^{-1}(\!(\hbar)\!))},\\
		\mathbb{H}_S^{(m)}&:=\frac{\mathcal{A}^0_{c_X}(\!(\hbar)\!)^{(m)}}{(Q_S+\hbar\Delta)(\mathcal{A}_{c_X}^{-1}(\!(\hbar)\!)^{(m)})}=\frac{\mathcal{A}^0_{c_X}[\![\hbar]\!]\hbar^{-m}}{(Q_S+\hbar\Delta)(\mathcal{A}_{c_X}^{-1}[\![\hbar]\!]\hbar^{-m})}
	\end{align*}
	for $m\in\mathbb{Z}$.
\end{definition}
We also have an exact sequence:
\[
	\begin{tikzcd}
		0 \arrow[r] & \mathbb{H}_S^{(-1)} \arrow[r] &\mathbb{H}_S^{(0)}\arrow[r,"r^{(0)}"] &\mathcal{A}_{c_X}^0/Q_S(\mathcal{A}_{c_X}^{-1})\arrow[r] &0.
	\end{tikzcd}
\]

\begin{proposition}
	The triple $\left(\mathcal{A}_{c_X}^\bullet (\!(\hbar)\!),Q_S+\hbar\Delta,\ell_2^\Delta(~,~)\right)$ is a shifted differential graded Lie algebra. Note that $\ell_2^\Delta=\frac{1}{\hbar}\ell_2^{Q_S+\hbar\Delta}$.
\end{proposition}

\begin{proposition}
	If the background charge $c_X$ is 0, i.e., $X_{\underline{G}}$ is Calabi-Yau, the quintuple $\left(\mathcal{A}_0^\bullet (\!(\hbar)\!),\cdot,Q_S+\hbar\Delta,\ell_2^\Delta,Q_S\right)$ is a dGBV algebra.
\end{proposition}

Let $\mathbb{H}:=H_{\mathrm{prim}}^{n-k}(X_{\underline{G}},\mathbb{C})$. From the previous discussion, we have $\mathbb{C}$-linear isomorphisms:
\[
	\begin{tikzcd}
		r^{(0)}\left(\mathbb{H}_S^{(0)}\right)\arrow[r,"\sim","\mathbf{q}"'] &\displaystyle\frac{\mathcal{A}_0^0}{(Q_S+\hbar\Delta)(\mathcal{A}_0^{-1})}\arrow[r,"\sim","\varphi"'] & \mathbb{H}.
	\end{tikzcd}
\]

\subsection{Deformation of dGBV algebra and the Gauss-Manin connection} \label{sec2.8}
From now on throughout the paper, 
\[
	\textrm{we will assume that $X_{\ud G}$ is Calabi-Yau,}
\]
 i.e., the background charge $c_X=0$.
\begin{definition}\label{T_H}
	Let $\mu:=\dim_{\mathbb{C}}\mathbb{H}=|I|$. Then,
	\[
		T_{\mathbb{H}}\stackrel{\mathrm{def}}{:=}\textrm{The space of formal tangent vector fields on }\mathbb{H}.
	\]
\end{definition}

Let us consider the deformation of dGBV algebra. 
If we let $\{t^\alpha:\alpha\in I\}$ to be the coordinate of the affine manifold $\mathbb{H}$, then $T_\mathbb{H}\simeq\mathbb{H}[\![\underline{t}]\!]$ in which we identify $\{\partial/\partial t^\alpha:\alpha\in I\}$ as the basis of $\mathbb{H}$, thus we have
\[
	\begin{tikzcd}
		T_{\mathbb{H}}\simeq\mathbb{H}[\![\underline{t}]\!] & \arrow[l,"\Phi","\sim"'] \left(\mathcal{A}_0^0/Q_S(\mathcal{A}_0^{-1})\right)[\![\underline{t}]\!]
	\end{tikzcd}
\]
where $\Phi$ in Theorem \ref{Phi_isom} is assumed to be $\mathbb{C}[\![\underline{t}]\!]$-linear. Note that $Q_S(\mathcal{A}_0^{-1})$ is the Jacobian ideal of the commutative ring $\mathcal{A}_0^0=\mathbb{C}[\underline{q}]_0$ generated by $\frac{\partial S}{\partial q_1},\dots\frac{\partial S}{\partial q_N}$. It suggests us to consider the following complex: 
\[
	\big(\mathcal{A}_0^\bullet[\![\underline{t}]\!](\!(\hbar)\!),Q_S+\hbar\Delta\big).
\]
where $Q_S$ and $\Delta$ are assumed to be $\mathbb{C}[\![\underline{t}]\!]$-linear. Clearly, we have the following proposition:
\begin{proposition}
	The quintuple $\left(\mathcal{A}_0^\bullet [\![\underline{t}]\!](\!(\hbar)\!),\cdot,Q_S+\hbar\Delta,\ell_2^\Delta,Q_S\right)$ is a dGBV algebra.
\end{proposition}
Now we deform differentials $Q_S+\hbar\Delta$ and $Q_S$ in order to make a deformed dGBV algebra. Choose an arbitrary $\Gamma\in \cA^0_0\otimes\mathbb{C}[\![\underline{t}]\!]$, then $\Gamma$ is a solution to the Maurer-Cartan equation for the DGLA $\left(\mathcal{A}_{0}^\bullet [\![\underline{t}]\!](\!(\hbar)\!),Q_S+\hbar\Delta,\ell_2^\Delta(~,~)\right)$, i.e.,
\[
	(Q_S+\hbar\Delta)(e^\Gamma-1):=(Q_S+\hbar\Delta)(\Gamma)+\frac{1}{2}\ell_2^{Q_S+\hbar\Delta}(\Gamma,\Gamma)=0.
\]
Define
\[
	Q_{S+\Gamma}:=\ell_2^{\Delta}(S+\Gamma,\,\cdot\,)=\sum_{i=1}^N\frac{\partial(S+\Gamma)}{\partial q_i}\frac{\partial}{\partial\eta_i}.
\]
Then, the operator $Q_{S+\Gamma}+\hbar\Delta$ is a $\mathbb{C}[\![\underline{t}]\!](\!(\hbar)\!)$-linear map on $\mathcal{A}^\bullet [\![\underline{t}]\!](\!(\hbar)\!)$ of degree 1, charge 0, weight 0, and satisfies $(Q_{S+\Gamma}+\hbar\Delta)^2=0$ (see Lemma 3.3, \cite{PP16}). In fact, a direct computation (the deformation formalism) implies the following proposition:
\begin{proposition}
	The quintuple $\left(\mathcal{A}_0^\bullet [\![\underline{t}]\!](\!(\hbar)\!),\cdot,Q_{S+\Gamma}+\hbar\Delta,\ell_2^\Delta,Q_{S+\Gamma}\right)$ is a dGBV algebra.
\end{proposition}

\begin{definition}\label{H-module}
	We define the 0-th cohomology $\mathcal{H}_{S+\Gamma}$ and a filtered piece $\mathcal{H}_{S+\Gamma}^{(m)}$ as follows:
	\[
		\mathcal{H}_{S+\Gamma} :=\frac{\mathcal{A}_0^0[\![\underline{t}]\!](\!(\hbar)\!)}{(Q_{S+\Gamma}+\hbar\Delta)(\mathcal{A}_0^{-1}[\![\underline{t}]\!](\!(\hbar)\!))},\quad\mathcal{H}_{S+\Gamma}^{(m)}:=\frac{\mathcal{A}_0^0[\![\underline{t}]\!][\![\hbar]\!]\hbar^{-m}}{(Q_{S+\Gamma}+\hbar\Delta)(\mathcal{A}_0^{-1}[\![\underline{t}]\!][\![\hbar]\!]\hbar^{-m})}
	\]
	for $m\in\mathbb{Z}$. We call this as the $\hbar$-filtration.
\end{definition}

\begin{lemma}\label{Basis_lemma_F}
	Let $\{u_\alpha:\alpha\in I\}$ be a $\mathbb{C}$-basis of $\mathcal{A}_0^0/Q_S(\mathcal{A}_0^{-1})$. Suppose that the set of polynomials $\{U_\alpha:U_\alpha\in\mathbb{C}[\![\underline{t}]\!][\![\hbar]\!],\alpha\in I\}$ satisfies the condtions $U_\alpha|_{\underline{t}=0,\hbar=0}=u_\alpha$ for all $\alpha\in I$, then $\{U_\alpha:U_\alpha\in\mathbb{C}[\![\underline{t}]\!][\![\hbar]\!],\alpha\in I\}$ is a $\mathbb{C}[\![\underline{t}]\!][\![\hbar]\!]$-basis of $\mathcal{H}_{S+\Gamma}^{(0)}$.
\end{lemma}

\begin{proof}
	First, we prove that $\{u_\alpha:\alpha\in I\}$ is a $\mathbb{C}[\![\hbar]\!]$-basis of $\mathcal{A}_0^0[\![\hbar]\!]/(Q_S+\hbar\Delta)(\mathcal{A}_0^{-1}[\![\hbar]\!])$.
	For all $\psi=\sum_{r\geq0}\psi_r\frac{\hbar^r}{r!}\in\mathcal{A}_0^0[\![\hbar]\!]$, consider the following equation:
	\[
		\sum_{r\geq0}\psi_r\frac{\hbar^r}{r!}=\sum_\rho\Big(\sum_{r=0}^\infty{m_{\psi_r}}^\rho\cdot\frac{\hbar^r}{r!}\Big)u_\rho+Q_S\Big(\sum_{r=0}^\infty\lambda_{\psi_r}\frac{\hbar^r}{r!}\Big)+\hbar\Delta\Big(\sum_{r=0}^\infty\lambda_{\psi_r}\frac{\hbar^r}{r!}\Big).
	\]
	By comparing the $\hbar$-power terms, we get
	\begin{align*}
		\psi_0&=\sum_\rho{m_{\psi_0}}^\rho u_\rho+Q_S(\lambda_{\psi_0}),\\
		\psi_1-\Delta(\lambda_{\psi_0})&=\sum_\rho{m_{\psi_1}}^\rho u_\rho+Q_S(\lambda_{\psi_1}),\\
		\frac{\psi_2}{2}-\Delta(\lambda_{\psi_1})&=\sum_\rho\frac{1}{2}{m_{\psi_2}}^\rho u_\rho+Q_S\Big(\frac{\lambda_{\psi_2}}{2}\Big),\\
		&\vdots\\
		\frac{\psi_r}{r!}-\Delta\Big(\frac{\lambda_{\psi_{r-1}}}{(r-1)!}\Big)&=\sum_\rho\frac{1}{r!}{m_{\psi_r}}^\rho u_\rho+Q_S\Big(\frac{\lambda_{\psi_r}}{r!}\Big).
	\end{align*}
	The coefficient $m_{\psi_{r}}{}^\rho$ is independent of the choice of $\lambda_{\psi_{r-1}}$ by Lemma \ref{Q-D_Lemma}. Therefore, this procedure uniquely determines $\frac{1}{r!}{m_{\psi_r}}^r$.\par
	For a proof that $\{U_\alpha:U_\alpha\in\mathbb{C}[\![\underline{t}]\!][\![\hbar]\!],\alpha\in I\}$ is a $\mathbb{C}[\![\underline{t}]\!][\![\hbar]\!]$-basis of $\mathcal{H}_{S+\Gamma}^{(0)}$, we use the same procedure; it requires to compare not only $\hbar$-power terms, but also $\underline{t}$-power terms, which we leave to the reader.
\end{proof}

\begin{definition}\label{GMC}
	We define \textit{the Gauss-Manin connection} on $\mathcal{H}_{S+\Gamma}$. We use the notation $[\cdot]$ to denote the cohomology class.
	\begin{enumerate}[(i)]
		\item For each $\alpha\in I$, we define a connection $\downtriangle_\alpha^{\frac{S+\Gamma}{\hbar}}:\mathcal{H}_{S+\Gamma}\to \mathcal{H}_{S+\Gamma}$ by
		\begin{align*}
			\downtriangle_\alpha^{\frac{S+\Gamma}{\hbar}}([w]):=&\left[e^{\tfrac{-(S+\Gamma)}{\hbar}}\frac{\partial}{\partial t^\alpha}\left(e^{\tfrac{S+\Gamma}{\hbar}}w\right)\right]\\
			=&\left[\frac{\partial}{\partial t^\alpha}w+\frac{1}{\hbar}\frac{\partial(S+\Gamma)}{\partial t^\alpha}w\right]
		\end{align*}
		for $w\in\mathcal{A}_0^0 [\![\underline{t}]\!](\!(\hbar)\!)$.
		\item We also define a connection $\downtriangle_{\hbar^{-1}}^{\frac{S+\Gamma}{\hbar }}:\mathcal{H}_{S+\Gamma}\to \mathcal{H}_{S+\Gamma}$ by
		\begin{align*}
			\downtriangle_{\hbar^{-1}}^{\frac{S+\Gamma}{\hbar}}([w]):=&\left[e^{\tfrac{-(S+\Gamma)}{\hbar}}\frac{\partial}{\partial \hbar^{-1}}\left(e^{\tfrac{S+\Gamma}{\hbar}}w\right)\right]\\
			=&\left[\frac{\partial}{\partial \hbar^{-1}}w+(S+\Gamma)w\right]
		\end{align*}
		for $w\in\mathcal{A}_0^0 [\![\underline{t}]\!](\!(\hbar)\!)$.
	\end{enumerate}
\end{definition}
These are well-defined on cohomology classes, since $Q_{S+\Gamma}+\hbar\Delta$ commutes with $\downtriangle_\alpha^{\frac{S+\Gamma}{\hbar}}$ and $\downtriangle_{\hbar^{-1}}^{\frac{S+\Gamma}{\hbar}}$. Note that 
\[
	\downtriangle_\alpha^{\frac{S+\Gamma}{\hbar}}\left(\mathcal{H}_{S+\Gamma}^{(m)}\right)\subseteq\mathcal{H}_{S+\Gamma}^{(m+1)}\quad\textrm{and}\quad\downtriangle_{\hbar^{-1}}^{\frac{S+\Gamma}{\hbar}}\left(\mathcal{H}_{S+\Gamma}^{(m)}\right)\subseteq\mathcal{H}_{S+\Gamma}^{(m)},\quad m\in\mathbb{Z}.
\]
Thus, $\hbar\downtriangle_\alpha^{\frac{S+\Gamma}{\hbar}}$ preserves the $\hbar$-filtration.
\section{\texorpdfstring{An algorithm for formal $F$-manifolds}{An algorithm for formal F-manifold}}\label{Sec_F}

\subsection{\texorpdfstring{Definition of formal $F$-manifolds}{Definition of formal F-manifolds}} \label{sec3.1}
In this subsection, we briefly review the definition of the formal $F$-manifold structure in \cite{HM}. We will only consider a pure even manifold for simplicity, so the signs disappear.
\begin{definition}\label{Def_F}
	Let $M$ be a complex manifold of finite dimension, whose holomorphic structure sheaf and holomorphic tangent sheaf are denoted by $\mathcal{O}_M$ and $\mathcal{T}_{M}$ respectively. Let $\circ$ be a $\mathcal{O}_{M}$-bilinear operation $\circ :\mathcal{T}_M\times\mathcal{T}_M \rightarrow\mathcal{T}_M$. Then $(M, \circ)$ is called an $F$-manifold if the following conditions are satisfied for all $X,Y,Z,W \in \mathcal{T}_M$:
	\begin{enumerate}[\normalfont(F1)]
		\item\label{F1}(Associativity) The multiplication $\circ$ is associative:
		\[
			(X \circ Y) \circ Z = X \circ (Y \circ Z),
		\]
		\item\label{F2}(Commutativity) The multiplication $\circ$ is commutative:
		\[
			X \circ Y = Y \circ X,
		\]
		\item\label{F3}(Potential)\footnote{We only consider a pure even manifold, so the equality \eqref{fmfd} is revised accordingly. See 1.definition in \cite{HM} for the equality with sign on a supermanifold.} The following equality holds:
		\begin{multline}\label{fmfd}
			[X \circ Y, Z \circ W] - [X \circ Y,Z] \circ W - Z \circ [X \circ Y,W] - X \circ [Y, Z \circ W] + X \circ [Y,Z] \circ W\\
			+ X \circ Z \circ [Y,W] - Y \circ [X, Z \circ W] + Y \circ [X,Z] \circ W + Y \circ Z \circ [X,W] = 0.
		\end{multline}
	\end{enumerate}
\end{definition}
One can similarly define its formal version: we consider the formal structure sheaf and the formal tangent bundle instead of the holomorphic structure sheaf and the holomorphic tangent bundle.\par
Let $t_M := \{t^\alpha\}$ be the formal coordinates on the open subset $U_i$. Choose sufficiently small open subset $U$ of $U_i$, then we could assume that $\mathcal{O}(U)=\mathbb{C}[\![t_M]\!]$. On the local coordinates $\{t_M,U\}$, the product $\partial_\alpha \circ\partial_\beta$ is written as a linear combination of $\{\partial_\alpha:=\partial/\partial{t^\alpha}\}$ with coefficients in $\mathbb{C}[\![t_M]\!]$. Let $A_{\alpha\beta}{}^\gamma \in\mathbb{C}[\![t_M]\!]$ be a formal power series representing the 3-tensor field such that
\[
	\partial_\alpha\circ \partial_\beta := \sum_\rho A_{\alpha\beta}{}^\rho\cdot\partial_\rho.
\]
Then, consider the following conditions:

\begin{enumerate}[\normalfont(C1)]
	\item\label{C1}(Associativity)
	\[
		\sum_\rho A_{\alpha\beta}{}^\rho A_{\rho\gamma}{}^\delta = \sum_\rho A_{\beta\gamma}{}^\rho A_{\rho\alpha}{}^\delta,
	\]
	\item\label{C2}(Commutativity)
	\[
		A_{\beta\alpha}{}^\gamma = A_{\alpha\beta}{}^\gamma,
	\]
	\item\label{C3}(Potential)
	\[
		\partial_\alpha A_{\beta\gamma}{}^\delta = \partial_\beta A_{\alpha\gamma}{}^\delta.
	\]
\end{enumerate}

Recall that $T_M$ is the space of formal tangent vector fields on a complex manifold $M$.

\begin{proposition}\label{Coor_F}
	Let $\circ:T_M\times T_M\to T_M$ be a $\mathcal{O}_M$-bilinear operator. Assume that $(M, \circ)$ satisfies the conditions \ref{C1}, \ref{C2}, and \ref{C3} in local coordinates. Then $(M,\circ)$ is a formal $F$-manifold.
\end{proposition}
\begin{proof}
	It is sufficient to check for commuting vector fields $(X,Y,Z,W) = (\partial_\alpha, \partial_\beta, \partial_\gamma, \partial_\delta).$ Then the left hand side of \ref{F3} becomes
	\begin{align*}
		&\Big[\sum_\phi A_{\alpha\beta}{}^\phi \partial_\phi, \sum_\rho A_{\gamma\delta}{}^\rho \partial_\rho\Big] - \Big[\sum A_{\alpha\beta}{}^\phi \partial_\phi, \partial_\gamma\Big] \circ \partial_\delta - \partial_\gamma \circ \Big[\sum_\phi A_{\alpha\beta}{}^\phi \partial_\phi, \partial_\delta\Big]\\
		& - \partial_\alpha \circ [\partial_\beta, \sum A_{\gamma\delta}{}^\rho \partial_\rho] - \partial_\beta \circ [\partial_\alpha, \sum A_{\gamma\delta}{}^\rho \partial_\rho]\\
		=&\sum_{\phi,\rho} (A_{\alpha\beta}{}^\phi (\partial_\phi A_{\gamma\delta}{}^\rho) \partial_\rho - A_{\gamma\delta}{}^\rho (\partial_\rho A_{\alpha\beta}{}^\phi)\partial_\phi) + \sum_{\phi,\rho} (\partial_\gamma A_{\alpha\beta}{}^\phi) A_{\phi\delta}{}^\rho \partial_\rho+ \sum_{\phi,\rho} (\partial_\delta A_{\alpha\beta}{}^\phi) A_{\gamma\phi}{}^\rho \partial_\rho\\
		& - \sum_{\phi,\rho} (\partial_\beta A_{\gamma\delta}{}^\rho)A_{\alpha\rho}{}^\phi \partial_\phi - \sum_{\phi,\rho} (\partial_\alpha A_{\gamma\delta}{}^\rho) A_{\beta\rho}{}^\phi \partial_\phi\\
		=&\sum_{\phi,\rho}(\partial_\gamma (A_{\alpha\beta}{}^\phi A_{\phi\delta}{}^\rho) - \partial_\gamma A_{\beta\delta}{}^\phi A_{\alpha\phi}{}^\rho - \partial_\alpha(A_{\gamma\delta}{}^\phi A_{\beta\phi}{}^\rho) + \partial_\alpha A_{\delta\beta}{}^\phi A_{\gamma\phi}{}^\rho)\partial_\rho\\
		=&\sum_{\phi,\rho}(A_{\beta\delta}{}^\phi \partial_\gamma A_{\alpha\phi}{}^\rho - A_{\delta\beta}{}^\phi \partial_\alpha A_{\gamma\phi}{}^\rho)\partial_\rho\\
		=&0,
	\end{align*}
	where the second equality holds by potential \ref{C3}, the third equality holds by associativity \ref{C1}, and the last equality holds by commutativity \ref{C2} and potential \ref{C3}. Therefore the three conditions \ref{C1}, \ref{C2}, and \ref{C3} induce \ref{F3}.
\end{proof}

\subsection{A connection matrix for the Gauss-Manin connection}\label{sec3.2}
In order to construct a formal $F$-manifold structure on $\mathbb{H}=H_{\mathrm{prim}}^{n-k}(X,\mathbb{C})$, we will calculate a connection matrix for the Gauss-Manin connection. Choose a solution $\Gamma\in\mathcal{A}_0^0[\![\underline{t}]\!]$ of the Maurer-Cartan equation, and suppose that $\partial_\alpha\Gamma|_{\underline{t}=0}=u_\alpha$ for all $\alpha\in I$, where $\{u_\alpha:\alpha\in I\}$ is a $\mathbb{C}$-basis of $\mathcal{A}_0^0/Q_S(\mathcal{A}_0^{-1})$. Recall that $\Gamma$ gives us the 0-th cohomology $\mathcal{H}_{S+\Gamma}$ and a filtered piece $\mathcal{H}_{S+\Gamma}^{(m)}$ such that
\[
		\mathcal{H}_{S+\Gamma} :=\frac{\mathcal{A}_0^0[\![\underline{t}]\!](\!(\hbar)\!)}{(Q_{S+\Gamma}+\hbar\Delta)(\mathcal{A}_0^{-1}[\![\underline{t}]\!](\!(\hbar)\!))},\quad\mathcal{H}_{S+\Gamma}^{(m)}:=\frac{\mathcal{A}_0^0[\![\underline{t}]\!][\![\hbar]\!]\hbar^{-m}}{(Q_{S+\Gamma}+\hbar\Delta)(\mathcal{A}_0^{-1}[\![\underline{t}]\!][\![\hbar]\!]\hbar^{-m})}
\]
for $m\in\mathbb{Z}$. Since $\partial_\alpha\Gamma=\hbar\downtriangle_\alpha^{\frac{S+\Gamma}{\hbar}}1$ is a $\mathbb{C}[\![\underline{t}]\!][\![\hbar]\!]$-basis of $\mathcal{H}_{S+\Gamma}^{(0)}$ by Lemma \ref{Basis_lemma_F}, there is a connection matrix $\mathbf{A}_{\alpha\beta}{}^\rho\in\mathbb{C}[\![\underline{t}]\!][\![\hbar]\!]$ with respect to the basis $\{\hbar\downtriangle_\rho^{\frac{S+\Gamma}{\hbar}}1:\rho\in I\}$ such that
\[
	\hbar\downtriangle_\beta^{\frac{S+\Gamma}{\hbar}}(\hbar\downtriangle_\alpha^{\frac{S+\Gamma}{\hbar}}1)=\sum_{\rho\in I}\mathbf{A}_{\alpha\beta}{}^\rho\cdot(\hbar\downtriangle_\rho^{\frac{S+\Gamma}{\hbar}}1)+(Q_{S+\Gamma}+\hbar\Delta)(\mathbf{\Lambda}_{\alpha\beta})
\]
for some $\mathbf{\Lambda}_{\alpha\beta}\in\mathcal{A}_{0}^{-1}[\![\underline{t}]\!][\![\hbar]\!]$ and all $\alpha,\beta\in I$. Suppose that there is $\Gamma$ which makes $\mathbf{A}_{\alpha\beta}{}^\rho$ and $\mathbf{\Lambda}_{\alpha\beta}$ have no $\hbar$-power terms, i.e., there is a connection matrix $A_{\alpha\beta}{}^\rho\in\mathbb{C}[\![\underline{t}]\!]$ such that
\begin{equation}\label{F-QM}
	\hbar\downtriangle_\beta^{\frac{S+\Gamma}{\hbar}}(\hbar\downtriangle_\alpha^{\frac{S+\Gamma}{\hbar}}1)=\sum_{\rho\in I}A_{\alpha\beta}{}^\rho\cdot(\hbar\downtriangle_\rho^{\frac{S+\Gamma}{\hbar}}1)+(Q_{S+\Gamma}+\hbar\Delta)(\Lambda_{\alpha\beta}),
\end{equation}
for some $\Lambda_{\alpha\beta}\in\mathcal{A}_{0}^{-1}[\![\underline{t}]\!]$ and all $\alpha,\beta\in I$. Then, by comparing $\hbar$-power terms, the equation \eqref{F-QM} reduces further to the following:
\begin{equation}\label{F-QM2}
	\begin{aligned}
		\partial_\alpha\Gamma\cdot\partial_\beta\Gamma&=\sum_{\rho\in I}A_{\alpha\beta}{}^\rho\cdot\partial_\rho\Gamma+Q_{S+\Gamma}(\Lambda_{\alpha\beta}),\\
		\partial_\beta\partial_\alpha\Gamma&=\Delta(\Lambda_{\alpha\beta}),
	\end{aligned}
\end{equation}
for all $\alpha,\beta\in I$. We claim that the coefficients $A_{\alpha\beta}{}^\rho$ give a formal $F$-manifold structure on $\mathbb{H}$. For simplicity, we use the notation $\Gamma_{\alpha_1\alpha_2\cdots\alpha_\ell}:=\partial_{\alpha_\ell}\cdots\partial_{\alpha_2}\partial_{\alpha_1}\Gamma$.

\begin{remark}

	\begin{enumerate}
	\item The equation \eqref{F-QM2} is a special case of the equation \eqref{WPF} for weak primitive forms as summarized in Table \ref{Comparison_Table}.
	\item The equation which appeared in the proof of Lemma 7.1 in \cite{BK} is the equation \eqref{F-QM2} with $\alpha=\beta$.
	\end{enumerate}
\end{remark}

\begin{proposition}
	Assume that $\Gamma,A_{\alpha\beta}{}^\rho$, and $\Lambda_{\alpha\beta}$ satisfy the equation \eqref{F-QM}. Then $A_{\alpha\beta}{}^\rho$ satisfies \textrm{\normalfont(associativity)}, \textrm{\normalfont(commutativity)}, and \textrm{\normalfont(potential)}.
	In particular, $A_{\alpha\beta}{}^\rho$ equip $\mathbb{H}=H_{\mathrm{prim}}^{n-k}(X,\mathbb{C})$ with a formal $F$-manifold structure by Proposition \ref{Coor_F}.
\end{proposition}

\begin{proof}
	By applying $\hbar\downtriangle_\gamma^{\frac{S+\Gamma}{\hbar}}$ to the equation \eqref{F-QM}, we get
	\begin{equation}\label{hbar3_Ex}
		\begin{aligned}
			&\hbar\downtriangle^{\frac{S+\Gamma}{\hbar}}_\gamma(\hbar\downtriangle^{\frac{S+\Gamma}{\hbar}}_\beta(\hbar\downtriangle^{\frac{S+\Gamma}{\hbar}}_\alpha1))=\hbar\downtriangle_\gamma^{\frac{S+\Gamma}{\hbar}}\sum_{\rho\in I}A_{\alpha\beta}{}^\rho\Gamma_\rho+(Q_{S+\Gamma}+\hbar\Delta)(\hbar\downtriangle^{\frac{S+\Gamma}{\hbar}}_\gamma\Lambda_{\alpha\beta})\\
			=&\sum_{\rho\in I}\Big[\hbar\partial_\gamma A_{\alpha\beta}{}^\rho\Gamma_\rho+A_{\alpha\beta}{}^\rho\hbar\downtriangle^{\frac{S+\Gamma}{\hbar}}_\gamma\Gamma_\rho\Big]+(Q_{S+\Gamma}+\hbar\Delta)(\hbar\downtriangle^{\frac{S+\Gamma}{\hbar}}_\gamma\Lambda_{\alpha\beta})\\
			=&\sum_{\delta\in I}\bigg[\hbar\partial_\gamma A_{\alpha\beta}{}^\delta+\sum_{\rho\in I}A_{\alpha\beta}{}^\rho A_{\gamma\rho}{}^\delta\bigg]\Gamma_\delta+(Q_{S+\Gamma}+\hbar\Delta)\Big(\hbar\downtriangle_\gamma^{\frac{S+\Gamma}{\hbar}}\Lambda_{\alpha\beta}+\sum_{\rho\in I}A_{\alpha\beta}{}^\rho\Lambda_{\gamma\rho}\Big).
		\end{aligned}	
	\end{equation}
	Note that
	\[
		\hbar\downtriangle_\gamma^{\frac{S+\Gamma}{\hbar}}(\hbar\downtriangle_\beta^{\frac{S+\Gamma}{\hbar}}(\hbar\downtriangle_\alpha^{\frac{S+\Gamma}{\hbar}}1))=\hbar^2\Gamma_{\alpha\beta\gamma}+\hbar(\Gamma_\alpha\Gamma_{\beta\gamma}+\Gamma_{\beta}\Gamma_{\alpha\gamma}+\Gamma_\gamma\Gamma_{\alpha\beta})+\Gamma_\alpha\Gamma_\beta\Gamma_\gamma.
	\]
	Since the left hand side of equation \eqref{hbar3_Ex} is invariant under the permutation of $\alpha,\beta,\gamma$ at each $\hbar$-power term and the coefficient of $\Gamma_\delta$ in the right hand side is unique by Lemma \ref{Basis_lemma_F},
	\[
		\sum_{\rho}A_{\alpha\beta}{}^\rho A_{\gamma\rho}{}^\delta\quad\textrm{and}\quad\partial_{\alpha}A_{\beta\gamma}{}^\rho
	\]
	are also invariant under the permutation of $\alpha,\beta,\gamma$. Therefore, the associativity and the potential are satisfied. It is important that $A_{\alpha\beta}{}^\rho$ has no $\hbar$-power terms.\par
	The commutativity follows from the equation \eqref{F-QM} directly by the same argument.
\end{proof}
Therefore, the key question to construct a formal $F$-manifold structure on $\mathbb{H}=H_{\mathrm{prim}}^{n-k}(X,\mathbb{C})$ is the following:
\begin{quote}
	Can we find a Maurer-Cartan solution $\Gamma$ whose connection matrix $A_{\alpha\beta}{}^\rho$ (and $\Lambda_{\alpha\beta}$) has no $\hbar$-power terms?
\end{quote}

\subsection{An explicit algorithm}\label{Subsec_Alg}
We will provide an explicit algorithm which calculates $\Gamma, A_{\alpha\beta}{}^\rho$ satisfying the equation \eqref{F-QM}. We use the following notations:
\begin{equation}\label{Exp_notation}
	\begin{aligned}
		\Gamma&=\sum_{\alpha\in I}u_\alpha\cdot t^\alpha+\sum_{m\geq2}\sum_{\underline{\alpha}\in I^m}\frac{1}{m!}u_{\underline{\alpha}}t^{\underline{\alpha}}\in\mathcal{A}_0^0 [\![\underline{t}]\!],\\
		A_{\alpha\beta}{}^\rho&=a_{\alpha\beta}{}^\rho+\sum_{m\geq1}\sum_{\underline{\alpha}\in I^m}\frac{1}{m!}a_{\alpha\beta\underline{\alpha}}{}^\rho t^{\underline{\alpha}}\in\mathbb{C}[\![\underline{t}]\!],\\
		\Lambda_{\alpha\beta}&=\lambda_{\alpha\beta}+\sum_{m\geq1}\sum_{\underline{\alpha}\in I^m}\frac{1}{m!}\lambda_{\alpha\beta\underline{\alpha}}t^{\underline{\alpha}}\in\mathcal{A}_0^{-1} [\![\underline{t}]\!],
	\end{aligned}
\end{equation}
where $u_{\underline{\alpha}}\in\mathcal{A}_0^0$, $a_{\alpha\beta}{}^\rho,a_{\alpha\beta\underline{\alpha}}{}^\rho\in\mathbb{C}$, $\lambda_{\alpha\beta},\lambda_{\alpha\beta\underline{\alpha}}\in\mathcal{A}_0^{-1}$, and $t^{\alpha_1\cdots\alpha_m}:=t^{\alpha_1}\cdots t^{\alpha_m}$. We are assuming that $u_{\underline{\alpha}},a_{\alpha\beta\underline{\alpha}}{}^\rho,\lambda_{\alpha\beta\underline{\alpha}}$ are invariant under the permutation of indices of $\underline{\alpha}$ under our notation. This assumption is important for our algorithm (see Remark \ref{nontrivialE}). In order to make $\{\Gamma_\rho:\rho\in I\}$ be a $\mathbb{C}[\![\underline{t}]\!][\![\hbar]\!]$-basis of $\mathcal{H}_{S+\Gamma}^{(0)}$, we define the $t^\alpha$-term of $\Gamma$ as $u_\alpha$  where $\{u_\alpha:\alpha\in I\}$ is a $\mathbb{C}$-basis of $\mathcal{A}_0^0/Q_S(\mathcal{A}_0^{-1})$; see Lemma \ref{Basis_lemma_F}. We have to determine $u_{\underline{\alpha}},a_{\alpha\beta}{}^\rho,a_{\alpha\beta\underline{\alpha}}{}^\rho$ which satisfy the equation \eqref{F-QM}.\par 
For the algorithm, Definition \ref{u_alpha} and Lemma \ref{Q-D_Lemma} (the weak $Q_S\Delta$-lemma) are required.
%

\begin{definition}\label{u_alpha}
	Let $|\underline{\alpha}|=m$ which means that $\underline{\alpha}\in I^m$. We define the notation $u_{\underline{\alpha}}^{(i)}$ as follows:
	\[
		u^{(i)}_{\underline{\alpha}}=\sum_{\substack{\underline{\alpha}_1\sqcup\cdots\sqcup \underline{\alpha}_{m-i}=\underline{\alpha}\\\underline{\alpha}_j\neq\emptyset}}\frac{1}{(m-i)!}u_{\underline{\alpha}_1}\cdots u_{\underline{\alpha}_{m-i}},\quad(0\leq i\leq m-1),	
	\]
	where the notation $\underline{\alpha}_1\sqcup\cdots\sqcup \underline{\alpha}_{m-i}=\underline{\alpha}$ means $\underline{\alpha}_1\cup\cdots\cup \underline{\alpha}_{m-i}=\underline{\alpha}$ and $\underline{\alpha}_k\cap \underline{\alpha}_\ell=\emptyset$ for $\ell\neq k$.
\end{definition}
Note that $u_{\underline{\alpha}}^{(i)}$ is invariant under the permutation of indices of $\underline{\alpha}$. For example,
\[
	\begin{aligned}
		u_{\alpha\beta\gamma}^{(0)}&=u_\alpha u_\beta u_\gamma,\\
		u_{\alpha\beta\gamma}^{(1)}&=u_\alpha u_{\beta\gamma}+ u_{\beta} u_{\alpha\gamma}+ u_\gamma u_{\alpha\beta}\\
		u_{\alpha\beta\gamma}^{(2)}&=u_{\alpha\beta\gamma},
	\end{aligned}
	\quad\textrm{and}\quad
	\begin{aligned}
		u_{\alpha\beta\gamma\delta}^{(0)}&=u_\alpha u_\beta u_\gamma u_\delta,\\
		u_{\alpha\beta\gamma\delta}^{(1)}&=u_\alpha u_{\beta\gamma\delta}+ u_{\beta} u_{\alpha\gamma\delta}+ u_\gamma u_{\alpha\beta\delta}+ u_{\delta} u_{\alpha\beta\gamma},\\
		u_{\alpha\beta\gamma\delta}^{(2)}&= u_{\alpha\beta} u_{\gamma\delta}+ u_{\alpha\gamma} u_{\beta\delta}+ u_{\alpha\delta} u_{\beta\gamma},\\
		u_{\alpha\beta\gamma\delta}^{(3)}&=u_{\alpha\beta\gamma\delta}.
	\end{aligned}
\]\par
Now we give the algorithm for $\Gamma$ and $A_{\alpha\beta}{}^\rho$ as follows:
\begin{description}
	\item[Step-0] Choose a $\mathbb{C}$-basis $\{u_\alpha:\alpha\in I\}$ of $J_S:=\mathcal{A}_0^0/Q_S(\mathcal{A}_0^{-1})$. Note that $J_S$ is a finite dimensional $\mathbb{C}$-vector space, and there is a partition of $\{u_\alpha:\alpha\in I\}$ in terms of weights, i.e., we decompose $\{u_\alpha:\alpha\in I\}=\bigcup_{r=0}^{n-k}\{u_{\alpha}:\alpha\in I_r\}$ where $I=I_0\sqcup\cdots\sqcup I_{n-k}$. We use the notations $u_{\mathrm{min}}=u_{\alpha_1}$ and $u_{\mathrm{max}}=u_{\alpha_2}$ for $\alpha_1\in I_{0}$ and $\alpha_2\in I_{n-k}$ where $|I_0|=|I_{n-k}|=1$ (since $X_{\ud G}$ is a Calabi-Yau manifold).
	\item[Step-1] Determine $a_{\alpha\beta}^{(0)}$ and $\lambda_{\alpha\beta}^{(0)}$ using a basis $\{u_\alpha:\alpha\in I\}$ of $J_S$ as follows:
	\[
		u_\alpha u_\beta=\sum_{\rho\in I}a_{\alpha\beta}^{(0)}{}^\rho u_\rho+Q_S({\lambda_{\alpha\beta}^{(0)}}).
	\]
	Note that $a_{\alpha\beta}^{(0)}{}^\rho\in\mathbb{C}$ is unique and $\lambda_{\alpha\beta}^{(0)}$ is unique up to $\ker Q_S$. Therefore, $\Delta(\lambda_{\alpha\beta}^{(0)})$ is unique up to $\mathrm{Im}(Q_S)$ by the weak $Q_S\Delta$-lemma \ref{Q-D_Lemma}. Then, define $u_{\alpha\beta}:=\Delta(\lambda_{\alpha\beta}^{(0)})$ and $a_{\alpha\beta}{}^\rho:=a_{\alpha\beta}^{(0)}{}^\rho$. 
	\item[Step-$\boldsymbol\ell$ $(\ell\geq2)$] Suppose that $|\underline{\alpha}|=\ell+1$. Determine $a_{\underline{\alpha}}^{(i)}{}^\rho$ and $\lambda_{\underline{\alpha}}^{(i)}$ ($0\leq i\leq \ell-1$) in sequence as follows:
	\begin{equation}\label{ind_u}
			\begin{aligned}
				u_{\underline{\alpha}}^{(0)}&=\sum_{\rho\in I}a_{\underline{\alpha}}^{(0)}{}^\rho u_\rho+Q_S(\lambda_{\underline{\alpha}}^{(0)}),\\
				u_{\underline{\alpha}}^{(1)}-\Delta(\lambda_{\underline{\alpha}}^{(0)})&=\sum_{\rho\in I}a_{\underline{\alpha}}^{(1)}{}^\rho u_\rho+Q_S(\lambda_{\underline{\alpha}}^{(1)}),\\
				&\quad\vdots\\
				u_{\underline{\alpha}}^{(\ell-1)}-\Delta(\lambda_{\underline{\alpha}}^{(\ell-2)})&=\sum_{\rho\in I}a_{\underline{\alpha}}^{(\ell-1)}{}^\rho u_\rho+Q_S(\lambda_{\underline{\alpha}}^{(\ell-1)}).
		\end{aligned}
	\end{equation}
	Since $\lambda_{\underline{\alpha}}^{(i)}$ is unique up to $\ker Q_S$ for $0\leq i\leq\ell-2$, $\Delta(\lambda_{\underline{\alpha}}^{(i)})$ is unique up to $\mathrm{Im}(Q_S)$ by the weak $Q_S\Delta$-lemma \ref{Q-D_Lemma}. Therefore, $a_{\underline{\alpha}}^{(\ell-1)}{}^\rho$ is independent of the choices of $\lambda_{\underline{\alpha}}^{(i)}$. Then, define $a_{\underline{\alpha}}{}^\rho:=a_{\underline{\alpha}}^{(\ell-1)}{}^\rho$ and $u_{\underline{\alpha}}:=\Delta(\lambda_{\underline{\alpha}}^{(\ell-1)})$.
\end{description}
By this inductive algorithm, we can completely determine $\Gamma, A_{\alpha\beta}{}^\rho$, and $\Lambda_{\a\b}$ which turn out to satisfy (see subsection \ref{sec3.4}) the equations \eqref{F-QM2}.

\begin{corollary}\label{Cor_Alg}
	Suppose that $\Gamma,\Gamma'\in\mathcal{A}_0^0	[\![\underline{t}]\!]$ satisfy the equation \eqref{F-QM} with the connection matrices $A_{\alpha\beta}{}^\rho$ and $A'_{\alpha\beta}{}^\rho$ respectively. If $\Gamma_\alpha|_{\underline{t}=0}=\Gamma'_\alpha|_{\underline{t}=0}=u_\alpha$ for some $\mathbb{C}$-basis $\{u_\alpha:\alpha\in I\}$ of $\mathcal{A}_0^{0}/Q_S(\mathcal{A}_0^{-1})$, then $A_{\alpha\beta}{}^\rho=A'_{\alpha\beta}{}^\rho$. In other words, the result $A_{\alpha\beta}{}^\rho$ of the above algorithm depends only on the choice of $\mathbb{C}$-basis $\{u_\alpha:\alpha\in I\}$ of $\mathcal{A}_0^{0}/Q_S(\mathcal{A}_0^{-1})$.
\end{corollary}

\begin{remark}[Non-triviality of solving the differential equations \eqref{F-QM2}]\label{nontrivialE}
	The symmetries among the indices of $\underline{\alpha}$ in $u_{\underline{\alpha}},a_{\alpha\beta\underline{\alpha}}{}^\rho,\lambda_{\alpha\beta\underline{\alpha}}$ play an important role. If one tries to solve the equations in a naive way by comparing the coefficients of the $\ud t$-powers, then one would get into trouble that the symmetries among the indices is not guaranteed. But our algorithm guarantees that the relevant quantities such as 
	$u_{\underline{\alpha}},a_{\alpha\beta\underline{\alpha}}{}^\rho,\lambda_{\alpha\beta\underline{\alpha}}$ are invariant under the permutation of $\ud \a$ as the equations \eqref{ind_u} indicate. Suppose that there is no condition such as $u_{\alpha\beta}=u_{\beta\alpha}$, i.e., we do not put symmetry restrictions on the index notation. Then, the equation
	\[
		\partial_\beta\partial_\alpha\Gamma=\Delta(\Lambda_{\alpha\beta})
	\]
	in the equations \eqref{F-QM2} gives the following equation
	\[
		u_{\alpha\beta}+u_{\beta\alpha}=2\Delta(\lambda_{\alpha\beta}),
	\]
	but then it cannot determine the values of $u_{\alpha\beta},u_{\beta\alpha}$ individually from the data $2\Delta(\lambda_{\alpha\beta})$ of the previous step. This implies that the naive way of solving  \eqref{F-QM2} inductively with non-symmetric notations does not work well, either.
\end{remark}

\begin{remark}[Difficulty of constructing an algorithm for a flat metric]
	For a formal Frobenius manifold structure (see Definition \ref{Def_Frob}), we need to define a flat metric $g$ on $T_\mathbb{H}$ (i.e., $\partial_\gamma g(\partial_\alpha,\partial_\beta)=0$ for all $\alpha,\beta,\gamma\in I$), which is compatible with the formal $F$-manifold structure. A natural candidate of a metric would be $g(\partial_\alpha,\partial_\beta)=A_{\alpha\beta}{}^{\max}$ using the weight filtration (the Hodge filtration), since such a choice corresponds to a cup product on $\bH$ and satisfies the compatibility condition (Invariance) in Definition \ref{Def_Frob}. To make $g$ be flat, we have to find $A_{\alpha\beta}{}^\rho$ satisfying the equation \eqref{F-QM} with $a_{\underline{\alpha}}{}^\mathrm{max}=0$ for all $|\underline{\alpha}|\geq3$. However, our algorithm for solving the equation \eqref{F-QM} is not refined enough to satisfy the conditions $a_{\underline{\alpha}}{}^\mathrm{max}=0$ for all $|\underline{\alpha}|\geq3$. For example, if we run the algorithm, we get non-zero $u_{\a\b}$ and the step 3 in our algorithm gives the following equations:
	\begin{align*}
		u_\alpha u_\beta u_\gamma&=\sum_{\rho\in I}a_{\alpha\beta\gamma}^{(0)}{}^\rho u_\rho+Q_S(\lambda_{\alpha\beta\gamma}^{(0)}),\\
		u_\alpha u_{\beta\gamma}+ u_{\beta} u_{\alpha\gamma}+ u_\gamma u_{\alpha\beta}-\Delta(\lambda_{\alpha\beta\gamma}^{(0)})&=\sum_{\rho\in I}a_{\alpha\beta\gamma}{}^\rho u_\rho+Q_S(\lambda_{\alpha\beta\gamma}^{(1)}).
	\end{align*}
	It is hard to choose $\mathbb{C}$-basis $\{u_\alpha:\alpha\in I\}$ satisfying $a_{\alpha\beta\gamma}{}^\mathrm{max}=0$. Note that Corollary \ref{Cor_Alg} implies that $a_{\alpha\beta\gamma}{}^\mathrm{max}$ only depends on the choice of $\{u_\alpha:\alpha\in I\}$. On the other hand, in section \ref{Sec_Frob}, we will use a Maurer-Cartan solution $L=\sum_{\ud \a} t^\a u_\a$, i.e., $\G$ with the condition $u_{\ud \a}$=0 for $|\ud \a|\geq 2$, and solve a differential equation \eqref{Fdeq} instead of \eqref{F-QM2} in order to construct a formal Frobenius manifold.
Note that both \eqref{Fdeq} and \eqref{F-QM2} are special cases of the differential equation \eqref{WPF} for weak primitive forms.
\end{remark}

\subsection{A proof why the algorithm works}\label{sec3.4}
In this subsection, we give a proof of the algorithm in subsection \ref{Subsec_Alg}.
\begin{definition}\label{Def_UBC}
	We define quantities $\mathbf{U},\mathbf{B},\mathbf{C}$, which are our key players in the proof of the algorithm, as follows:
	\[
		\mathbf{U}_{\alpha_1\cdots\alpha_m}:=\hbar\downtriangle_{\alpha_m}^{\frac{S+\Gamma}{\hbar}}\cdots\hbar\downtriangle_{\alpha_1}^{\frac{S+\Gamma}{\hbar}}1,
	\]
	\[
		\left\{\begin{aligned}
			{\mathbf{B}_{\alpha_1\alpha_2}}^\rho:=&{A_{\alpha_1\alpha_2}}^\rho,\\
			{\mathbf{B}_{\alpha_1\alpha_2\alpha_3}}^\rho:=&\sum_{\delta\in I}{\mathbf{B}_{\alpha_1\alpha_2}}^\delta{\mathbf{B}_{\delta\alpha_3}}^\rho+\hbar\cdot\partial_{\alpha_3}({\mathbf{B}_{\alpha_1\alpha_2}}^\rho),\\
			&\vdots\\
			{\mathbf{B}_{\alpha_1\cdots\alpha_m}}^\rho:=&\sum_{\delta\in I}{\mathbf{B}_{\alpha_1\cdots\alpha_{m-1}}}^\delta{\mathbf{B}_{\delta\alpha_m}}^\rho+\hbar\cdot\partial_{\alpha_m}({\mathbf{B}_{\alpha_1\cdots\alpha_{m-1}}}^\rho),
		\end{aligned}\right.
	\]
	\[
		\left\{\begin{aligned}
			\mathbf{C}_{\alpha_1\alpha_2}:=&\Lambda_{\alpha_1\alpha_2},\\
			\mathbf{C}_{\alpha_1\alpha_2\alpha_3}:=&\sum_{\delta\in I}{\mathbf{B}_{\alpha_1\alpha_2}}^\delta\mathbf{C}_{\delta\alpha_3}+\hbar\downtriangle_{\alpha_3}^{\frac{S+\Gamma}{\hbar}}\mathbf{C}_{\alpha_1\alpha_2},\\
			&\vdots\\
			\mathbf{C}_{\alpha_1\cdots\alpha_m}:=&\sum_{\delta\in I}{\mathbf{B}_{\alpha_1\cdots\alpha_{m-1}}}^\delta\mathbf{C}_{\delta\alpha_m}+\hbar\downtriangle_{\alpha_m}^{\frac{S+\Gamma}{\hbar}}\mathbf{C}_{\alpha_1\cdots\alpha_{m-1}}.
		\end{aligned}\right.
	\]
\end{definition}

\begin{theorem}\label{Thm_UBD}
	The equations \eqref{F-QM} hold for all $\alpha,\beta\in I$ if and only if the following inductive equations hold for all $m\geq2$ and all $\underline{\alpha}\in I^m$:
	\begin{equation}\label{Ind_QM}
		\mathbf{U}_{\underline{\alpha}}=\sum_{\rho\in I}{\mathbf{B}_{\underline{\alpha}}}^\rho\Gamma_\rho+(Q_{S+\Gamma}+\hbar\Delta)(\mathbf{C}_{\underline{\alpha}}).
	\end{equation}
\end{theorem}

\begin{proof}
	It is trivial that the above inductive equations imply the equation \eqref{F-QM} because the $m=2$ case of \eqref{Ind_QM} is same as \eqref{F-QM}.\par
	Suppose that the equation \eqref{F-QM} holds for all $\alpha, \beta\in I$. Therefore, the inductive equations hold for all $\underline{\alpha}\in I^2$. Assume that the inductive equations hold for all $\underline{\alpha}\in I^\ell$, i.e.,
	\begin{equation}\label{ind_ell}
		\mathbf{U}_{\alpha_1\cdots\alpha_\ell}=\sum_{\rho\in I}{\mathbf{B}_{\alpha_1\cdots\alpha_\ell}}^\rho\Gamma_\rho+(Q_{S+\Gamma}+\hbar\Delta)(\mathbf{C}_{\alpha_1\cdots\alpha_\ell}),
	\end{equation}
	for all $(\alpha_1,\dots,\alpha_\ell)\in I^\ell$. Take $\hbar\downtriangle_{\alpha_{\ell+1}}^{\frac{S+\Gamma}{\hbar}}$ on the equation \eqref{ind_ell}. Then we have the following equations:
	\begin{align*}
		\hbar\downtriangle_{\alpha_{\ell+1}}^{\frac{S+\Gamma}{\hbar}}\mathbf{U}_{\alpha_1\cdots\alpha_\ell}=&\sum_{\rho\in I}\hbar\downtriangle_{\alpha_{\ell+1}}^{\frac{S+\Gamma}{\hbar}}({\mathbf{B}_{\alpha_1\cdots\alpha_\ell}}^\rho\Gamma_\rho)+(Q_{S+\Gamma}+\hbar\Delta)(\hbar\downtriangle_{\alpha_{\ell+1}}^{\frac{S+\Gamma}{\hbar}}\mathbf{C}_{\alpha_1\cdots\alpha_\ell}),\\
		\mathbf{U}_{\alpha_1\cdots\alpha_{\ell+1}}=&\sum_{\rho\in I}\Big[\hbar\partial_{\alpha_{\ell+1}}{\mathbf{B}_{\alpha_1\cdots\alpha_\ell}}^\rho\Gamma_\rho+{\mathbf{B}_{\alpha_1\cdots\alpha_\ell}}^\rho\hbar\downtriangle_{\alpha_{\ell+1}}^{\frac{S+\Gamma}{\hbar}}\hbar\downtriangle_{\alpha_{\rho}}^{\frac{S+\Gamma}{\hbar}}1\Big]\\
		&\hspace{5em}+(Q_{S+\Gamma}+\hbar\Delta)(\hbar\downtriangle_{\alpha_{\ell+1}}^{\frac{S+\Gamma}{\hbar}}\mathbf{C}_{\alpha_1\cdots\alpha_\ell})\\
		=\sum_{\rho\in I}\bigg[\hbar&\partial_{\alpha_{\ell+1}}{\mathbf{B}_{\alpha_1\cdots\alpha_\ell}}^\rho\Gamma_\rho+{\mathbf{B}_{\alpha_1\cdots\alpha_\ell}}^\rho\Big[\sum_{\delta\in I}A_{\alpha_{\ell+1}\rho}{}^\delta\Gamma_\delta+(Q_{S+\Gamma}+\hbar\Delta)(\Lambda_{\alpha_{\ell+1}\rho})\Big]\bigg]\\
		&\hspace{5em}+(Q_{S+\Gamma}+\hbar\Delta)(\hbar\downtriangle_{\alpha_{\ell+1}}^{\frac{S+\Gamma}{\hbar}}\mathbf{C}_{\alpha_1\cdots\alpha_\ell})\\
		=\sum_{\rho\in I}\bigg[\hbar&\partial_{\alpha_{\ell+1}}{\mathbf{B}_{\alpha_1\cdots\alpha_\ell}}^\rho+\sum_{\delta\in I}\mathbf{B}_{\alpha_1\cdots\alpha_\ell}{}^\delta A_{\alpha_{\ell+1}\delta}{}^\rho\bigg]\Gamma_\rho\\
		&\hspace{5em}+(Q_{S+\Gamma}+\hbar\Delta)(\sum_{\delta\in I}\mathbf{B}_{\alpha_1\cdots\alpha_\ell}{}^\delta\Lambda_{\alpha_{\ell+1}\delta}+\hbar\downtriangle_{\alpha_{\ell+1}}^{\frac{S+\Gamma}{\hbar}}\mathbf{C}_{\alpha_1\cdots\alpha_\ell}).
	\end{align*}
	Recall that $\hbar\downtriangle_{\alpha}^{\frac{S+\Gamma}{\hbar}}$ and $Q_{S+\Gamma}+\hbar\Delta$ commute. The last equation implies
	\[
		\mathbf{U}_{\alpha_1\cdots\alpha_{\ell+1}}=\sum_{\rho\in I}{\mathbf{B}_{\alpha_1\cdots\alpha_{\ell+1}}}^\rho\Gamma_\rho+(Q_{S+\Gamma}+\hbar\Delta)(\mathbf{C}_{\alpha_1\cdots\alpha_{\ell+1}})
	\]
	by the definitions of $\mathbf{U},\mathbf{B}$, and $\mathbf{C}$. By the mathematical induction, the result follows.
\end{proof}

By a direct computation, we get the following lemma:
\begin{lemma}
	For $\underline{\alpha}\in I^m$, we have the following formula:
	\[
		\mathbf{U}_{\underline{\alpha}}=\sum_{i=0}^{m-1}{}^iU_{\underline{\alpha}}\hbar^i,\quad\textrm{where}\quad{}^iU_{\underline{\alpha}}=\sum_{\substack{\underline{\alpha}_1\sqcup\cdots\sqcup \underline{\alpha}_{m-i}=\underline{\alpha}\\\underline{\alpha}_j\neq\emptyset}}\frac{1}{(m-i)!}\Gamma_{\underline{\alpha}_1}\cdots\Gamma_{\underline{\alpha}_{m-i}},
	\]
	where $\underline{\alpha}_1\sqcup\cdots\sqcup \underline{\alpha}_{m-i}=\underline{\alpha}$ means $\underline{\alpha}_1\cup\cdots\cup \underline{\alpha}_{m-i}=\underline{\alpha}$, and $\underline{\alpha}_k\cap \underline{\alpha}_\ell=\emptyset$ for $\ell\neq k$.
\end{lemma}
For example,
\begin{align*}
	\mathbf{U}_{\alpha\beta\gamma}=&\Gamma_\alpha\Gamma_\beta\Gamma_\gamma+\hbar(\Gamma_\alpha\Gamma_{\beta\gamma}+\Gamma_{\beta}\Gamma_{\alpha\gamma}+\Gamma_\gamma\Gamma_{\alpha\beta})+\hbar^2\Gamma_{\alpha\beta\gamma},\\
	\mathbf{U}_{\alpha\beta\gamma\delta}=&\Gamma_\alpha\Gamma_\beta\Gamma_\gamma\Gamma_\delta+\hbar(\Gamma_\alpha\Gamma_{\beta\gamma\delta}+\Gamma_{\beta}\Gamma_{\alpha\gamma\delta}+\Gamma_\gamma\Gamma_{\alpha\beta\delta}+\Gamma_{\delta}\Gamma_{\alpha\beta\gamma})\\
	&+\hbar^2(\Gamma_{\alpha\beta}\Gamma_{\gamma\delta}+\Gamma_{\alpha\gamma}\Gamma_{\beta\delta}+\Gamma_{\alpha\delta}\Gamma_{\beta\gamma})+\hbar^3\Gamma_{\alpha\beta\gamma\delta}.
\end{align*}
Therefore, we get the following result:
\begin{equation}\label{rel_U-u}
	{}^iU_{\underline{\alpha}}|_{\underline{t}=0}=u^{(i)}_{\underline{\alpha}}=\sum_{\substack{\underline{\alpha}_1\sqcup\cdots\sqcup \underline{\alpha}_{m-i}=\underline{\alpha}\\\underline{\alpha}_j\neq\emptyset}}\frac{1}{(m-i)!}u_{\underline{\alpha}_1}\cdots u_{\underline{\alpha}_{m-i}},
\end{equation}
where $u_{\underline{\alpha}}^{(i)}$ appeared in Definition \ref{u_alpha}.\par
By the definitions of $\mathbf{B}$ and $\mathbf{C}$, the $\hbar$-degree of $\mathbf{B}_{\underline{\alpha}}{}^\rho$ and $\mathbf{C}_{\underline{\alpha}}$ is $m-2$ when $\underline{\alpha}\in I^m$, i.e., we can write
\[
	\mathbf{B}_{\underline{\alpha}}{}^\rho=\sum_{i=0}^{m-2}{}^iB_{\underline{\alpha}}{}^\rho\hbar^i,\quad\mathbf{C}_{\underline{\alpha}}=\sum_{i=0}^{m-2}{}^iC_{\underline{\alpha}}\hbar^i,
\]
for some ${}^iB_{\underline{\alpha}}{}^\rho\in\mathbb{C}[\![\underline{t}]\!]$, and ${}^iC_{\underline{\alpha}}\in\mathcal{A}_0^{-1}[\![\underline{t}]\!]$. Therefore, by comparing the $\hbar$-power terms of \eqref{Ind_QM}, we get the following sequence of equations:
\begin{equation}\label{IND_U}
	\begin{aligned}
		{}^0U_{\underline{\alpha}}&=\sum_{\rho\in I}{}^0B_{\underline{\alpha}}{}^\rho\Gamma_\rho+Q_{S+\Gamma}({}^0C_{\underline{\alpha}}),\\
		{}^1U_{\underline{\alpha}}&=\sum_{\rho\in I}{}^1B_{\underline{\alpha}}{}^\rho\Gamma_\rho+Q_{S+\Gamma}({}^1C_{\underline{\alpha}})+\Delta({}^0C_{\underline{\alpha}}),\\
		&\quad\vdots\\
		{}^{m-2}U_{\underline{\alpha}}&=\sum_{\rho\in I}{}^{m-2}B_{\underline{\alpha}}{}^\rho\Gamma_\rho+Q_{S+\Gamma}({}^{m-2}C_{\underline{\alpha}})+\Delta({}^{m-3}C_{\underline{\alpha}}),\\
		{}^{m-1}U_{\underline{\alpha}}&=\Delta({}^{m-2}C_{\underline{\alpha}}).
	\end{aligned}
\end{equation}

\begin{lemma}\label{Lemm_B}
	The constant term of $\underline{t}^{\ud \a}$-expansion of ${}^{m-2}B_{\underline{\alpha}}{}^\rho$ is $a_{\underline{\alpha}}{}^\rho$ which appeared in \eqref{Exp_notation}, i.e.,
	\[
		{}^{m-2}B_{\underline{\alpha}}{}^\rho|_{\underline{t}=0}=a_{\underline{\alpha}}{}^\rho.
	\]
\end{lemma}

\begin{proof}
	It is enough to show that
	\begin{equation}\label{Eqn_B}
		{}^{m-2}B_{\underline{\alpha}}{}^\rho=\partial_{\alpha_m}\partial_{\alpha_{m-1}}\cdots\partial_{\alpha_3}A_{\alpha_1\alpha_2}{}^\rho,
	\end{equation}
	where $\underline{\alpha}=\alpha_1\alpha_2\cdots\alpha_m\in I^m$ for $m\geq 2$. Use the mathematical induction. When $m=2$, $\mathbf{B}_{\alpha_1\alpha_2}{}^\rho=A_{\alpha_1\alpha_2}{}^\rho$ by Definition \ref{Def_UBC}, which implies that ${}^0B_{\alpha_1\alpha_2}{}^\rho=A_{\alpha_1\alpha_2}{}^\rho$.\par
	Suppose that equation \eqref{Eqn_B} is true for $|\underline{\alpha}|=\ell$. By Definition \ref{Def_UBC},
	\[
		{\mathbf{B}_{\alpha_1\cdots\alpha_{\ell+1}}}^\rho=\sum_{\delta\in I}{\mathbf{B}_{\alpha_1\cdots\alpha_{\ell}}}^\delta{\mathbf{B}_{\delta\alpha_{\ell+1}}}^\rho+\hbar\cdot\partial_{\alpha_{\ell+1}}({\mathbf{B}_{\alpha_1\cdots\alpha_{\ell}}}^\rho).
	\]
	Therefore, the $\hbar^{\ell-1}$-term of $\mathbf{B}_{\alpha_1\cdots\alpha_{\ell+1}}{}^\rho$ is $\partial_{\alpha_{\ell+1}}({}^{\ell-2}B_{\alpha_1\cdots\alpha_\ell}{}^\rho)=\partial_{\alpha_{\ell+1}}\cdots\partial_{\alpha_3}A_{\alpha_1\alpha_2}{}^\rho$.
\end{proof}
According to Lemma \ref{Lemm_B} and the equation \eqref{rel_U-u}, we get the equations \eqref{ind_u} by evaluating $\underline{t}=0$ on the equations \eqref{IND_U}. This implies that the algorithm in subsection \ref{Subsec_Alg} works to give a solution to \eqref{F-QM}.

\section{Weak primitive forms}\label{Sec_WP}

\subsection{\texorpdfstring{Weak primitive forms and formal $F$-manifolds}{Weak primitive forms and formal F-manifolds}}\label{sec4.1}
Here we explain how weak primitive forms give arise to formal $F$-manifolds.
For this, we consider a Maurer-Cartan solution of a particular shape, $\G=L:=\sum_{\rho\in I}u_\rho t^\rho$ with $\wt(\G)=1$. Let $\boldsymbol\zeta:={}^0\zeta+{}^1\zeta\hbar$ where ${}^0\zeta,{}^1\zeta\in\mathcal{A}_0^0 [\![\underline{t}]\!]$ with ${}^0\zeta|_{\underline{t}=0}=1$ and $\wt(\boldsymbol\zeta)=0$. 
Note that $\hbar\downtriangle_{\rho}^{\frac{S+L}{\hbar}}\boldsymbol\zeta|_{\underline{t}=0,\hbar=0}=u_\rho$ for all $\rho\in I$. Therefore, $\{\hbar\downtriangle_{\rho}^{\frac{S+L}{\hbar}}\boldsymbol\zeta:\rho\in I\}$ is a $\mathbb{C}[\![\underline{t}]\!][\![\hbar]\!]$-basis of $\mathcal{H}_{S+L}^{(0)}$ (Definition \ref{H-module}) by Lemma \ref{Basis_lemma_F}. It means that there exists a unique $\mathbf{A}_{\alpha\beta}{}^\rho\in\mathbb{C}[\![\underline{t}]\!][\![\hbar]\!]$ for all $\alpha,\beta,\rho\in I$ such that
\begin{equation}\label{GCM_demo}
	\hbar\downtriangle_\beta^{\frac{S+L}{\hbar}}(\hbar\downtriangle_\alpha^{\frac{S+L}{\hbar}}\boldsymbol\zeta)=\sum_{\rho\in I}\mathbf{A}_{\alpha\beta}{}^\rho\cdot\downtriangle_\rho^{\frac{S+L}{\hbar}}\boldsymbol\zeta+(Q_{S+L}+\hbar\Delta)(\mathbf{\Lambda}_{\alpha\beta}),
\end{equation}
for some $\mathbf{\Lambda}_{\alpha\beta}\in\mathcal{A}_0^{-1}[\![\underline{t}]\!][\![\hbar]\!]$. 
Now we assume that $\boldsymbol\zeta$ is a weak primitive form (Definition \ref{Def_WP}), i.e., $\mathbf{A}_{\alpha\beta}{}^\rho$ has a particular shape $\mathbf{A}_{\alpha\beta}{}^\rho={}^0A_{\alpha\beta}{}^\rho+{}^1A_{\alpha\beta}{}^\rho\hbar$ and $\mathbf{\Lambda}_{\alpha\beta}={}^0\Lambda_{\alpha\beta}+{}^1\Lambda_{\alpha\beta}\hbar$ where ${}^0A_{\alpha\beta}{}^\rho,{}^1A_{\alpha\beta}{}^\rho\in\mathcal{A}_0^0 [\![\underline{t}]\!]$ and ${}^0\Lambda_{\alpha\beta},{}^1\Lambda_{\alpha\beta}\in\mathcal{A}_0^{-1}[\![\underline{t}]\!]$ satisfy the equation \eqref{GCM_demo}. Then, we can rewrite the equation \eqref{GCM_demo} as:
\begin{equation}\label{GCM}
	\hbar\downtriangle_\beta^{\frac{S+L}{\hbar}}(\hbar\downtriangle_\alpha^{\frac{S+L}{\hbar}}\boldsymbol\zeta)=\sum_{\rho\in I}({}^0A_{\alpha\beta}{}^\rho+{}^1A_{\alpha\beta}{}^\rho\hbar)\cdot\hbar\downtriangle_\rho^{\frac{S+L}{\hbar}}\boldsymbol\zeta+(Q_{S+L}+\hbar\Delta)({}^0\Lambda_{\alpha\beta}+{}^1\Lambda_{\alpha\beta}\hbar).
\end{equation}
Note that the equation \eqref{GCM} and the equation (S2) (Proposition 7.3 in \cite{ST}) are the same equations; this is the motivation for Definition \ref{Def_WP} of weak primitive forms.
In order to explain its connection to formal $F$-manifolds, consider the following equation (with a simplified notation $\downtriangle_\a=\downtriangle_\a^{\frac{S+L}{\hbar}}$):
\begin{align*}
	\hbar\downtriangle_\gamma\hbar\downtriangle_\beta\hbar\downtriangle_\alpha\boldsymbol\zeta=&\hbar\downtriangle_\gamma \Big[\sum_{\rho\in I}\left(\mathbf{A}_{\alpha\beta}{}^\rho\right)\hbar\downtriangle_\rho\boldsymbol\zeta\Big]+(Q_{S+L}+\hbar\Delta)(\hbar\downtriangle_\gamma\mathbf{\Lambda}_{\alpha\beta})\\
	=&\sum_{\rho\in I}\Big[\hbar\partial_\gamma(\mathbf{A}_{\alpha\beta}{}^\rho)\hbar\downtriangle_\rho\boldsymbol\zeta+(\mathbf{A}_{\alpha\beta}{}^\rho)\hbar\downtriangle_\gamma\hbar\downtriangle_\rho\boldsymbol\zeta\Big]+(Q_{S+L}+\hbar\Delta)(\hbar\downtriangle_\gamma\mathbf{\Lambda}_{\alpha\beta})\\
	=&\sum_{\delta\in I}\bigg[\hbar\mathbf{A}_{\alpha\beta,\gamma}{}^\delta+\sum_{\rho\in I}\big(\mathbf{A}_{\alpha\beta}{}^\rho\mathbf{A}_{\gamma\rho}{}^\delta\big)\bigg]\hbar\downtriangle_\delta\boldsymbol\zeta\\
	&+(Q_{S+L}+\hbar\Delta)\Big(\hbar\downtriangle_\gamma\mathbf{\Lambda}+\sum_{\rho\in I}\mathbf{A}_{\alpha\beta}{}^\rho\mathbf{\Lambda}_{\gamma\rho}\Big).
\end{align*}
Therefore, the following quantities are invariant under the permutation of $\alpha,\beta,\gamma$ for all $\delta$:
\begin{equation}\label{weak_F}
	\begin{aligned}
		\hbar\mathbf{A}_{\alpha\beta,\gamma}{}^\delta+&\sum_{\rho\in I}\big(\mathbf{A}_{\alpha\beta}{}^\rho\mathbf{A}_{\gamma\rho}{}^\delta\big)\\
		=&\sum_{\rho\in I}{}^0A_{\alpha\beta}{}^\rho\cdot{}^0A_{\gamma\rho}{}^\delta+\hbar\Big[{}^0A_{\alpha\beta,\gamma}{}^\delta+\sum_{\rho\in I}{}^0A_{\alpha\beta}{}^\rho\cdot{}^1A_{\gamma\rho}{}^\delta+{}^1A_{\alpha\beta}{}^\rho\cdot{}^0A_{\gamma\rho}{}^\delta\Big]\\
		&+\hbar^2\Big[{}^1A_{\alpha\beta,\gamma}{}^\delta+\sum_{\rho\in I}{}^1A_{\alpha\beta}{}^\rho\cdot{}^1A_{\gamma\rho}{}^\delta\Big].
	\end{aligned}
\end{equation}
Suppose that ${}^1A_{\alpha\beta}{}^\rho=0$ for all $\alpha,\beta,\rho\in I$. Then, the equation \eqref{weak_F} implies that the following quantities are invariant under the permutation of $\alpha,\beta,\gamma$ for all $\delta$:
\[
	\sum_{\rho\in I}{}^0A_{\alpha\beta}{}^\rho\cdot{}^0A_{\gamma\rho}{}^\delta\quad\textrm{and}\quad {}^0A_{\alpha\beta,\gamma}{}^\delta.
\]
Note that ${}^0A_{\alpha\beta}{}^\rho={}^0A_{\beta\alpha}{}^\rho$ from \eqref{GCM}. Therefore, ${}^0A_{\alpha\beta}{}^\rho$ induces the structure of the formal $F$-manifold structure when ${}^1A_{\alpha\beta}{}^\rho=0$.

\begin{remark}\label{rmk_1}
	Suppose that we have the Levi-Civita connection represented by
	\[
		\nabla_{\partial_\alpha}\partial_\beta:=\sum_{\rho_\in I}{}^1A_{\alpha\beta}{}^\rho\partial_\rho,
	\]
	which is compatible with the metric $\langle~,~\rangle$ defined by
	\[
		\langle\partial_\alpha,\partial_\beta\rangle:={}^0A_{\alpha\beta}{}^\mathrm{max}
	\]
	where ${}^0A_{\alpha\beta}{}^\rho$ and ${}^1A_{\alpha\beta}{}^\rho$ appeared in the equation \eqref{GCM}, i.e., ${}^1A_{\alpha\beta}{}^\rho$ is the Christoffel symbol of the Levi-Civita connection $\nabla$. Then, the above argument implies that taking the flat coordinates (${}^1A_{\alpha\beta}{}^\rho=0$) for weak primitive forms (and $\hbar \to 0$) induces a formal $F$-manifold structure, as we indicated in Table \ref{Table1}.
\end{remark}

\subsection{An explicit algorithm}\label{Subsec_Alg_prim}
We will give an explicit algorithm which calculates $\boldsymbol\zeta={}^0\zeta+{}^1\zeta\hbar$, $\mathbf{A}_{\alpha\beta}{}^{\rho}={}^0A_{\alpha\beta}{}^\rho+{}^1A_{\alpha\beta}{}^\rho\hbar$ satisfying the equation \eqref{GCM}. For notational convenience, let $12\cdots m$ denote $\alpha_1\alpha_2\cdots\alpha_m$. We use the following notations:
\[
	{}^iA_{12}{}^\rho={}^ia_{12}{}^\rho+\sum_{m=1}^\infty\frac{1}{m!}\sum_{\underline{\alpha}\in I^m}{}^ia_{12\underline{\alpha}}{}^\rho\cdot t^{\underline{\alpha}},\quad{}^i\zeta={}^i\bar{\zeta}+\sum_{m=1}^\infty\frac{1}{m!}\sum_{\underline{\alpha}\in I^m}{}^i\bar{\zeta}_{\underline{\alpha}}\cdot t^{\underline{\alpha}}.\quad(i=0,1)
\]
Let ${}^0\bar{\zeta}=1$, so that $\hbar\downtriangle_\rho\boldsymbol\zeta$ is a basis of $\mathcal{H}_{S+L}^{(0)}$, and let ${}^1\bar{\zeta}=0$, so that the weight condition $\wt(\boldsymbol\zeta)=0$ holds. We have to determine ${}^i\bar{\zeta}_{\underline{\alpha}}$ and ${}^ia_{\underline{\alpha}}{}^\rho$ $(i=0,1)$ which satisfy the equation \eqref{GCM}.\
For the algorithm, we need some definition.

\begin{definition}\label{Def_v-b}
	Let $\underline{\alpha}=\alpha_1\cdots\alpha_m=1\cdots m$. We define the notations $v_{\underline{\alpha}}(\ell,j)$ and $b_{\underline{\alpha}}{}^\rho(\ell,j)$ as follows:
	\[
		v_{\underline{\alpha}}(\ell,j)=\sum_{\substack{\underline{\alpha}_1\sqcup\underline{\alpha}_2=\underline{\alpha}\\|\underline{\alpha}_1|=\ell}}u(\underline{\alpha}_2)\cdot{}^j\bar{\zeta}_{\underline{\alpha}_1}\quad\textrm{where}\quad u(\underline{\alpha}):=u_1 u_2\cdots u_{m-1}u_m,
	\]
	and
	\[
		b_{\underline{\alpha}}{}^\rho(\ell,j)=\!\!\!\!\!\!\sum_{\substack{\underline{\alpha}_1\sqcup\cdots\sqcup\underline{\alpha}_{s+1}=3\cdots m\\\underline{\alpha}_i\neq\emptyset~\textrm{when}~i\neq1}}\frac{1}{s!}\sum_{\substack{r_1+\cdots r_{s+1}=\ell-j\\r_i=0~\textrm{or}~1}}\Big(\sum_{\delta_{s}}\cdots\sum_{\delta_1}{}^{r_1}{a_{12\underline{\alpha}_1}}^{\delta_1}\cdots{}^{r_{s+1}}{a_{\delta_{s}\underline{\alpha}_{s+1}}}^{\rho}\Big).
	\]
	where $s=m-j-2$. We define $b_{\underline{\alpha}}{}^\rho(\ell,j)=0$ when $s=m-j-2<0$.
\end{definition}

Note that $v_{\underline{\alpha}}(\ell,j)$ is invariant under the permutation of the indices of $\underline{\alpha}$. For example,
\[
	\begin{aligned}
		v_{123}(0,0)&=u_{123}{}^0\bar{\zeta}=u_{123},\\
		v_{123}(1,0)&=u_{12}{}^0\bar{\zeta}_3+u_{13}{}^0\bar{\zeta}_2+u_{23}{}^0\bar{\zeta}_1,\\
		v_{123}(1,1)&=u_{12}{}^1\bar{\zeta}_3+u_{13}{}^1\bar{\zeta}_2+u_{23}{}^1\bar{\zeta}_1,\\
		v_{123}(2,0)&=u_{1}{}^0\bar{\zeta}_{23}+u_{2}{}^0\bar{\zeta}_{13}+u_{3}{}^0\bar{\zeta}_{12},\\
		v_{123}(2,1)&=u_{1}{}^1\bar{\zeta}_{23}+u_{2}{}^1\bar{\zeta}_{13}+u_{3}{}^1\bar{\zeta}_{12},
	\end{aligned}\quad\textrm{and}\quad
	\begin{aligned}
		b_{123}{}^\rho(0,0)& = \sum_\delta {}^0a_{12}{}^\delta\cdot{}^0a_{\delta3}{}^\rho,\\
		b_{123}{}^\rho(1,0)& = \sum_\delta {}^1a_{12}{}^\delta\cdot{}^0a_{\delta3}{}^\rho+{}^0a_{12}{}^\delta\cdot{}^1a_{\delta3}{}^\rho,\\
		b_{123}{}^\rho(1,1)& = {}^0a_{123}{}^\rho,\\
		b_{123}{}^\rho(2,0)& = \sum_\delta {}^1a_{12}{}^\delta\cdot{}^1a_{\delta3}{}^\rho,\\
		b_{123}{}^\rho(2,1)& = {}^1a_{123}{}^\rho.
	\end{aligned}
\]
	
Let us introduce the algorithm for $\boldsymbol\zeta$ and $\mathbf{A}_{\alpha\beta}{}^\rho$ as follows:

\begin{description}
	\item[Step-0] Choose a $\mathbb{C}$-basis $\{u_\alpha:\alpha\in I\}$ of $J_S:=\mathcal{A}_0^0/Q_S(\mathcal{A}_0^{-1})$, and ${}^0\bar{\zeta}_\rho$, ${}^1\bar{\zeta}_\rho$. We can choose arbitrary ${}^0\bar{\zeta}_\rho$ and ${}^1\bar{\zeta}_\rho$ with weight conditions.
	\item[Step-1] Determine $b_{12}^{(0)}$, $c_{12}^{(0)}$, $b_{12}^{(1)}$, and $c_{12}^{(1)}$ using a basis $\{u_\alpha:\alpha\in I\}$ of $J_S$ as follows:
	\begin{equation}
		\begin{aligned}
			v_{12}(0,0)&=\sum_{\rho\in I}b_{12}^{(0)}{}^\rho\cdot u_\rho+Q_S(c_{12}^{(0)}),\\
			v_{12}(1,0)+v_{12}(0,1)-\sum_{\rho\in I}b_{12}^{(0)}{}^\rho\cdot{}^0\bar{\zeta}_\rho-\Delta(c_{12}^{(0)})&=\sum_{\rho\in I}b_{12}^{(1)}{}^\rho\cdot u_\rho+Q_S(c_{12}^{(1)}).
		\end{aligned}
	\end{equation}
	Note that $b_{12}^{(i)}{}^\rho\in\mathbb{C}$ is unique and $c_{1}^{(i)}$ is unique up to $\ker Q_S$ for $i=0,1$. Therefore, $\Delta(c_{12}^{(i)})$ is unique up to $\mathrm{Im}(Q_S)$ by the weak $Q_S\Delta$-lemma \ref{Q-D_Lemma}. Then define
	\[
		{}^0a_{12}{}^\rho:=b_{12}^{(0)}{}^\rho,\quad{}^1a_{12}{}^\rho:=b_{12}^{(1)}{}^\rho,
	\]
	and
	\[
		{}^0\bar{\zeta}_{12}:=\sum_{\rho\in I}\big(b_{12}^{(1)}{}^\rho\cdot{}^0\bar{\zeta}_\rho+b_{12}^{(0)}{}^\rho\cdot{}^1\bar{\zeta}_\rho\big)+\Delta(c_{12}^{(1)})-v_{12}(1,1),\quad{}^1\bar{\zeta}_{12}:=\sum_{\rho\in I}b_{12}^{(1)}{}^\rho\cdot{}^1\bar{\zeta}_\rho.
	\]
	
	\item[Step-$\boldsymbol\ell$ $(\ell\geq2)$] Suppose that $|\underline{\alpha}|=\ell+1$. Determine $b_{\underline{\alpha}}^{(i)}{}^\rho$ and $c_{\underline{\alpha}}^{(i)}$ ($0\leq i\leq \ell$) in sequence as follows:
	\begin{equation}\label{ind_prim}
		\begin{aligned}
			v_{\underline{\alpha}}(0,0)&=\sum_{\rho\in I}b_{\underline{\alpha}}^{(0)}{}^\rho u_\rho+Q_S(c_{\underline{\alpha}}^{(0)}),\\
			v_{\underline{\alpha}}(1,0)+v_{\underline{\alpha}}(0,1)-\sum_{\rho\in I}b_{\underline{\alpha}}^{(0)}{}^\rho\cdot{}^0\bar{\zeta}_\rho-\Delta(c_{\underline{\alpha}}^{(0)})&=\sum_{\rho\in I}b_{\underline{\alpha}}^{(1)}{}^\rho u_\rho+Q_S(c_{\underline{\alpha}}^{(1)}),\\
			\left[\begin{array}{l}
				v_{\underline{\alpha}}(2,0)+v_{\underline{\alpha}}(1,1)\\
				-{\displaystyle\sum_{\rho\in I}}\Big[b_{\underline{\alpha}}^{(1)}{}^\rho\cdot{}^0\bar{\zeta}_\rho+b_{\underline{\alpha}}^{(0)}{}^\rho\cdot{}^1\bar{\zeta}_\rho\Big]-\Delta(c_{\underline{\alpha}}^{(1)})
			\end{array}\right]
			&=\sum_{\rho\in I}b_{\underline{\alpha}}^{(2)}{}^\rho u_\rho+Q_S(c_{\underline{\alpha}}^{(2)}),\\
			&\vdots\\
			\left[\begin{array}{l}
				v_{\underline{\alpha}}(\ell,0)+v_{\underline{\alpha}}(\ell-1,1)\\
				-{\displaystyle\sum_{\rho\in I}}\Big[b_{\underline{\alpha}}^{(\ell-1)}{}^\rho\cdot{}^0\bar{\zeta}_\rho+b_{\underline{\alpha}}^{(\ell-2)}{}^\rho\cdot{}^1\bar{\zeta}_\rho\Big]-\Delta(c_{\underline{\alpha}}^{(\ell-1)})
			\end{array}\right]
			&=\sum_{\rho\in I}b_{\underline{\alpha}}^{(\ell)}{}^\rho u_\rho+Q_S(c_{\underline{\alpha}}^{(\ell)}),\\
		\end{aligned}
	\end{equation}
	This inductive procedure works well, because the left hand sides of the equation \eqref{ind_prim} are already determined by the previous steps for $i\leq\ell-1$. Moreover, $b_{\underline{\alpha}}^{(\ell-1)}{}^\rho$ and $b_{\underline{\alpha}}^{(\ell)}{}^\rho$ are independent of the choices of $c_{\underline{\alpha}}^{(i)}$, since $c_{\underline{\alpha}}^{(i)}$ is unique up to $\ker Q_S$ for $0\leq i\leq\ell-1$, thus $\Delta(c_{\underline{\alpha}}^{(i)})$ is unique up to $\mathrm{Im}(Q_S)$ by Lemma \ref{Q-D_Lemma}.   Then, define
	\begin{equation}\label{ind_b_for_a}
		{}^0a_{\underline{\alpha}}{}^\rho:=b_{\underline{\alpha}}^{(\ell-1)}{}^\rho-\sum_{j=0}^{\ell-2}b_{\underline{\alpha}}{}^\rho(\ell-1,j),\quad{}^1a_{\underline{\alpha}}{}^\rho:=b_{\underline{\alpha}}^{(\ell)}{}^\rho-\sum_{j=0}^{\ell-2}{}^{n-1}b_{\underline{\alpha}}{}^\rho(\ell,j),
	\end{equation}
	and
	\begin{equation}\label{zeta_alg}
		\begin{aligned}
				{}^0\bar{\zeta}_{\underline{\alpha}}&:=\sum_{\rho\in I}\big(b_{\underline{\alpha}}^{(\ell)}{}^\rho\cdot{}^0\bar{\zeta}_\rho+b_{\underline{\alpha}}^{(\ell-1)}{}^\rho\cdot{}^1\bar{\zeta}_\rho\big)+\Delta(c_{\underline{\alpha}}^{(\ell)})-v_{\underline{\alpha}}(\ell,1),\\
				{}^1\bar{\zeta}_{\underline{\alpha}}&:=\sum_{\rho\in I}b_{\underline{\alpha}}^{(\ell)}{}^\rho\cdot{}^1\bar{\zeta}_\rho.
		\end{aligned}
	\end{equation}
\end{description}
By this algorithm, we can determine $\boldsymbol\zeta$ and $\mathbf{A}_{\alpha\beta}{}^\rho$ which turn out to satisfy (see subsection \ref{sec4.3}) the equation \eqref{GCM}.

\subsection{A proof why the algorithm works}\label{sec4.3}
In this subsection, we give a proof of the algorithm in subsection \ref{Subsec_Alg_prim}.
\begin{definition}\label{Def_UBC_B}
	We define quantities $\mathbf{V},\mathbf{D},\mathbf{E}$, which are our key players in the proof of the algorithm, as follows:
	\[
		\mathbf{V}_{\alpha_1\cdots\alpha_m}:=\hbar\downtriangle_{\alpha_m}^{\frac{S+L}{\hbar}}\cdots\hbar\downtriangle_{\alpha_1}^{\frac{S+L}{\hbar}}\boldsymbol\zeta,
	\]
	\[
		\left\{\begin{aligned}
			{\mathbf{D}_{\alpha_1\alpha_2}}^\rho:=&{\mathbf{A}_{\alpha_1\alpha_2}}^\rho,\\
			{\mathbf{D}_{\alpha_1\alpha_2\alpha_3}}^\rho:=&\sum_{\delta\in I}{\mathbf{D}_{\alpha_1\alpha_2}}^\delta{\mathbf{D}_{\delta\alpha_3}}^\rho+\hbar\cdot\partial_{\alpha_3}({\mathbf{D}_{\alpha_1\alpha_2}}^\rho),\\
			&\vdots\\
			{\mathbf{D}_{\alpha_1\cdots\alpha_m}}^\rho:=&\sum_{\delta\in I}{\mathbf{D}_{\alpha_1\cdots\alpha_{m-1}}}^\delta{\mathbf{D}_{\delta\alpha_m}}^\rho+\hbar\cdot\partial_{\alpha_m}({\mathbf{D}_{\alpha_1\cdots\alpha_{m-1}}}^\rho),
		\end{aligned}\right.
	\]
	\[
		\left\{\begin{aligned}
			\mathbf{E}_{\alpha_1\alpha_2}:=&\mathbf{\Lambda}_{\alpha_1\alpha_2},\\
			\mathbf{E}_{\alpha_1\alpha_2\alpha_3}:=&\sum_{\delta\in I}{\mathbf{D}_{\alpha_1\alpha_2}}^\delta\mathbf{E}_{\delta\alpha_3}+\hbar\downtriangle_{\alpha_3}^{\frac{S+L}{\hbar}}\mathbf{E}_{\alpha_1\alpha_2},\\
			&\vdots\\
			\mathbf{E}_{\alpha_1\cdots\alpha_m}:=&\sum_{\delta\in I}{\mathbf{D}_{\alpha_1\cdots\alpha_{m-1}}}^\delta\mathbf{E}_{\delta\alpha_m}+\hbar\downtriangle_{\alpha_m}^{\frac{S+L}{\hbar}}\mathbf{E}_{\alpha_1\cdots\alpha_{m-1}}.
		\end{aligned}\right.
	\]
\end{definition}

\begin{theorem}
	The equations \eqref{GCM} hold for all $\alpha,\beta\in I$ if and only if the following inductive equations hold for all $m\geq2$ and all $\underline{\alpha}\in I^m$:
	\begin{equation}\label{Ind_QM_prim}
		\mathbf{V}_{\underline{\alpha}}=\sum_{\rho\in I}{\mathbf{D}_{\underline{\alpha}}}^\rho\cdot\hbar\downtriangle_\rho^{\frac{S+L}{\hbar}}\boldsymbol\zeta+(Q_{S+L}+\hbar\Delta)(\mathbf{E}_{\underline{\alpha}}).
	\end{equation}
\end{theorem}

\begin{proof}
	Use the same argument as in the proof of Theorem \ref{Thm_UBD}.
\end{proof}

By the mathematical induction, we get the following two lemmas:
\begin{lemma}\label{V_lemma}
	Let $\underline{\alpha}=\alpha_1\cdots\alpha_m$, then
	\begin{multline*}
		\mathbf{V}_{\underline{\alpha}}=V_{\underline{\alpha}}(0,0)+\sum_{\ell=1}^m(V_{\underline{\alpha}}(\ell,0)+V_{\underline{\alpha}}(\ell-1,1))\hbar^\ell+V_{\underline{\alpha}}(m,1)\hbar^{m+1}\\
		\textrm{where}~V_{\underline{\alpha}}(\ell,j)=\sum_{\substack{\underline{\alpha}_1\sqcup\underline{\alpha}_2=\underline{\alpha}\\|\underline{\alpha}_1|=\ell}}u(\underline{\alpha}_2)\cdot{}^j\zeta_{\underline{\alpha}_1}.
	\end{multline*}
\end{lemma}
\begin{lemma}\label{D_lemma}
	Let $\underline{\alpha}=\alpha_1\cdots\alpha_m$, then
	\begin{align*}
		\mathbf{D}_{\underline{\alpha}}{}^\rho&=\sum_{\ell=0}^{m-1}{}^\ell D_{\underline{\alpha}}{}^\rho\hbar^\ell\quad\textrm{where}~{}^\ell D_{\underline{\alpha}}{}^\rho=\sum_{j=0}^\ell B_{\underline{\alpha}}{}^\rho(\ell,j),\\
		&\textrm{and}~B_{\underline{\alpha}}{}^\rho(\ell,j)=\!\!\!\!\!\!\!\sum_{\substack{\underline{\alpha}_1\sqcup\cdots\sqcup\underline{\alpha}_{s+1}=3\cdots m\\\underline{\alpha}_i\neq\emptyset~\textrm{when}~i\neq1}}\frac{1}{s!}\sum_{\substack{r_1+\cdots r_{s+1}=\ell-j\\r_i=0~\textrm{or}~1}}\Big(\sum_{\delta_{s}}\!\cdots\!\sum_{\delta_1}{}^{r_1}{A_{12\underline{\alpha}_1}}^{\delta_1}\cdots{}^{r_{s+1}}{A_{\delta_{s}\underline{\alpha}_{s+1}}}^{\rho}\Big)
		\end{align*}
	with $s=m-j-2$.
\end{lemma}
Therefore, we obtain
\begin{equation}\label{res_VB}
	V_{\underline{\alpha}}(\ell,j)|_{\underline{t}=0}=v_{\underline{\alpha}}(\ell,j)\quad\textrm{and}\quad B_{\underline{\alpha}}{}^\rho(\ell,j)|_{\underline{t}=0}=b_{\underline{\alpha}}{}^\rho(\ell,j),
\end{equation}
where $v_{\underline{\alpha}}(\ell,j)$ and $b_{\underline{\alpha}}{}^\rho(\ell,j)$ appeared in Definition \ref{Def_v-b}.

\begin{lemma}\label{Lemma_res_D}
	We get the following formulas:
	\[
		{}^{m-1}D_{\underline{\alpha}}{}^\rho|_{\underline{t}=0}={}^0a_{\underline{\alpha}}{}^\rho+\sum_{j=0}^{m-3}b_{\underline{\alpha}}{}^\rho(m-2,j)\quad\textrm{and}\quad{}^{m-2}D_{\underline{\alpha}}{}^\rho|_{\underline{t}=0}={}^1a_{\underline{\alpha}}{}^\rho+\sum_{j=0}^{m-3}b_{\underline{\alpha}}{}^\rho(m-1,j).
	\]
\end{lemma}

\begin{proof}[Sketch of proof]
	Use Lemma \ref{V_lemma}, Lemma \ref{D_lemma}, and the equation \eqref{res_VB}.
\end{proof}
We can write 
\[
	\mathbf{E}_{\underline{\alpha}}=\sum_{\ell=0}^{m-1}{}^\ell E_{\underline{\alpha}}\hbar^\ell
\]
for ${}^\ell E_{\underline{\alpha}}\in\mathcal{A}_{0}^{-1}[\![\underline{t}]\!]$.
Therefore, by comparing the $\hbar$-power terms of \eqref{Ind_QM}, we get the following sequence of equations:
\begin{equation}\label{Prim_IND}
	\begin{aligned}
		V_{\underline{\alpha}}(0,0)=&\sum_{\rho\in I}\Big[{}^0D_{\underline{\alpha}}{}^\rho u_\rho\Big]+Q_{S+L}({}^0E_{\underline{\alpha}}),\\
		V_{\underline{\alpha}}(1,0)+V_{\underline{\alpha}}(0,1)=&\sum_{\rho\in I}\Big[{}^0D_{\underline{\alpha}}{}^\rho{}^0\zeta_\rho+{}^1D_{\underline{\alpha}}{}^\rho u_\rho\Big]+Q_{S+L}({}^1E_{\underline{\alpha}})+\Delta({}^0E_{\underline{\alpha}}),\\
		V_{\underline{\alpha}}(2,0)+V_{\underline{\alpha}}(1,1)=&\sum_{\rho\in I}\Big[{}^0D_{\underline{\alpha}}{}^\rho{}^1\zeta_\rho+{}^1D_{\underline{\alpha}}{}^\rho{}^0\zeta_\rho+{}^2D_{\underline{\alpha}}{}^\rho u_\rho\Big]+Q_{S+L}({}^2E_{\underline{\alpha}})+\Delta({}^1E_{\underline{\alpha}}),\\
		&\vdots\\
		V_{\underline{\alpha}}(m-1,0)+V_{\underline{\alpha}}(m-2,1)=&\sum_{\rho\in I}\Big[{}^{m-3}D_{\underline{\alpha}}{}^\rho{}^1\zeta_\rho+{}^{m-2}D_{\underline{\alpha}}{}^\rho{}^0\zeta_\rho+{}^{m-1}D_{\underline{\alpha}}{}^\rho u_\rho\Big]\\
		&+Q_{S+L}({}^{m-1}E_{\underline{\alpha}})+\Delta({}^{m-1}E_{\underline{\alpha}}),\\
		V_{\underline{\alpha}}(m,0)+V_{\underline{\alpha}}(m-1,1)=&\sum_{\rho\in I}\Big[{}^{m-2}D_{\underline{\alpha}}{}^\rho{}^1\zeta_\rho+{}^{m-2}D_{\underline{\alpha}}{}^\rho{}^1\zeta_\rho\Big]+\Delta({}^{m-1}E_{\underline{\alpha}}),\\
		V_{\underline{\alpha}}(m,1)=&\sum_{\rho\in I}\Big[{}^{m-1}D_{\underline{\alpha}}{}^\rho{}^1\zeta_\rho\Big].\\
	\end{aligned}
\end{equation}
By evaluating $\underline{t}=0$ on the equations \eqref{Prim_IND}, we get the equations \eqref{ind_prim} and \eqref{zeta_alg}. We also get the equations \eqref{ind_b_for_a} by Lemma \ref{Lemma_res_D}. This says that the algorithm in subsection \ref{Subsec_Alg_prim} provides us a solution to \eqref{GCM} 

\section{An algorithm for formal Frobenius manifolds and higher residue pairings}\label{Sec_Frob}

\subsection{Definition of formal Frobenius manifolds} \label{sec5.1}
In this subsection we briefly review the definition of Frobenius manifolds in \cite{ST}, \cite{BK} and formal Frobenius manifolds. We will only consider a pure even manifold same as an $F$-manifold case.

\begin{definition}\label{Def_Frob}
	Let $M$ be a complex manifold of finite dimension, and $t_M := \{t^\a\}$ be the formal coordinates on the open subset $U_i \subset M$. Choose sufficiently small open subset $U$ of $U_i$, then we could assume that $\cO(U)=\mathbb{C}[\![t_M]\!]$. On the local coordinates $\{t_M,U\}$, the product $\delta \circ \delta'$ is written as the linear combinations of $\{\partial_\a:=\partial/\partial_{t^\a}\}$ with coefficients in $\bC[\![t_M]\!]$. Let $A_{\a\b}{}^\g \in \bC[\![t_M]\!]$ be a formal power series representing the 3-tensor field such that
	\[
		\partial_\alpha \circ \partial_\beta := \sum_\gamma A_{\alpha\beta}{}^\gamma \partial_\gamma.
	\]
	Let $g$ be a non-degenerate symmetric pairing on $\mathbb{C}$-vector space spanned by $\{\partial_\alpha\}$. Then one can extend $g$ to the symmetric $\mathbb{C}[\![\underline{t}]\!]$-pairing. Let $g_{\alpha\beta}:=g(\partial_\alpha,\partial_\beta)$. Then $(M,\circ,g)$ is called a formal Frobenius manifold if the following conditions are satisfied:
	\begin{enumerate}[(D1)]
		\item\label{D1} (Associativity)
		\[
			\sum_\rho A_{\alpha\beta}{}^\rho A_{\rho\gamma}{}^\delta = \sum_\rho A_{\beta\gamma}{}^\rho A_{\rho\alpha}{}^\delta.
		\]
		\item\label{D2} (Commutativity)
		\[
			A_{\beta\alpha}{}^\gamma = A_{\alpha\beta}{}^\gamma.
		\]
		\item\label{D3} (Invariance) Put $A_{\alpha\beta\gamma} = \sum_\rho A_{\alpha\beta}{}^\rho g_{\rho\gamma}$,
		\[
			A_{\alpha\beta\gamma} = A_{\beta\gamma\alpha}.
		\]
		\item\label{D4} (Flat identity) The distinguished one $\partial_0$ is the identity with respect to $\circ$:
		\[
			A_{0\alpha}{}^\beta =\boldsymbol\delta_\alpha^\beta \quad (\text{where $\boldsymbol\delta_\alpha^\beta$ is the Kronecker delta}).
		\]
		\item\label{D5}(Potential)
		\[
			\partial_\alpha A_{\beta\gamma}{}^\delta= \partial_\beta A_{\alpha\gamma}{}^\delta.
		\]
	\end{enumerate}
\end{definition}

Suppose that both the commutativity condition \ref{D2} and the invariance condition \ref{D3} hold. Then the potential condition \ref{D5} implies the existence of a single power series function $\Phi \in \bC[\![t_M]\!]$ satisfying
\[
	g(\partial_\alpha\circ\partial_\beta,\partial_\gamma) = \partial_\alpha\partial_\beta\partial_\gamma \Phi.
\]

\subsection{\texorpdfstring{Formal Frobenius manifold structures on $\mathbb{H}$}{Formal Frobenius manifold structure on H}} \label{sec5.2}
We now construct a formal manifold structure on the complex manifold $\mathbb{H}\simeq\mathbb{C}^\mu$  by using Theorem \ref{Phi_isom} which says that $\mathbb{H}\simeq\mathcal{A}_0^0/Q_S(\mathcal{A}_0^{-1})$. Let $\{u_\alpha:\alpha\in I\}$ be a $\mathbb{C}$-basis of $J_S=\mathcal{A}_0^0/Q_S(\mathcal{A}_0^{-1})$. Then, $\{u_\alpha:\alpha\in I\}$ is also a $\mathbb{C}[\![\underline{t}]\!](\!(\hbar)\!)$-basis of $J_S^{\underline{t}}(\!(\hbar)\!):=\mathcal{A}_0^0 [\![\underline{t}]\!](\!(\hbar)\!)/Q_S(\mathcal{A}_0^{-1}[\![\underline{t}]\!](\!(\hbar)\!))\simeq J_S\otimes_{\mathbb{C}}\mathbb{C}[\![\underline{t}]\!](\!(\hbar)\!)$. Consider the following equation:
\begin{equation}\label{Jac_u}
	u_\alpha\cdot u_\beta=\sum_{\rho\in I}a_{\alpha\beta}{}^\rho\cdot u_\rho+Q_S(\lambda_{\alpha\beta}),
\end{equation}
where $a_{\alpha\beta}{}^\rho\in\mathbb{C}$ and $\lambda_{\alpha\beta}\in\mathcal{A}_0^{-1}$. Since $\{u_\rho\}$ is a $\mathbb{C}$-basis of $ J_S$, $a_{\alpha\beta}{}^\rho$ is uniquely determined.\par
We use the $A_{\alpha\beta}{}^\rho$ to define a structure of $\mathbb{C}[\![\underline{t}]\!]$-algebra on $ J_S^{\underline{t}}(\!(\hbar)\!)$ as follows:
\begin{equation}\label{Fdeq}
	u_\alpha\cdot u_\beta=\sum_{\rho\in I}A_{\a\b}{}^\rho\cdot u_\rho+Q_S({\Lambda}_{\alpha\beta}),
\end{equation}
where $A_{\a\b}^\rho\in\mathbb{C}[\![\underline{t}]\!](\!(\hbar)\!)$ and ${\Lambda}_{\alpha\beta}\in\mathcal{A}_0^{-1}[\![\underline{t}]\!](\!(\hbar)\!)$. Since $A_{\a\b}{}^\rho$ is also uniquely determined due to the facts that $\{u_\rho : \rho \in I \}$ is a $\mathbb{C}[\![\underline{t}]\!](\!(\hbar)\!)$-basis of $ J_S^{\underline{t}}(\!(\hbar)\!)$, $A_{\a\b}{}^\rho=a_{\alpha\beta}{}^\rho\in\mathbb{C}$ by \eqref{Jac_u} and $\Lambda_{\a\b}$ is $\lambda_{\a\b}$ up to the kernel of $Q_S$.\par
We use $g_{\alpha\beta}$ in order to define the symmetric $\mathbb{C}[\![\underline{t}]\!](\!(\hbar)\!)$-pairing on $ J_S^{\underline{t}}(\!(\hbar)\!)$ as follows:
\[
	g_{\alpha\beta}:=A_{\a\b}{}^{\mathrm{max}}.
\]

\begin{proposition}
	The 3-tensor field $A_{\a\b}{}^\rho$ and $g_{\alpha\beta}$ satisfy the axioms of the formal Frobenius manifold.
\end{proposition}

\begin{proof}
	(i) (Commutativity/Associativity) Since $u_\alpha\cdot u_\beta=u_\beta\cdot u_\alpha$, $A_{\a\b}{}^\gamma={A}_{\beta\alpha}{}^\gamma$ for all $\alpha,\beta,\gamma$.\par
	For associativity, consider the following computations:
	\begin{align*}
		(u_\alpha\cdot u_\beta)\cdot u_\gamma&=\left(\sum_{\rho\in I}A_{\a\b}{}^\rho\cdot u_\rho+Q_S({\Lambda}_{\alpha\beta})\right)\cdot u_\gamma=\sum_{\rho\in I}A_{\a\b}{}^\rho\cdot u_\rho\cdot u_\gamma+Q_S({\Lambda}_{\alpha\beta}u_\gamma)\\
		&=\sum_{\rho\in I}A_{\a\b}{}^\rho\left(\sum_{\delta\in I}{A}_{\rho\gamma}{}^\delta\cdot u_\delta+Q_S({\Lambda}_{\rho\gamma})\right)+Q_S({\Lambda}_{\alpha\beta}u_\gamma)\\
		&=\sum_{\delta\in I}\left(\sum_{\rho\in I}A_{\a\b}{}^\rho {A}_{\rho\gamma}{}^\delta\right)u_\delta+Q_S\left(\sum_{\rho\in I}{\Lambda}_{\rho\gamma}+{\Lambda}_{\alpha\beta}u_\gamma\right),
	\end{align*}
	and
	\begin{align*}
		u_\alpha\cdot(u_\beta\cdot u_\gamma)&=u_\alpha\left(\sum_{\rho\in I}{A}_{\beta\gamma}{}^\rho\cdot u_\rho+Q_S({\Lambda}_{\beta\gamma})\right)=\sum_{\rho\in I}{A}_{\beta\gamma}{}^\rho\cdot u_\rho\cdot u_\alpha+Q_S(u_\alpha{\Lambda}_{\beta\gamma})\\
		&=\sum_{\rho\in I}{A}_{\beta\gamma}{}^\rho\left(\sum_{\delta\in I}{A}_{\rho\alpha}{}^\delta\cdot u_\delta+Q_S({\Lambda}_{\rho\alpha})\right)+Q_S(u_\alpha{\Lambda}_{\beta\gamma})\\
		&=\sum_{\delta\in I}\left(\sum_{\rho\in I}{A}_{\beta\gamma}{}^\rho {A}_{\rho\alpha}{}^\delta\right)u_\delta+Q_S\left(\sum_{\rho\in I}{\Lambda}_{\rho\alpha}+u_\alpha{\Lambda}_{\beta\gamma}\right).
	\end{align*}
	Since $(u_\alpha u_\beta)u_\gamma=u_\alpha(u_\beta u_\gamma)$ and $\{u_\rho\}$ is a $\mathbb{C}[\![\underline{t}]\!]$-basis of $ J_S^{\underline{t}}$, we get
	\[
		\sum_{\rho\in I}A_{\a\b}{}^\rho {A}_{\rho\gamma}{}^\delta=\sum_{\rho\in I}{A}_{\beta\gamma}{}^\rho {A}_{\rho\alpha}{}^\delta
	\]
	for all $\alpha,\beta,\gamma,\delta$.\par
	(ii) (Invariance) By associativity
	\[
		{A}_{\alpha\beta\gamma}=\sum_{\rho\in I}A_{\a\b}{}^\rho {A}_{\rho\gamma}{}^\mathrm{max}=\sum_{\rho\in I}{A}_{\beta\gamma}{}^\rho {A}_{\rho\alpha}{}^\mathrm{max}={A}_{\beta\gamma\alpha},
	\]
	for all $\alpha,\beta,\gamma$.\par
	(iii) (Identity) Since we choose $u_\mathrm{min}=u_0=1$, ${A}_{0\alpha}{}^\beta=\boldsymbol\delta_{\alpha}^{\beta}$ for all $\alpha,\beta$.\par
	(iv) (Potential) Note that $A_{\a\b}{}^\gamma=a_{\alpha\beta}{}^\gamma\in\mathbb{C}$ for all $\alpha,\beta,\gamma$. Hence, $\partial_{\delta}A_{\a\b}{}^\gamma=0$ for all $\alpha,\beta,\gamma,\delta$.
\end{proof}

\begin{remark}\label{relax}
We can extend the solutions $a_{\a\b}{}^\rho$ and $\lambda_{\a\b}$ of \eqref{Fdeq} to the solutions $\mathbf{A}_{\a\b}{}^{\rho}$ and $\mathbf{\Lambda}_{\a\b}$ of the differential equation \eqref{WPF} 
                 \[
			\hbar\downtriangle_\beta^{\frac{L}{\hbar}}\hbar\downtriangle_\alpha^{\frac{L}{\hbar}}1=u_\a u_\b =\sum_{\rho}\mathbf{A}_{\alpha\beta}{}^\rho\cdot u_\rho+Q_{S+L}(\mathbf{\Lambda}_{\alpha\beta})+ \hbar\Delta(\mathbf{\Lambda}_{\alpha\beta}) ,
		\]
by setting $\boldsymbol\zeta =1$, $\Gamma =L$. 
In other words, in terms of the language of the differential equation for weak primitive forms, the above construction corresponds to $\boldsymbol\zeta =1$, $\Gamma =L$, and $\mathbf{A}_{\a\b}{}^{\rho} = a_{\a\b}{}^\rho + {}^1A_{\a\b}{}^\rho \hbar + {}^2A_{\a\b}{}^{\rho} \hbar^2+ \cdots $, as we indicated in Table \ref{Comparison_Table}. But we have to sacrifice the condition \ref{cond2} of Definition \ref{Def_WP}. We could have relaxed the definition of weak primitive forms by allowing higher $\hbar$-power terms in $\mathbf{A}_{\a\b}{}^{\rho}$.
\end{remark}
\subsection{Modified higher residue pairings} \label{sec5.3}
The higher residue pairings are important tools for constructing Frobenius manifolds in \cite{ST}. We introduce its slight variant, the \textit{modified higher residue pairing} for the formal Frobenius structure on $\mathbb{H}$. Recall that $T_\mathbb{H}$ is the space of formal tangent vector fields on $\mathbb{H}$ given in Definition \ref{T_H}. 
\begin{definition}
	Let $[u]:=u+\mathrm{Im}(Q_S)$ in $J_S^{\underline{t}}:=\mathcal{A}_0^0 [\![\underline{t}]\!]/Q_S(\mathcal{A}_0^{-1}[\![\underline{t}]\!])$. Define the map $\eta$ as
	\begin{align*}
		\eta:T_\mathbb{H}\times T_\mathbb{H} &\longrightarrow\mathbb{C}[\![\underline{t}]\!],\\
		([w_1],[w_2])&\longmapsto A_{u_1u_2}{}^\mathrm{max},\quad u_i\in\mathcal{A}_0^0 [\![\underline{t}]\!]~\textrm{and}~[w_i]=[\Phi(u_i)]
	\end{align*}
	where $A_{u_1u_2}{}^\mathrm{max}$ is uniquely determined by
	\[
		u_1u_2=\sum_{\rho\in I}A_{u_1u_2}{}^\rho u_\rho+Q_S(\Lambda_{u_1u_2})
	\]
	where $A_{u_1u_2}{}^\rho\in\mathbb{C}[\![\underline{t}]\!]$, $\Lambda_{u_1u_2}\in\mathcal{A}_0^{-1}[\![\underline{t}]\!]$.
\end{definition}
The following proposition is due to Konno, Theorem 6.1, \cite{Ko91}; the hypersurface case was proved in \cite{CG} and the case of the smooth projective complete intersection with equal degrees $d_1=\cdots = d_k$ was proved in \cite{T}.
\begin{proposition}\label{Prop_comm}
	The following diagram commutes up to a numerical factor:
	\[
		\begin{tikzcd}
			\mathbb{H}\times\mathbb{H}\arrow[r,"\eta|_{\underline{t}=0}"]\arrow[d,"\textrm{cup product}"'] & \mathbb{C}\arrow[d,equal] \\
			H^{2(n-k)}(X_{\underline{G}},\mathbb{C})\arrow[r,"\sim"] & \mathbb{C}.
		\end{tikzcd}
	\]
\end{proposition}

The modified higher residue pairing is an extension of the cup product of $\mathbb{H}$ to the quantization cohomology $T_\mathbb{H}(\!(\hbar)\!)$ so that it is compatible with the Gauss-Manin connection.
Note that Gauss-Manin connection $\downtriangle^{\frac{S+L}{\hbar}}$ is also well defined on $J_S^{\underline{t}}(\!(\hbar)\!)$, since $Q_S$ commutes with $\downtriangle_{\alpha}^{\frac{S+L}{\hbar}}$ and $\downtriangle_{\hbar^{-1}}^{\frac{S+L}{\hbar}}$. 
For a Laurent series $\mathbf{\Phi}(\hbar)$ in $\hbar$, we set $\mathbf{\Phi}^*(\hbar) = \mathbf{\Phi}(-\hbar)$.

\begin{definition}[modified higher residue pairing]\label{mhrp}
	A modified higher residue pairing for $T_\mathbb{H}(\!(\hbar)\!)$ is a $\mathbb{C}[\![\underline{t}]\!]$-bilinear symmetric pairing
	\begin{align*}
		K:T_{\mathbb{H}}(\!(\hbar)\!)\times T_{\mathbb{H}}(\!(\hbar)\!)&\longrightarrow\mathbb{C}[\![\underline{t}]\!](\!(\hbar)\!)\\
		(w_1,w_2)&\longmapsto K(w_1,w_2)=\sum_{r\in\mathbb{Z}}K^{(r)}(w_1,w_2)\hbar^r,\quad K^{(r)}(w_1,w_2)\in\mathbb{C}[\![\underline{t}]\!]
	\end{align*}
	such that
	\begin{enumerate}[(H1)]
		\item\label{H1} $K(w_1,w_2)=K(w_2,w_1)^*$
		\item\label{H2} $g(\hbar)K(w_1,w_2)=K(g(\hbar)w_1,w_2)=K(w_1,g^*(\hbar)w_2), \quad g(\hbar) \in \mathbb{C}(\!(\hbar)\!)$
		\item\label{H3} $\partial_\alpha \left(K(w_1,w_2)\right)=K(\downtriangle_{\alpha}^{\frac{S+L}{\hbar}}w_1,w_2)+K(w_1,\downtriangle_\alpha^{\frac{S+L}{\hbar}}w_2)$
		\item\label{H4} $\partial_{\hbar^{-1}}\left(K(w_1,w_2)\right)=K(\downtriangle_{\hbar^{-1}}^{\frac{S+L}{\hbar}}w_1,w_2)-K(w_1,\downtriangle_{\hbar^{-1}}^{\frac{S+L}{\hbar}}w_2)$
		\item\label{H5}  the following diagram commutes:
		\[
			\begin{tikzcd}
				T_\mathbb{H}[\![\hbar]\!]\times T_\mathbb{H}[\![\hbar]\!]\arrow[r,"K"]\arrow[d,"\hbar=0"'] & \mathbb{C}[\![\underline{t}]\!][\![\hbar]\!]\arrow[d,"\hbar=0"'] \\
				T_\mathbb{H}\times T_\mathbb{H}\arrow[r,"\eta"] & \mathbb{C}[\![\underline{t}]\!],
			\end{tikzcd}
		\]
		where $\eta$ is a $\mathbb{C}[\![\underline{t}]\!]$-bilinear symmetric pairing coming from the cup-product pairing of $\mathbb{H}=H_{\mathrm{prim}}^{n-k}(X_{\underline{G}},\mathbb{C})$ and the vertical maps substitute $\hbar$ with 0.
	\end{enumerate}
\end{definition}

\begin{theorem} \label{hrp}
	If we define
	\[
		K \left([w_1],[w_2]\right):=\mathbf{A}_{u_1u_2^*}{}^\mathrm{max},
	\]
	for $[w_i]=[\Phi(u_i)],~i=1,2$, and
	\[
		u_1u_2^*=\sum_{\rho\in I}\mathbf{A}_{u_1u_2^*}{}^\rho u_\rho+Q_S(\mathbf{\Lambda}_{u_1u_2^*}),
	\]
	then $K:T_\mathbb{H}(\!(\hbar)\!)\times T_\mathbb{H}(\!(\hbar)\!)\to\mathbb{C}[\![\underline{t}]\!](\!(\hbar)\!)$ is the modified higher residue pairing.
\end{theorem}

\begin{proof}
	The conditions \ref{H1} and \ref{H2} are obvious by definition.\par
	For \ref{H3}, consider the following equations:
	\begin{align*}
		(\downtriangle_\alpha^{\frac{S+L}{\hbar}} u_1)u_2^*+u_1(\downtriangle_\alpha^{\frac{S+L}{\hbar}}u_2)^*&=\left(\partial_\alpha u_1+\frac{u_\alpha}{\hbar}u_1\right)u_2^*+u_1 \left(\partial_\alpha u_2+\frac{u_\alpha}{\hbar}u_2\right)^*\\
		&=(\partial_\alpha u_1)u_2^*+u_1\partial_\alpha(u_2^*)\\
		&=\partial_\alpha(u_1u_2^*).
	\end{align*}
	We know that $\partial_\alpha(\mathbf{A}_{u_1u_2^*}{}^\rho)=\mathbf{A}_{\partial_\alpha(u_1u_2^*)}{}^\rho$, since
	\begin{align*}
		\partial_\alpha(u_1u_2^*)&=\partial_\alpha \left(\sum_{\rho\in I}\mathbf{A}_{u_1u_2^*}{}^\rho u_\rho+Q_S(\mathbf{\Lambda}_{u_1u_2^*})\right)=\sum_{\rho\in I}(\partial_\alpha\mathbf{A}_{u_1u_2^*}{}^\rho)u_\rho+Q_S(\partial_\alpha\mathbf{\Lambda}_{u_1u_2^*}).
	\end{align*}
	Therefore, condition \ref{H3} holds.\par
	For \ref{H4}, consider the following equations:
	\begin{align*}
		(\downtriangle_{\hbar^{-1}}^{\frac{S+L}{\hbar}}u_1)u_2^*-u_1(\downtriangle_{\hbar^{-1}}^{\frac{S+L}{\hbar}}u_2)^*&=(\partial_{\hbar^{-1}}u_1+(S+L)u_1)u_2^*-u_1(\partial_{\hbar^{-1}}u_2+(S+L)u_2)^*\\
		&=(\partial_{\hbar^{-1}}u_1)u_2^*-u_1(-\partial_{\hbar^{-1}}u_2^*)\\
		&=\partial_{\hbar^{-1}}(u_1u_2^*).
	\end{align*}
	By the same argument as before, condition \ref{H4} also holds.\par
	The condition \ref{H5} follows from Proposition \ref{Prop_comm}.
\end{proof}

The reason why we call Definition \ref{mhrp} the \textit{modified} higher residue pairing is its slight difference from  the original higher residue pairings in Theorem 5.1, \cite{ST}; the original one is defined on the quantization cohomology $\mathcal{H}_{S+L}$ but our pairing is defined on  $T_{\mathbb{H}}(\!(\hbar)\!)$. Note that $T_{\mathbb{H}}(\!(\hbar)\!)$ and $\mathcal{H}_{S+L}$ are isomorphic as $\bC[\![\ud t]\!](\!(\hbar)\!)$-modules but it seems very hard to find an explicit isomorphism between them as  $\bC[\![\ud t]\!](\!(\hbar)\!)$-modules with the Gauss-Manin connections (i.e., as differential modules). 
This is the essential reason why we can not find an algorithm of the original higher residue pairing on $\mathcal{H}_{S+L}$, which we explain more details below.
\begin{definition}\label{mhrp2}
	Define the map $K_{S+L}:\mathcal{H}_{S+L}^{(0)}\times\mathcal{H}_{S+L}^{(0)}\to\mathbb{C}[\![\underline{t}]\!][\![\hbar]\!]$ as follows:
	\[
		K_{S+L}(\hbar\downtriangle_\alpha^{\frac{S+L}{\hbar}}\boldsymbol\zeta,\hbar\downtriangle_\beta^{\frac{S+L}{\hbar}}\boldsymbol\zeta):={}^0A_{\alpha\beta}{}^\mathrm{max},
	\]
	where $\boldsymbol\zeta$ is a weak primitive form satisfying
	\begin{equation}\label{WP_Sec5}
		\hbar\downtriangle_\beta^{\frac{S+L}{\hbar}}(\hbar\downtriangle_\alpha^{\frac{S+L}{\hbar}}\boldsymbol\zeta)=\sum_{\rho\in I}({}^0A_{\alpha\beta}{}^\rho+{}^1A_{\alpha\beta}{}^\rho\hbar)\cdot\downtriangle_\rho^{\frac{S+L}{\hbar}}\boldsymbol\zeta+(Q_{S+L}+\hbar\Delta)({}^0\Lambda_{\alpha\beta}+{}^1\Lambda_{\alpha\beta}).
	\end{equation}
	Then, extend $K_{S+L}$ to $\mathcal{H}_{S+L}^{(0)}$ using the fact that $\{\hbar\downtriangle_\rho^{\frac{S+L}{\hbar}}\boldsymbol\zeta:\rho\in I\}$ is a $\mathbb{C}[\![\underline{t}]\!][\![\hbar]\!]$-basis of $\mathcal{H}_{S+L}^{(0)}$ as follows:
	\[
		K_{S+L}(\mathbf{\Phi},\mathbf{\Psi}):=\sum_{\rho,\phi\in I}\mathbf{A}_{\mathbf{\Phi}}{}^\rho\cdot(\mathbf{A}_{\mathbf{\Psi}}{}^\phi)^*\cdot{}^0A_{\rho\phi}{}^\mathrm{max},
	\]
	where
	\begin{align*}
		\mathbf{\Phi}&=\sum_{\rho\in I}\mathbf{A}_{\mathbf{\Phi}}{}^\rho\hbar\downtriangle_\rho^{\frac{S+L}{\hbar}}\boldsymbol\zeta+(Q_{S+L}+\hbar\Delta)(\mathbf{\Lambda_{\mathbf{\Phi}}}),\\
		\mathbf{\Psi}&=\sum_{\rho\in I}\mathbf{A}_{\mathbf{\Psi}}{}^\rho\hbar\downtriangle_\rho^{\frac{S+L}{\hbar}}\boldsymbol\zeta+(Q_{S+L}+\hbar\Delta)(\mathbf{\Lambda_{\mathbf{\Psi}}}).
	\end{align*}
\end{definition}

\begin{definition}
	Define the map $\widetilde{\eta}$ as
	\begin{align*}
		\widetilde{\eta}:\frac{\mathcal{A}_0^0 [\![\underline{t}]\!]}{Q_{S+L}(\mathcal{A}_0^{-1}[\![\underline{t}]\!])}\times\frac{\mathcal{A}_0^0 [\![\underline{t}]\!]}{Q_{S+L}(\mathcal{A}_0^{-1}[\![\underline{t}]\!])}&\longrightarrow\mathbb{C}[\![\underline{t}]\!],\\
		(w_1,w_2)&\longmapsto A_{w_1w_2}{}^\mathrm{max},\quad w_i\in\mathcal{A}_0^0 [\![\underline{t}]\!],
	\end{align*}
	where $A_{w_1w_2}{}^\mathrm{max}$ is uniquely determined by
	\[
		w_1w_2=\sum_{\rho\in I}A_{w_1w_2}{}^\rho u_\rho\cdot{}^0\zeta+Q_{S+L}(\Lambda_{w_1w_2})
	\]
	and $A_{w_1w_2}{}^\rho\in\mathbb{C}[\![\underline{t}]\!]$, $\Lambda_{w_1w_2}\in\mathcal{A}_0^{-1}[\![\underline{t}]\!]$.
\end{definition}

Then, Proposition \ref{md_prop} follows from Definition \ref{mhrp2} immediately:
\begin{proposition}\label{md_prop}
	The map $K_{S+L}$ satisfies the following conditions for all $w_1,w_2\in\mathcal{H}_{S+L}^{(0)}$ and $g(\hbar)\in\mathbb{C}[\![\hbar]\!]$:
	\begin{enumerate}[\normalfont(h1)]
		\item $K_{S+L}(w_1,w_2)=K_{S+L}(w_2,w_1)^*$.
		\item\addtocounter{enumi}{2} $g(\hbar)K_{S+L}(w_1,w_2)=K_{S+L}(g(\hbar)w_1,w_2)=K(w_1,g^*(\hbar)w_2)$.
		\item The following diagram commutes:
		\[
			\begin{tikzcd}
				\mathcal{H}_{S+L}^{(0)}\times\mathcal{H}_{S+L}^{(0)}\arrow[r,"K_{S+L}"]\arrow[d,"\hbar=0"'] & \mathbb{C}[\![\underline{t}]\!][\![\hbar]\!]\arrow[d,"\hbar=0"'] \\
				\frac{\mathcal{A}_0^0 [\![\underline{t}]\!]}{Q_{S+L}(\mathcal{A}_0^{-1}[\![\underline{t}]\!])}\times \frac{\mathcal{A}_0^0 [\![\underline{t}]\!]}{Q_{S+L}(\mathcal{A}_0^{-1}[\![\underline{t}]\!])}\arrow[r,"\widetilde{\eta}"] & \mathbb{C}[\![\underline{t}]\!],
			\end{tikzcd}
		\]
		where the vertical maps substitute $\hbar$ with 0.
	\end{enumerate}
\end{proposition}

However, the map $K_{S+L}$ does not necessarily satisfy the following two conditions for all $w_1,w_2\in\mathcal{H}_{S+L}^{(0)}$ and $g(\hbar)\in\mathbb{C}[\![\hbar]\!]$ in general:
\begin{enumerate}[\normalfont(h1),start=3]
	\item\label{h3} $\partial_\alpha \left(K_{S+L}(w_1,w_2)\right)=K_{S+L}(\downtriangle_{\alpha}^{\frac{S+L}{\hbar}}w_1,w_2)+K_{S+L}(w_1,\downtriangle_\alpha^{\frac{S+L}{\hbar}}w_2)$.
	\item\label{h4} $\partial_{\hbar^{-1}}\left(K_{S+L}(w_1,w_2)\right)=K_{S+L}(\downtriangle_{\hbar^{-1}}^{\frac{S+L}{\hbar}}w_1,w_2)-K_{S+L}(w_1,\downtriangle_{\hbar^{-1}}^{\frac{S+L}{\hbar}}w_2)$.
\end{enumerate}
From now on we use the simplified notations $\downtriangle_\a=\downtriangle_\a^{\frac{S+L}{\hbar}}$ and $\downtriangle_{\hbar^{-1}}=\downtriangle_{\hbar^{-1}}^{\frac{S+L}{\hbar}}$.
We give some criterions when the weak primitive form $\boldsymbol\zeta$ guarantees \ref{h3} and \ref{h4}:
\begin{proposition}\label{h3_prop}
	The map $K_{S+L}$ satisfies the conditions \ref{h3} if the following equation holds:
	\begin{equation}\label{h3_cond}
		\partial_\gamma{}^0A_{\alpha\beta}{}^\mathrm{max}=\sum_{\rho\in I}\Big({}^1A_{\gamma\alpha}{}^\rho\cdot{}^0A_{\rho\beta}{}^\mathrm{max}+{}^1A_{\gamma\beta}{}^\rho\cdot{}^0A_{\alpha\rho}{}^\mathrm{max}\Big).
	\end{equation}
	where ${}^0A_{\alpha\beta}{}^\rho$ and ${}^1A_{\alpha\beta}{}^\rho$ appeared in the equation \eqref{WP_Sec5}.
\end{proposition}

\begin{proof}
	Since the map $K_{S+L}$ is extended by the $\mathbb{C}[\![\underline{t}]\!][\![\hbar]\!]$-basis $\{\hbar\downtriangle_\rho^{\frac{S+L}{\hbar}}\boldsymbol\zeta:\rho\in I\}$, the condition \ref{h3} is equivalent to
	\[
		\hbar\partial_\gamma K_{S+L}(\hbar\downtriangle_\alpha\boldsymbol\zeta,\hbar\downtriangle_\beta\boldsymbol\zeta)=K_{S+L}(\hbar\downtriangle_\gamma\hbar\downtriangle_\alpha\boldsymbol\zeta,\hbar\downtriangle_\beta\boldsymbol\zeta)-K_{S+L}(\hbar\downtriangle_\alpha\boldsymbol\zeta,\hbar\downtriangle_\gamma\hbar\downtriangle_\beta\boldsymbol\zeta).
	\]
	Suppose that the condition \eqref{h3_cond} holds. Then, consider the following equations:
	\begin{align*}
		&K_{S+L}(\hbar\downtriangle_\gamma\hbar\downtriangle_\alpha\boldsymbol\zeta,\hbar\downtriangle_\beta\boldsymbol\zeta)-K_{S+L}(\hbar\downtriangle_\alpha\boldsymbol\zeta,\hbar\downtriangle_\gamma\hbar\downtriangle_\beta\boldsymbol\zeta)\\
		&=\sum_{\rho\in I}\mathbf{A}_{\gamma\alpha}{}^\rho\cdot{}^0A_{\rho\beta}{}^\mathrm{max}-\sum_{\rho\in I}(\mathbf{A}_{\gamma\beta}{}^\rho)^*\cdot{}^0A_{\alpha\rho}{}^\mathrm{max}\\
		&=\Big[\sum_{\rho\in I}{}^0A_{\gamma\alpha}{}^\rho\cdot{}^0A_{\rho\beta}{}^\mathrm{max}\!-\!{}^0A_{\gamma\beta}{}^\rho\cdot^0A_{\alpha\rho}{}^\mathrm{max}\Big]\!+\!\hbar\Big[\sum_{\rho\in I}\big({}^1A_{\gamma\alpha}{}^\rho\cdot{}^0A_{\rho\beta}{}^\mathrm{max}+{}^1A_{\gamma\beta}{}^\rho\cdot{}^0A_{\alpha\rho}{}^\mathrm{max}\big)\Big]\\
		&=0+\hbar\big[\partial_\gamma{}^0A_{\alpha\beta}{}^\mathrm{max}\big]\\
		&=\hbar\partial_\gamma K_{S+L}(\hbar\downtriangle_\alpha\boldsymbol\zeta,\hbar\downtriangle_\beta\boldsymbol\zeta).
	\end{align*}
	Therefore, the condition \ref{h3} holds.
\end{proof}

\begin{remark}\label{metric_comp}
	Suppose that ${}^1A_{\alpha\beta}{}^\rho$ is the Christoffel symbol of the Levi-Civita connection $\nabla$ as in Remark \ref{rmk_1}. Then, the compatibility with the metric
	\[
		\partial_\gamma\langle\partial_\alpha,\partial_\beta\rangle=\langle\nabla_{\partial_\gamma}\partial_\alpha,\partial_\beta\rangle+\langle\partial_\alpha,\nabla_{\partial_\gamma}\partial_\beta\rangle
	\]
	is equivalent to the condition \eqref{h3_cond}.
\end{remark}

For the condition \ref{h4}, we define the coefficient $\mathbf{B}_{\alpha}{}^\rho\in\mathbb{C}[\![\underline{t}]\!][\![\hbar]\!]$ by
\begin{equation}\label{Notation_B}
	\hbar\downtriangle_{\hbar^{-1}}\downtriangle_\alpha\boldsymbol\zeta=\sum_{\rho\in I}\mathbf{B}_\alpha{}^\rho\hbar\downtriangle_\rho\boldsymbol\zeta+(Q_{S+L}+\hbar\Delta)(\mathbf{B}_\alpha).
\end{equation}
Note that $\downtriangle_{\hbar^{-1}}\circ\downtriangle_{\alpha}=\downtriangle_\alpha\circ\downtriangle_{\hbar^{-1}}$, and
\[
	\partial_{\hbar^{-1}}K_{S+L}(\hbar\downtriangle_\alpha\boldsymbol\zeta,\hbar\downtriangle_\beta\boldsymbol\zeta)=\partial_{\hbar^{-1}}{}^0A_{\alpha\beta}{}^\mathrm{max}=0.
\]
We know that
\begin{align*}
	K_{S+L}(\downtriangle_{\hbar^{-1}}\hbar\downtriangle_\alpha\boldsymbol\zeta,\hbar\downtriangle_\beta\boldsymbol\zeta)&+K_{S+L}(\hbar\downtriangle_\alpha\boldsymbol\zeta,\downtriangle_{\hbar^{-1}}\hbar\downtriangle_\beta\boldsymbol\zeta)\\
	&=K_{S+L}(\hbar\downtriangle_{\hbar^{-1}}\downtriangle_\alpha\boldsymbol\zeta,\hbar\downtriangle_\beta\boldsymbol\zeta)+K_{S+L}(\hbar\downtriangle_\alpha\boldsymbol\zeta,\hbar\downtriangle_{\hbar^{-1}}\downtriangle_\beta\boldsymbol\zeta),
\end{align*}
because
\begin{align*}
	K_{S+L}&(\downtriangle_{\hbar^{-1}}\hbar\downtriangle_\alpha\boldsymbol\zeta,\hbar\downtriangle_\beta\boldsymbol\zeta)=K_{S+L}(-\hbar^2\downtriangle_\alpha\boldsymbol\zeta+\hbar\downtriangle_{\hbar^{-1}}\downtriangle_\alpha\boldsymbol\zeta,\hbar\downtriangle_\beta\boldsymbol\zeta)\\
	&=-\hbar K_{S+L}(\hbar\downtriangle_\alpha\boldsymbol\zeta,\hbar\downtriangle_\beta\boldsymbol\zeta)+K_{S+L}(\hbar\downtriangle_{\hbar^{-1}}\downtriangle_\alpha\boldsymbol\zeta,\hbar\downtriangle_\beta\boldsymbol\zeta)
\end{align*}
and
\begin{align*}
	K_{S+L}&(\hbar\downtriangle_\alpha\boldsymbol\zeta,\downtriangle_{\hbar^{-1}}\hbar\downtriangle_\beta\boldsymbol\zeta)=K_{S+L}(\hbar\downtriangle_\alpha\boldsymbol\zeta,-\hbar^2\downtriangle_\beta\boldsymbol\zeta+\hbar\downtriangle_{\hbar^{-1}}\downtriangle_\beta\boldsymbol\zeta)\\
	&=\hbar K_{S+L}(\hbar\downtriangle_\alpha\boldsymbol\zeta,\hbar\downtriangle_\beta\boldsymbol\zeta)+K_{S+L}(\hbar\downtriangle_\alpha\boldsymbol\zeta,\hbar\downtriangle_{\hbar^{-1}}\downtriangle_\beta\boldsymbol\zeta).
\end{align*}
By \eqref{Notation_B},
\begin{equation}\label{HighRes_hbar}
	\begin{aligned}
		K_{S+L}(\hbar\downtriangle_{\hbar^{-1}}\downtriangle_\alpha\boldsymbol\zeta,\hbar\downtriangle_\beta\boldsymbol\zeta)+&K_{S+L}(\hbar\downtriangle_\alpha\boldsymbol\zeta,\hbar\downtriangle_{\hbar^{-1}}\downtriangle_\beta\boldsymbol\zeta)\\
		&=\sum_{\rho\in I}\Big[\mathbf{B}_\alpha{}^\rho\cdot{}^0A_{\rho\beta}{}^\mathrm{max}+(\mathbf{B}_\beta{}^\rho)^*\cdot{}^0A_{\alpha\rho}{}^\mathrm{max}\Big].
	\end{aligned}
\end{equation}
Consider the following equation:
\begin{align*}
	\hbar\downtriangle_\beta\hbar\downtriangle_{\hbar^{-1}}\downtriangle_\alpha\boldsymbol\zeta=&\hbar\downtriangle_\beta\Big(\sum_{\rho\in I}\mathbf{B}_\alpha{}^\rho\hbar\downtriangle_\rho\boldsymbol\zeta\Big)+(Q_{S+L}+\hbar\Delta)(\hbar\downtriangle_\beta\mathbf{B}_\alpha)\\
	=&\sum_{\rho\in I}\Big[\hbar\partial_\beta\mathbf{B}_\alpha{}^\rho\hbar\downtriangle_\rho\boldsymbol\zeta+\mathbf{B}_\alpha{}^\rho\hbar\downtriangle_\beta\hbar\downtriangle_\rho\boldsymbol\zeta\Big]+(Q_{S+L}+\hbar\Delta)(\hbar\downtriangle_\beta\mathbf{B}_\alpha)\\
	=&\sum_{\delta\in I}\Big[\hbar\partial_\beta\mathbf{B}_\alpha{}^\delta+\sum_{\rho\in I}\mathbf{B}_\alpha{}^\rho\cdot\mathbf{A}_{\beta\rho}{}^\delta\Big]\hbar\downtriangle_\delta\boldsymbol\zeta\\
	&+(Q_{S+L}+\hbar\Delta)\Big(\hbar\downtriangle_\beta\mathbf{B}_\alpha+\sum_{\rho\in I}\mathbf{B}_\alpha{}^\rho\mathbf{\Lambda}_{\beta\rho}\Big).
\end{align*}
Therefore, the following quantity is invariant under the permutation of $\alpha, \beta$ for all $\delta$:
\begin{equation}\label{hbar_a_b}
	\hbar\partial_\beta\mathbf{B}_\alpha{}^\delta+\sum_{\rho\in I}\mathbf{B}_\alpha{}^\rho\cdot\big({}^0A_{\beta\rho}{}^\delta+\hbar\cdot{}^1A_{\beta\rho}{}^\delta\big).
\end{equation}

\begin{proposition}\label{h4_prop}
	If $\mathbf{B}_\alpha{}^\rho=\hbar B_\alpha{}^\rho$ with $B_\alpha{}^\rho\in\mathbb{C}[\![\underline{t}]\!]$ and for all $\alpha,\rho\in I$, then the condition \ref{h4} holds.
\end{proposition}

\begin{proof}
	Suppose that $\mathbf{B}_\alpha{}^\rho=\hbar B_\alpha{}^\rho$ with $B_\alpha{}^\rho\in\mathbb{C}[\![\underline{t}]\!]$ for all $\alpha,\rho\in I$. Then \eqref{hbar_a_b} implies that
	\[
		\sum_{\rho\in I}B_\alpha{}^\rho\cdot{}^0A_{\beta\rho}{}^\delta
	\]
	has the same value up to the permutation of $\alpha,\beta$ for all $\delta$. Therefore the RHS of the equation \eqref{HighRes_hbar} becomes zero, since
	\[
		\sum_{\rho\in I}\mathbf{B}_\alpha{}^\rho\cdot{}^0A_{\rho\beta}{}^\mathrm{max}+(\mathbf{B}_\beta{}^\rho)^*\cdot{}^0A_{\alpha\rho}{}^\mathrm{max}=\sum_{\rho\in I}\hbar B_\alpha{}^\rho\cdot{}^0A_{\rho\beta}{}^\mathrm{max}-\hbar B_\beta{}^\rho\cdot{}^0A_{\alpha\rho}{}^\mathrm{max}=0.
	\]
	It implies that \ref{h4} holds.
\end{proof}

\begin{remark}
	If one can find the weak primitive form $\boldsymbol\zeta$ satisfying the sufficient conditions in Propositions \ref{h3_prop} and \ref{h4_prop}, then the map $K_{S+L}$ in Definition \ref{mhrp2} and $\boldsymbol\zeta$ would lead to the construction of the original higher residue pairing and the primitive form in the sense of K. Saito. However, the algorithm in subsection \ref{Subsec_Alg_prim} cannot provide the weak primitive form satisfying those sufficient conditions. 
\end{remark}

\section{Appendix}


\subsection{Comparisons}\label{subcomparison}

\begin{table}[ht]
	\caption{Comparisons}
	\begin{tabular}{|c||c|c|c|c|c|c|}
		\hline
		&&&&&\\[-2ex]
		& \cite{BK}  & \cite{ST} & Section \ref{Sec_F} & Section \ref{Sec_WP} & Section \ref{Sec_Frob} \\[1ex]
		\hline
		\hline
		&&&&&\\[-2ex]	
		$Q$ & $\overline{\partial}$ & $\begin{array}{c}Q_f\\f~\textrm{\small has an isolated}\\\textrm{\small singularity}\end{array}$ & $Q_{S(\underline{q})}$ & $Q_{S(\underline{q})}$ & $Q_{S(\underline{q})}$ \\[4ex]
		\hline
		&&&&&\\[-2ex]
		$\Delta$ & $\begin{array}{c}\textrm{\small the twist of }\partial\\\textrm{\small using the}\\ \textrm{\small volume form }\Omega\end{array}$ & $\sum\frac{\partial}{\partial x_i}\frac{\partial}{\partial\eta_i}$& $\sum\frac{\partial}{\partial q_i}\frac{\partial}{\partial\eta_i}$ & $\sum\frac{\partial}{\partial q_i}\frac{\partial}{\partial\eta_i}$ & $\sum\frac{\partial}{\partial q_i}\frac{\partial}{\partial\eta_i}$\\[2ex]
		\hline
		&&&&&\\[-2ex]
		$\Gamma$& $\begin{array}{c}\displaystyle\Gamma\!=\!\\\displaystyle\sum_{n\geq1}\frac{1}{n!}\sum_{\underline{\alpha}}t^{\underline{\alpha}}u_{\underline{\alpha}}\\\textrm{(general)}\end{array}$ & $\begin{array}{c}\displaystyle L=\sum_\alpha t^\alpha u_\alpha \\\textrm{(linear)}\end{array}$ & $\Gamma$ & $L$ & $L$\\[4ex]
		\hline
		&&&&&\\[-2ex]
		$\mathbf{A}_{\alpha\beta}{}^\rho$& ${}^0A_{\alpha\beta}{}^\rho$& $\begin{array}{c}{}^0A_{\alpha\beta}{}^\rho\\+{}^1A_{\alpha\beta}{}^\rho\hbar\end{array}$&${}^0A_{\alpha\beta}{}^\rho$ &$\begin{array}{c}{}^0A_{\alpha\beta}{}^\rho\\+{}^1A_{\alpha\beta}{}^\rho\hbar\end{array}$ & $\begin{array}{l}\scriptstyle{}^0a_{\alpha\beta}{}^\rho\\\scriptstyle+{}^1A_{\alpha\beta}{}^\rho\hbar\\\scriptstyle+{}^2A_{\alpha\beta}{}^\rho\hbar^2+\cdots\end{array}$ \\[3ex]
		\hline
		&&&&&\\[-2ex]
		$\boldsymbol\zeta$ & 1 & $\begin{array}{c}\scriptstyle\boldsymbol\zeta={}^0\zeta+{}^1\zeta\hbar+{}^2\zeta\hbar^2+\cdots\\\textrm{(general)}\end{array}$ & 1 & ${}^0\zeta+{}^1\zeta\hbar$ & 1\\[2ex]
		\hline
	\end{tabular}
	\label{Comparison_Table}
\end{table}

We compare our construction with the Barannikov-Kontsevich's dGBV algebra (\cite{BK}) for a compact Calabi-Yau manifold $X=X_{\ud G}(\bC)$. In \cite{BK}, the holomorphic polyvector field valued in anti-holomorphic differential forms is twisted with a nowhere vanishing holomorphic volume form $\Omega$, so that the twisted algebra structure gives rise to the dGBV algebra; $\partial$ becomes the BV operator $\Delta$, and $\bar{\partial}$ is the differential. The exterior derivative $d$ becomes $\Delta+\bar{\partial}$. We don't use polyvector fields on the compact Calabi-Yau manifold (to construct dGBV algebras), which is a key tool in \cite{BK}; we rather use the polyvector fields on the affine space $\bA^{n+k+1}$ (by looking at the complement of $X_{\ud G}$ in $\BP^n$ and using the Cayley trick) instead of the compact manifold $X$. The difference between our dGBV algebra and Barannikov-Kontsevich's dGBV algebra lies in the different decomposition of the total differential $d=\Delta+\bar{\partial}$ and $K=\Delta+Q_S$, and consequently, we can provide a relevant theory on the primitive middle dimensional cohomology $\bH=H_\mathrm{prim}^{n-k}(X,\mathbb{C})$ of $X$ instead of the whole cohomology group $H^\bullet(X,\mathbb{C})$.
A key difference coming from the different decomposition of $d$ results in the failure of the $Q_S\Delta$-lemma (but there is its weak version; see Lemma \ref{Q-D_Lemma}) in our theory, which prohibits us from applying Barannikov-Kontsevich's theory (where the $\partial\overline\partial$-lemma plays a key role) directly to our example $T_\mathbb{H}$. In the hypersurface case, an explicit cochain map between the two constructions was given in \cite{KKP}. \par

Now we do a comparison with Saito-Takahashi \cite{ST}. The fact that the differential $Q_S$ is a derivation with respect to the multiplication of $\cA[\![\underline{t}]\!]$ (so that $H^0\big(\mathcal{A}[\![\underline{t}]\!](\!(\hbar)\!),Q_S\big)$ is again an algebra with the induced multiplication from $\cA[\![\underline{t}]\!](\!(\hbar)\!)$), enables us to develop an explicit algorithm based on the Gr\"obner basis of the Jacobian ideal of the polynomial $S$. This feature is not available in \cite{BK}: the differential $\overline{\partial}$ (the Dolbeault differential) can not be studied by the Jacobian ideal of some polynomial. The polynomial $f$ in \cite{ST} has an isolated singularity and our algorithm of weak primitive forms can be applied to \cite{ST} with the differential $Q_f$. On the other hand, our polynomial $S$ has a non-isolated singularity which needs a new idea to apply the Gr\"obner basis algorithm; assuming that $X_{\ud G}$ is Calabi-Yau, we will restrict  $\mathcal{A}$ to its charge zero part $\mathcal{A}_0$ so that the cohomology algebra $H^0\big(\mathcal{A}_0,Q_S\big)$ is finite dimensional\footnote{Note that $H^0\big(\mathcal{A},Q_S\big)$ is infinite dimensional over $\bC$.} over $\bC$ and the relevant algorithm works.  
Another difference is that we have a filtration (which we call the weight filtration) on the dGBV algebra which is compatible with the Hodge filtration on $\bH$; this enables us to define the \textit{modified higher residue pairing} rather easily (see Theorem \ref{hrp}). Note that defining the higher residue pairing in the same way (using the residue symbol) as \cite{Saitoh} is difficult due to the fact that $S$ has a non-isolated singularity. \par

\subsection{The Riemann-Hilbert-Birkhoff problem}\label{sec6.1}

The Riemann-Hilbert-Birkhoff problem is important to prove the existence of primitive forms associated to the higher residue pairing in \cite{ST}. In this appendix, we provide Lemma \ref{RHB} which is a reinterpretation of the RHB problem in the case $\underline{t}=0$  (Lemma 6.3 in \cite{ST}).\footnote{The general RHB problem with non-zero $\underline{t}$-power coefficients (Lemma 6.6 in \cite{ST}) will not be considered here, since the main interest of this article is weak primitive forms rather than primitive forms.}
Lemma $\ref{RHB}$ will be a simple consequence of the weight (the Hodge filtration) homogeneity of the solution. Note that we can choose a $\mathbb{C}$-basis $\{u_\alpha:\alpha\in I\}=\{u_0,\dots,u_{\mu-1}\}$ of $\mathcal{A}_0^0/Q_S(\mathcal{A}_0^{-1})$ such that $\wt(u_\alpha)$ corresponds to the increasing Hodge filtration under the isomorphism $\mathcal{A}_0^0/Q_S(\mathcal{A}_0^{-1})\cong H_{\mathrm{prim}}^{n-k}(X_G)$ with $u_0=1$.
Recall the following $\mathbb{C}[\![{\hbar}]\!]$-module:
\[
	\mathbb{H}_{S}^{(0)}:=\frac{\mathcal{A}_0^0 [\![\hbar]\!]}{(Q_S+\hbar\Delta)(\mathcal{A}_0^{-1}[\![\hbar]\!])}.
\]
Define the set $\{\mathbf{v}_0,\dots,\mathbf{v}_{\mu-1}\}$ as follows:
\[
	\mathbf{v}_i=u_i+\sum_{\ell=1}^\infty{}^\ell v_i\cdot\hbar^\ell,\quad i=0,\dots,\mu-1,
\]
where ${}^\ell v_i$ is an arbitrary polynomial in $\mathcal{A}_0^0$ satisfying the weight condition $\wt({}^\ell v_i)+\ell=\wt(u_i)$ for all $\ell\geq 1$. Note that $\{\mathbf{v}_0,\dots,\mathbf{v}_{\mu-1}\}$ is a $\mathbb{C}[\![\underline{\hbar}]\!]$-basis of $\mathbb{H}_S^{(0)}$ by the same argument in the proof of Lemma \ref{Basis_lemma_F}. Then, we get the following Lemma:
\begin{lemma}\label{RHB}
	The above $\mathbb{C}[\![\hbar]\!]$-basis $\{\mathbf{v}_0,\dots,\mathbf{v}_{\mu-1}\}$ of $\mathbb{H}_S^{(0)}$ satisfies the following conditions:
	\begin{enumerate}[\normalfont(i)]
		\item $\mathbf{v}_0|_{\hbar=0}=1$.
		\item $\downtriangle^{\frac{S}{\hbar}}_{\hbar^{-1}}\vec{\mathbf{v}}\equiv C_S\vec{\mathbf{v}}-\hbar N_0\vec{\mathbf{v}}$ in $\mathbb{H}_S^{(0)}$ where ${}^t\vec{\mathbf{v}}=(\mathbf{v}_0,\dots,\mathbf{v}_{\mu-1})$, $C_S=\mathbf{0}\in M(\mu,\mathbb{C})$, and
		\[
			N_0=\left(\begin{matrix}
				k+\wt(u_0)& & \cdots &0\\
				\vdots&k+\wt(u_1) & &\\
				& &\ddots &\vdots\\
				0& \cdots& &k+\wt(u_{\mu-1})
			\end{matrix}\right)\in M(\mu,\mathbb{Q}).
		\]
	\end{enumerate}
\end{lemma}

\begin{proof}
	The condition (i) is trivial by the definition of $\mathbf{v}_0$.\par
	For the condition (ii), define $\mathbf{\Lambda}_i\in\mathcal{A}_0^{-1}[\![\hbar]\!]$ as
	\[
		\mathbf{\Lambda}_i:=\sum_{\ell=0}^\infty{}^\ell\Lambda_i\cdot\hbar^\ell\quad\textrm{where}~{}^\ell\Lambda_i=\sum_{j=1}^k(y_j\cdot{}^\ell v_i)\eta_j.
	\]
	Then, we get the following consequences:
	\begin{align*}
		Q_S({}^\ell\Lambda_i)&=Q_S\big(\sum_{j=1}^k(y_j\cdot{}^\ell v_i)\eta_j\big)=\big(\sum_{j=1}^k y_j\partial_{y_j}S){}^\ell v_i=S\cdot{}^\ell v_i,\\
		\Delta({}^\ell\Lambda_i)&=\Delta\big(\sum_{j=1}^k(y_j\cdot{}^\ell v_i)\eta_j\big)=\sum_{j=1}^k(\partial_{y_j}y_j\cdot{}^\ell v_i+y_j\cdot\partial_{y_j}{}^\ell v_i)=(k+wt({}^\ell v_i)){}^\ell v_i\\
		&=(k-\ell+wt(u_i)){}^\ell v_i.
	\end{align*}
	This implies the following system of equations:
	\begin{align*}
		S\cdot{}^0v_i&=Q_S({}^0\Lambda_i),\\
		S\cdot{}^1v_i-{}^1v_i&=-(k+wt(u_i)){}^0v_i+Q_S({}^1\Lambda_i)+\Delta({}^0\Lambda_i),\\
		S\cdot{}^2v_i-2{}^2v_i&=-(k+wt(u_i)){}^1v_i+Q_S({}^2\Lambda_i)+\Delta({}^1\Lambda_i),\\
		S\cdot{}^3v_i-3{}^3v_i&=-(k+wt(u_i)){}^2v_i+Q_S({}^3\Lambda_i)+\Delta({}^2\Lambda_i),\\
		\vdots
	\end{align*}
	The above equations are equivalent to
	\begin{align*}
		-\hbar(k+wt(u_i))\mathbf{v}_i+(Q_S+\hbar\Delta)(\mathbf{\Lambda}_i)&=S\cdot\mathbf{v}_i-\hbar({}^1v_i\hbar+2\cdot{}^2v_i\hbar^2+3\cdot{}^3v_i\hbar^3+\cdots)\\
		&=S\mathbf{v}_i+\partial_{\hbar^{-1}}(\mathbf{v}_i)=\downtriangle_{\hbar^{-1}}^{\frac{S}{\hbar}}\mathbf{v}_i.
	\end{align*}
	Therefore,
	\[
		\downtriangle_{\hbar^{-1}}^{\frac{S}{\hbar}}\mathbf{v}_i\equiv-\hbar(k+wt(u_i))\mathbf{v}_i~\textrm{in}~\mathbb{H}_S^{(0)}~\textrm{for all}~i=0,\dots,\mu-1,
	\]
	i.e., $\downtriangle^{\frac{S}{\hbar}}_{\hbar^{-1}}\vec{\mathbf{v}}\equiv C_S\vec{\mathbf{v}}-\hbar N_0\vec{\mathbf{v}}$ in $\mathbb{H}_S^{(0)}$.
\end{proof}
\subsection{Comments on the algorithm} \label{sec6.2}
In the body of the paper, we concentrated on solving
\begin{equation}\label{Rmk_eqn}
	\hbar\downtriangle_\beta\hbar\downtriangle_\alpha\boldsymbol\zeta=\sum_{\rho}\mathbf{A}_{\alpha\beta}{}^\rho\hbar\downtriangle_\rho\boldsymbol\zeta+(Q_{S+L}+\hbar\Delta)(\mathbf{\Lambda}_{\alpha\beta}),
\end{equation}
where $\downtriangle_\alpha:=\downtriangle_{\alpha}^{\frac{S+\G}{\hbar}}$, in certain cases (see Definition \ref{Def_WP}) which is related to some special structures such as formal $F$-manifolds, Frobenius manifolds and higher residue pairings. In this appendix, we provide some hints how to deal with \eqref{Rmk_eqn} in the other cases regardless of those special structures, assuming that $\Gamma=L:=\sum_{\a} t^{\a} u_{\a} $.

We vary conditions on the vanishing of $\hbar^{\ell}$-coefficients of $\mathbf{A}_{\alpha\beta}{}^\rho$ and $\boldsymbol\zeta$, to check whether we can construct relevant algorithms. In the following, ``unsolvable" means that our algorithm does not work.
 Consider the following cases:
\begin{enumerate}[(i)]
	\item\label{item1_app} ($\boldsymbol\zeta={}^0\zeta,$ $\mathbf{A}_{\alpha\beta}{}^\rho={}^0A_{\alpha\beta}{}^\rho$ : unsolvable) The equation \eqref{Rmk_eqn} becomes
	\begin{align*}
		\hbar\downtriangle_\beta\hbar\downtriangle_\alpha{}^0\zeta=&\sum_{\rho}{}^0A_{\alpha\beta}{}^\rho\hbar\downtriangle_\rho{}^0\zeta+(Q_{S+L}+\hbar\Delta)(\mathbf{\Lambda}_{\alpha\beta}),\\
		u_\alpha u_\beta{}^0\zeta+\hbar(u_\beta{}^0\zeta_\alpha+u_\alpha{}^0\zeta_\beta)+\hbar^2({}^0\zeta_{\alpha\beta})=&\sum_{\rho}{}^0A_{\alpha\beta}{}^\rho u_\rho{}^0\zeta+\hbar({}^0A_{\alpha\beta}{}^\rho\cdot{}^0\zeta_\rho)\\
		&+(Q_{S+L}+\hbar\Delta)(\mathbf{\Lambda}_{\alpha\beta}).
	\end{align*}
	Therefore, by comparing the $\hbar$-power terms, we have to solve the following equations:
	\begin{align*}
		u_\alpha u_\beta{}^0\zeta&=\sum_{\rho}{}^0A_{\alpha\beta}{}^\rho u_\rho{}^0\zeta+Q_{S+L}({}^0\Lambda_{\alpha\beta}),\\
		u_\beta{}^0\zeta_\alpha+u_\alpha{}^0\zeta_\beta&=\sum_\rho{}^0A_{\alpha\beta}{}^\rho\cdot{}^0\zeta_\rho+\Delta({}^0\Lambda_{\alpha\beta})+Q_{S+L}({}^1\Lambda_{\alpha\beta}),\\
		{}^0\zeta_{\alpha\beta}&=\Delta({}^1\Lambda_{\alpha\beta}).
	\end{align*}
	However, it is difficult to solve the equation
	\begin{equation}\label{diff_eqn}
		u_\beta{}^0\zeta_\alpha+u_\alpha{}^0\zeta_\beta=\sum_\rho{}^0A_{\alpha\beta}{}^\rho\cdot{}^0\zeta_\rho+\Delta({}^0\Lambda_{\alpha\beta})+Q_{S+L}({}^1\Lambda_{\alpha\beta})
	\end{equation}
	by the similar algorithm in subsection \ref{Subsec_Alg_prim}, since it is hard to choose ${}^0\zeta_\alpha$ for all $\alpha\in I$ such that $u_\beta{}^0\zeta_\alpha+u_\alpha{}^0\zeta_\beta-\sum_\rho{}^0A_{\alpha\beta}{}^\rho\cdot{}^0\zeta_\rho-\Delta({}^0\Lambda_{\alpha\beta})\in\mathrm{Im}(Q_{S+L})$.
But if we allow the $\hbar$-coefficient $^1A_{\alpha\beta}{}^\rho$ to be non-zero, we can find the algorithm.
	\item\label{item2_app} ($\boldsymbol\zeta={}^0\zeta,$ $\mathbf{A}_{\alpha\beta}{}^\rho={}^0A_{\alpha\beta}{}^\rho+{}^1A_{\alpha\beta}{}^\rho\hbar$ : solvable) The right hand side of the equation \eqref{Rmk_eqn} becomes
	\[
		\sum_{\rho}{}^0A_{\alpha\beta}{}^\rho u_\rho{}^0\zeta+\hbar({}^0A_{\alpha\beta}{}^\rho\cdot{}^0\zeta_\rho+{}^1A_{\alpha\beta}{}^\rho u_\rho)+\hbar^2({}^1A_{\alpha\beta}{}^\rho\cdot{}^0\zeta_\rho)+(Q_{S+L}+\hbar\Delta)(\mathbf{\Lambda}_{\alpha\beta}).
	\]
	Therefore, we have to solve the following equations:
	\begin{align*}
		u_\alpha u_\beta{}^0\zeta&=\sum_{\rho}{}^0A_{\alpha\beta}{}^\rho u_\rho{}^0\zeta+Q_{S+L}({}^0\Lambda_{\alpha\beta}),\\
		u_\beta{}^0\zeta_\alpha+u_\alpha{}^0\zeta_\beta&=\sum_\rho {}^0A_{\alpha\beta}{}^\rho\cdot{}^0\zeta_\rho+{}^1A_{\alpha\beta}{}^\rho u_\rho+\Delta({}^0\Lambda_{\alpha\beta})+Q_{S+L}({}^1\Lambda_{\alpha\beta}),\\
		{}^0\zeta_{\alpha\beta}&=\sum_{\rho}{}^1A_{\alpha\beta}{}^\rho\cdot{}^0\zeta_\rho+\Delta({}^1\Lambda_{\alpha\beta}).
	\end{align*}
	We can solve the above equations because of the existence of the terms ${}^1A_{\alpha\beta}{}^\rho$ and ${}^0\zeta_{\alpha\beta}$. For example, unlike the equation \eqref{diff_eqn}, we can choose ${}^0\zeta_\alpha$ for all $\alpha\in I$ such that
	\[
		u_\beta{}^0\zeta_\alpha+u_\alpha{}^0\zeta_\beta-\sum_\rho {}^0A_{\alpha\beta}{}^\rho\cdot{}^0\zeta_\rho-\Delta({}^0\Lambda_{\alpha\beta})=\sum_{\rho}{}^1A_{\alpha\beta}{}^\rho u_\rho+Q_{S+L}({}^1\Lambda_{\alpha\beta}),
	\]
	which uniquely determines ${}^1A_{\alpha\beta}{}^\rho$. In addition, we can put ${}^0\zeta_{\alpha\beta}$ as
	\[
		{}^0\zeta_{\alpha\beta}=\sum_{\rho}{}^1A_{\alpha\beta}{}^\rho\cdot{}^0\zeta_\rho+\Delta({}^1\Lambda_{\alpha\beta}).
	\]
	
	\item\label{item3_app} ($\boldsymbol\zeta={}^0\zeta+{}^1\zeta\hbar,$ $\mathbf{A}_{\alpha\beta}{}^\rho={}^0A_{\alpha\beta}{}^\rho$ : unsolvable) The left hand side of the equation \eqref{Rmk_eqn} becomes
	\begin{align*}
		\hbar\downtriangle_\beta&\hbar\downtriangle_\alpha({}^0\zeta+{}^1\zeta\hbar)\\
		&=u_\alpha u_\beta{}^0\zeta+\hbar(u_\alpha u_\beta{}^1\zeta+u_\beta{}^0\zeta_\alpha+u_\alpha{}^0\zeta_\beta)+\hbar^2(u_\beta{}^1\zeta_\alpha+u_\alpha{}^1\zeta_\beta+{}^0\zeta_{\alpha\beta})+\hbar^3({}^1\zeta_{\alpha\beta}),
	\end{align*}
	and the right hand side of the equation \eqref{Rmk_eqn} becomes
	\[
		\sum_{\rho}{}^0A_{\alpha\beta}{}^\rho u_\rho{}^0\zeta+\hbar({}^0A_{\alpha\beta}{}^\rho(u_\rho{}^1\zeta+{}^0\zeta_\rho))+\hbar^2({}^0A_{\alpha\beta}{}^\rho\cdot{}^1\zeta_\rho)+(Q_{S+L}+\hbar\Delta)(\mathbf{\Lambda}_{\alpha\beta}).
	\]
	Then, we get the equation
	\[
		u_\alpha u_\beta{}^1\zeta+u_\beta{}^0\zeta_\alpha+u_\alpha{}^0\zeta_\beta=\sum_\rho{}^0A_{\alpha\beta}{}^\rho(u_\rho{}^1\zeta+{}^0\zeta_\rho)+Q_{S+L}({}^1\Lambda_{\alpha\beta})+\Delta({}^0\Lambda_{\alpha\beta}),
	\]
	which is difficult to solve by the similar algorithm in subsection \ref{Subsec_Alg_prim}, since it is hard to choose ${}^0\zeta_\alpha$ for all $\alpha\in I$ such that
	\[
		u_\alpha u_\beta{}^1\zeta+u_\beta{}^0\zeta_\alpha+u_\alpha{}^0\zeta_\beta-\sum_\rho{}^0A_{\alpha\beta}{}^\rho(u_\rho{}^1\zeta+{}^0\zeta_\rho)-\Delta({}^0\Lambda_{\alpha\beta})\in\mathrm{Im}(Q_{S+L}).
	\]
But again if we allow the $\hbar$-coefficient $^1A_{\alpha\beta}{}^\rho$ to be non-zero, we can find the algorithm.
	\item\label{item4_app} ($\boldsymbol\zeta={}^0\zeta+{}^1\zeta\hbar,$ $\mathbf{A}_{\alpha\beta}{}^\rho={}^0A_{\alpha\beta}{}^\rho+{}^1A_{\alpha\beta}{}^\rho\hbar$ : solvable) This case is the main algorithm done in subsection \ref{Subsec_Alg_prim}.

	\item ($\boldsymbol\zeta={}^0\zeta+{}^1\zeta\hbar+{}^2\zeta\hbar^2+\cdots,$ $\mathbf{A}_{\alpha\beta}{}^\rho={}^0A_{\alpha\beta}{}^\rho$ : unsolvable) By the same argument in \ref{item1_app} and \ref{item3_app}, the following equation from the equation \ref{Rmk_eqn}
	\[
		u_\alpha u_\beta{}^1\zeta+u_\beta{}^0\zeta_\alpha+u_\alpha{}^0\zeta_\beta=\sum_\rho{}^0A_{\alpha\beta}{}^\rho(u_\rho{}^1\zeta+{}^0\zeta_\rho)+Q_{S+L}({}^1\Lambda_{\alpha\beta})+\Delta({}^0\Lambda_{\alpha\beta}),
	\]
	is unsolvable, since it is difficult to choose ${}^0\zeta_\alpha$ for all $\alpha\in I$ satisfying the above equation. Once again if we allow the $\hbar$-coefficient $^1A_{\alpha\beta}{}^\rho$ to be non-zero, we can find the algorithm.
%
	
	\item ($\boldsymbol\zeta={}^0\zeta+{}^1\zeta\hbar+{}^2\zeta\hbar^2 + \cdots + {}^n\zeta\hbar^n,$ $\mathbf{A}_{\alpha\beta}{}^\rho={}^0A_{\alpha\beta}{}^\rho+{}^1A_{\alpha\beta}{}^\rho\hbar$ : solvable for any $n\geq 0$) The case for $n=0,1$ is done in \ref{item2_app} and \ref{item4_app}. For $n=2$ case, by the same argument in \ref{item2_app} and \ref{item4_app}, we have the following equations:
	\begin{align*}
		u_\alpha u_\beta{}^0\zeta=&\sum_{\rho}{}^0A_{\alpha\beta}{}^\rho u_\rho{}^0\zeta+Q_{S+L}({}^0\Lambda_{\alpha\beta}),\\
		u_\alpha u_\beta{}^1\zeta+u_\beta{}^0\zeta_\alpha+u_\alpha{}^0\zeta_\beta=&\sum_{\rho}{}^0A_{\alpha\beta}{}^\rho(u_\rho{}^1\zeta+{}^0\zeta_\rho)+{}^1A_{\alpha\beta}{}^\rho u_\rho{}^0\zeta\\
		&+\Delta({}^0\Lambda_{\alpha\beta})+Q_{S+L}({}^1\Lambda_{\alpha\beta}),\\
		u_\alpha u_\beta{}^2\zeta+u_\beta{}^1\zeta_\alpha+u_\alpha{}^1\zeta_\beta+{}^0\zeta_{\alpha\beta}=&\sum_{\rho}{}^0A_{\alpha\beta}{}^\rho(u_\rho{}^2\zeta+{}^1\zeta_\rho)+{}^1A_{\alpha\beta}{}^\rho(u_\rho{}^1\zeta+{}^0\zeta_\rho)\\
		&+\Delta({}^1\Lambda_{\alpha\beta}),\\
		u_\beta{}^2\zeta_\alpha+u_\alpha{}^2\zeta_\beta+{}^1\zeta_{\alpha\beta}=&\sum_{\rho}{}^0A_{\alpha\beta}{}^\rho\cdot{}^2\zeta_\rho+{}^1A_{\alpha\beta}{}^\rho(u_\rho{}^2\zeta+{}^1\zeta_\rho),\\
		{}^2\zeta_{\alpha\beta}=&\sum_{\rho}{}^1A_{\alpha\beta}{}^\rho\cdot{}^2\zeta_\rho.
	\end{align*}
	They are solvable because of the existence of the terms ${}^1A_{\alpha\beta}{}^\rho$, ${}^0\zeta_{\alpha\beta}$, ${}^1\zeta_{\alpha\beta}$ and ${}^2\zeta_{\alpha\beta}$. The general case $n\geq 3$ can be checked similarly.
\end{enumerate}

These examples provide a general shape of how to solve \eqref{Rmk_eqn} and we leave to the reader to think about the case dealing with the higher coefficients ${}^{\ell}A_{\alpha\beta}{}^\rho$ of $\hbar^{\ell}$ ($\ell\geq 2$) involved.


\end{document}